\documentclass[12pt]{article}
\usepackage{amsfonts}
\usepackage{amsthm}
\usepackage{CJK}
\usepackage{amssymb}
\usepackage{amsmath}    
\usepackage{graphicx,floatrow}   
\usepackage{verbatim}   
\usepackage{color}      
\usepackage{subfigure}  
\usepackage{epsfig}
\usepackage{subfigure}
\usepackage{float}
\usepackage{cases}
\usepackage{threeparttable}

\newtheorem{theorem}{Theorem}[section]
\newtheorem{lemma}[theorem]{Lemma}

\newtheorem{prop}[theorem]{Proposition}

\theoremstyle{definition}
\newtheorem{remark}{Remark}[section]
\newtheorem{example}[remark]{Example}
\numberwithin{equation}{section}

\newcommand{\norm}[1]{\left\Vert#1\right\Vert}
\newcommand{\abs}[1]{\left\vert#1\right\vert}

\newcommand{\average}[1]{\ensuremath{\{\!\!\{#1\}\!\!\}} }
\newcommand{\jump}[1]{\ensuremath{[\![#1]\!]} }

\begin{document}

\title{Discontinuous Galerkin methods for short pulse type equations via hodograph transformations}
\date{}
\author{Qian Zhang\thanks{School of Mathematical Sciences, University of Science and Technology of China, Hefei, Anhui 230026, P.R. China.  E-mail: gelee@mail.ustc.edu.cn.}
\and Yinhua Xia\thanks{Corresponding author. School of Mathematical Sciences, University of Science and Technology of China, Hefei, Anhui 230026, P.R. China.  E-mail: yhxia@ustc.edu.cn. Research supported by NSFC grant Nos. 11871449, 11471306, and a grant from the Science \& Technology on Reliability \& Environmental Engineering Laboratory (No. 6142A0502020817).}
}

\maketitle

\begin{abstract}
In the present paper, we consider the discontinuous Galerkin (DG) methods for solving short pulse (SP) type equations. The short pulse equation has been shown to be completely integrable, which admits the loop-soliton, cuspon-soliton solutions as well as smooth-soliton solutions. Through hodograph transformations, these nonclassical solutions can be profiled as the smooth solutions of the coupled dispersionless (CD) system or the sine-Gordon equation. Thus, DG methods can be developed for the CD system or  the sine-Gordon equation to simulate the loop-soliton or cuspon-soliton solutions of the SP equation. 
The conservativeness or dissipation of the Hamiltonian or momentum for the semi-discrete DG schemes can be proved. Also we modify the above DG schemes and obtain an integration DG scheme. Theoretically the a-priori error estimates have been provided for the momentum conserved DG scheme and the integration DG scheme.  We also propose the DG scheme and the integration DG scheme for the sine-Gordon equation, in case the SP equation can not be transformed to the CD system. All these DG schemes can be adopted to the generalized or modified SP type equations. Numerical experiments are provided to illustrate the optimal order of accuracy and capability of these DG schemes.
\end{abstract}


\section{Introduction}
 In this paper, we mainly study the classic short pulse (SP) equation derived by Sch\"{a}fer and Wayne in \cite{Schafer_2004_PNP}
\begin{equation}\label{eqn:short pulse intro}
  \ u_{xt} = u + \frac{1}{6}(u^3)_{xx}.
\end{equation}
The SP equation models the propagation of ultra-short light pulses in silica optical fibers. Here, $\ u\in\mathbb{R}$ is a real-valued function which represents the magnitude of the electric field. It is well-known that the cubic nonlinear Schr\"{o}dinger (NLS) equation derived from the Maxwell's equation can describe the propagation of pulse in optical fibers. Two preconditions of this derivation need to be satisfied: First, the response of the material attains a quasi-steady-state and second that the pulse width is as large as the oscillation of the carrier frequency. And now we can create very short pulses by the advanced technology and the pulse spectrum is not narrowly localized around the carrier frequency, that is, when the pulse is as short as a few cycles of the central frequency. Therefore, we use SP equation to approximate the ultra-short light pulse.  And numerical experiments made in \cite{Chung_2005_N} show as the pulse shortens, the accuracy of the SP equation approximated to Maxwell's equation increases, however, the NLS equation becomes inaccuracy for the ultra-short pulse. If the pulse is as short as only one cycle of its carrier frequency, then the modified short pulse equation in
\cite{Sergei_2016_NSNS} is used to describe the propagation of pulse in optical fibers.
Similar to the extension of coupled nonlinear Schr\"{o}dinger equations from NLS equations, it is necessary to consider its two-component or multi-component generalizations for describing the effect of polarization or anisotropy \cite{Dimakis_2010_IGMA, Matsuno_2011_JMP, Shen_2017_JNMP}. For birefringent fibers, the authors in \cite{Feng_2012_JPMT, Feng_2015_PDNP} also introduced some extensions of the SP equation to describe the propagation of ultra-short pulse. We will introduce these extensions specifically in Section \ref{Extension}.

Integrable discretizations of short pulse type equations have received considerable attention
recently, especially the loop-soliton, antiloop-soliton and cuspon-soliton solutions in \cite{Feng_2012_JPMT, Feng_2015_JPMT, Feng_2015_PDNP, Feng_2014_PJMI, Shen_2017_JNMP}. The authors linked  the short pulse type equations with the coupled dispersionless (CD) type systems or the sine-Gordon type equations through the hodograph transformations. The key of the discretization is an introduction of a
nonuniform mesh, which plays a role of the hodograph transformations as in continuous
case. In this paper, we aim at solving the loop-soliton, cupson-soliton solutions of the short pulse type equations as well as smooth-soliton solutions.  Through the hodograph transformation $(x, t)\rightarrow(y, s)$ which was proposed in \cite{Sakovich_2005_JPSJ},
\begin{equation}\label{eqn:hodograph_T}
\begin{cases}
&\frac{\partial}{\partial x} = \frac{1}{\rho} \frac{\partial}{\partial y}, \\
&\frac{\partial}{\partial t} = \frac{\partial}{\partial s} + \frac{u^2}{2\rho}\frac{\partial}{\partial y},
\end{cases}
\end{equation}
we can establish the link between the SP equation \eqref{eqn:short pulse intro}  and the CD system  \cite{Kakuhata_1996_JPSJ},
\begin{equation}\label{eqn:CD}
\rho_s + (\frac{1}{2}u^2)_y = 0, \ u_{ys} = \rho u.
\end{equation}
 There exists some short pulse type equations which are failed to be transformed into   CD systems. Therefore, we consider an alternative approach by introducing a new variable $z$ and define another hodograph transformation,
\begin{equation}\label{eqn:hodograph_KS}
\begin{cases}
\frac{\partial}{\partial x} = (\cos z)^{-1} \frac{\partial}{\partial y}, \\
\frac{\partial}{\partial t} = \frac{\partial}{\partial s} + \frac{1}{2}z_s^2(\cos z)^{-1}\frac{\partial}{\partial y},
\end{cases}
\end{equation}
which connects the short pulse  equation \eqref{eqn:short pulse intro} with the sine-Gordon  equation \cite{Sergei_2016_NSNS}
 \begin{equation}
 z_{ys} = \sin{z}.
 \end{equation}
For the CD system or the sine-Gordon equation, we develop the discontinuous Galerkin (DG) schemes to obtain the high-order accuracy numerical solution $u_h(y,s)$ or $z_h(y,s)$. Consequently, a point-to-point profile $u_h(x,t)$ of loop-soliton, cuspon-soliton solutions of the SP equation can be obtained, which are shown by the numerical experiments in Section \ref{Numeical experiment}. 

The DG method was first introduced in 1973 by Reed and Hill in \cite{Reed1973} for solving steady state linear hyperbolic equations. The important ingredient of this method is the design of suitable inter-element boundary treatments (so called numerical fluxes) to obtain highly accurate and stable schemes in many situations. Within the DG framework, the method was extented to deal with derivatives of order higher than one, i.e. local discontinuous Galerkin (LDG) method. The first LDG method was introduced by Cockburn and Shu in \cite{Shu1998_Siam} for solving convection-diffusion equation. Their work was motivated by the successful numerical experiments of Bassi and Rebay \cite{Bassi1997_JCP} for compressible Navier-Stokes equations. Later, Yan and Shu developed a LDG method for a general KdV type equation containing third order
derivatives in \cite{Yan2002_Siam}, and they generalized the LDG method to PDEs with fourth and fifth spatial derivatives in \cite{Yan2002_JSC}. Levy, Shu and Yan \cite{Levy2004_JCP} developed LDG methods for nonlinear dispersive equations that have compactly supported traveling wave solutions, the so-called compactons. More recently, Xu and Shu further generalized the LDG method to solve a series of nonlinear wave equations \cite{Xu2004_JCM,Xu2005_JCP,Xu2005_PDNP,Xu2006_CMAME}. We refer to the review paper \cite{Xu2010_CiCP} of LDG methods for high-order time-dependent partial differential equations.

 Most recently, a series of schemes which called structure-preserving schemes have attracted considerable attention. For some integrable equations like KdV type equations \cite{Bona2013_MC, Xing2016_CiCP, Liu2016_JCP, Zhang2019CiCP},   Zakharov system \cite{Xia2010JCP}, Schr\"odinger-KdV system \cite{Xia2014CiCP}, Camassa-Holm equation \cite{Yu2018_JCAM}, etc., the authors proposed various conservative numerical schemes to ``preserve structure". These conservative numerical schemes have some advantages over the dissipative ones, for example, the Hamiltonian conservativeness can help reduce the phase error along the long time evolution and have a more accurate approximation to exact solutions for KdV type equations \cite{Zhang2019CiCP}. The CD system and the generalized CD system are integrable, thus they have an infinite number of conserved quantities \cite{Kakuhata_1995_SCAN}. For CD system, the following two invariants
\begin{equation}
H_0 = \int \rho u^2 \ dy, \ H_1 = \int \rho^2 + u^2_y \ dy .
\end{equation}
 are corresponding to the Hamiltonian $E_0, E_1$ of the SP equation \cite{Brunelli2005_JMP, Brunelli2006_PLA} via the hodograph transformation,
\begin{equation}
\begin{split}
&E_0 = \int u^2 \ dx, \ E_1 = \int \sqrt{1+u^2_x} \ dx. \\
\end{split}
\end{equation}
In this paper, we first construct $E_0$ conserved DG scheme for the SP equation directly. And for the loop-soliton and cuspon-soliton solutions, the $H_0$, $H_1$ conserved DG schemes for CD system are developed respectively, to profile the singular solutions of the SP equation.  Also we modify the above DG schemes and propose an integration DG scheme which can numerically achieve the optimal convergence rates for $\rho, u$, and $u_y$.  
Theoretically, we prove that the $H_1$ conserved DG scheme has the optimal order of accuracy for $\rho, u$ and $u_y$ in $L^2$ norm. The integration DG scheme can be proved the optimal  order of accuracy for $ \rho, u_y$ in $L^2$ norm and the suboptimal order of accuracy for  $u$ in $L^\infty$ norm. All these DG schemes can be adopted to the generalized or modified SP type equations. 

The rest of this paper is organized as follows: In Section 2, we develop the DG schemes for the SP equation directly, and  via the hodograph transformations for the CD system and the sine-Gordon equation.   Some notations for simplifying expressions are given in Section \ref{notation}. In Section \ref{subsec:SP}, we first propose the $E_0$ conserved DG scheme for the SP equation. To simulate the loop-soliton or cuspon-soliton solutions of the SP equation,  the  $H_0$, $H_1$ conserved DG schemes and the integration DG scheme are constructed for the CD system which links the SP equation by the hodograph transformation. Meanwhile, the a priori error estimates for $H_1$ conserved DG and integration DG schemes are also provided. Moreover, we develop two kinds of DG schemes for the sine-Gordon equation to introduce another resolution for the SP equation in Section \ref{SG and SP}. Section \ref{Extension} is devoted to summarize the generalized short pulse equations and introduce the corresponding conserved quantities briefly. Several numerical experiments are listed in Section \ref{Numeical experiment}, including the propagation and interaction of loop-soliton, cuspon-solution, breather solution of the short pulse type equations. We also show the accuracy and the change of conserved quantities in Section \ref{Numeical experiment}. Finally, some concluding remarks are given in Section \ref{conclusion}.

\section{The discontinuous Galerkin discretization}\label{CD and SP}
In this section, we present the discontinuous Galerkin discretization for solving the short pulse type equations.  In order to describe the methods, we first introduce some notations.

\subsection{Notations}\label{notation}

We denote the mesh $\mathcal{T}_h $ on the spatial $y$ by $I_j = [y_{j-\frac{1}{2}}, y_{j+\frac{1}{2}}] $ for $j = 1,\ldots, N$, with the cell center denoted by $ y_j = \frac{1}{2}(y_{j-\frac{1}{2}}+y_{j+\frac{1}{2}})$. The cell size is $\Delta y_j = y_{j+\frac{1}{2}}- y_{j-\frac{1}{2}} $ and $ h = \max\limits_{1\leq j\leq N} \ \Delta y_j$.  The finite element space as the solution and test function space consists of piecewise polynomials
$$V_h^k = \{v:v|_{I_j} \in P^k(I_j); 1\leq j\leq N\},$$
 where $P^k(I_j)$ denotes the set of polynomial of degree up to $k$ defined on the cell $I_j$. Note that functions in $V_h^k$ are allowed to be discontinuous across cell interfaces. We also denote by the $u_{j+\frac{1}{2}}^-$ and $u_{j+\frac{1}{2}}^+$ the values of $u$ at $y_{j+\frac{1}{2}}$, from the left cell $I_j$ and the right cell $I_{j+1}$ respectively. And the jump of $u$ is defined as $\jump{u} = u^+ - u^- $, the average of $u$ as $\average{u} = \frac{1}{2}(u^+ + u^-)$.
 To simplify expressions, we adopt the round bracket and angle bracket for the $L^2$ inner product on cell $I_j$ and its boundary
\begin{align*}\label{bracket_def}
(u, v)_{I_j}& = \int_{I_j} uv dy,\\
<\hat{u}, v>_{I_j}& = \hat{u}_{j+\frac{1}{2}}v_{j+\frac{1}{2}}^- - \hat{u}_{j-\frac{1}{2}}v_{j-\frac{1}{2}}^+,
\end{align*}
for one dimensional case.

For the spatial variable $x$, we denote the mesh $\mathcal{T}'_h $ by $I'_j = [x_{j-\frac{1}{2}}, x_{j+\frac{1}{2}}] $ for $j = 1,\ldots, N$. Similar to the notations on the mesh $\mathcal{T}_h$, we have $x_j, \Delta x_j$, and $h' = \max\limits_{1\leq j\leq N} \ \Delta x_j$. Without misunderstanding, we still use $u_{j+\frac{1}{2}}^-$ and $u_{j+\frac{1}{2}}^+$  denote the values of $u$ at $x_{j+\frac{1}{2}}$, from the left cell $I'_j$ and the right cell $I'_{j+1}$ respectively.

\subsection{The short pulse equation}\label{subsec:SP}
Recall the short pulse equation
\begin{equation}\label{eqn:short pulse}
  \ u_{xt} = u + \frac{1}{6}(u^3)_{xx},\ x \in I' = [x_L,x_R],
\end{equation}
where  $u(x,t) \in \mathbb{R}$ is a real-valued function, $t$ denotes the time coordinate and $x$ is the spatial scale. Through the hodograph transformation, it can be converted into a coupled dispersionless (CD) system
\begin{subnumcases}{\label{eqn:CD system}}
\rho_s + (\frac{1}{2}u^2)_y = 0, \label{eqn:CD system1}\\
u_{ys} = \rho u,\label{eqn:CD system2}
\end{subnumcases}
where $s$ denotes the time coordinate, and $y$ is the spatial scale, $y \in I = [y_L,y_R]$. The hodograph  transformation $(y, s) \rightarrow (x, t)$ is defined by
\begin{align}\label{eqn:hodograph_CD_SPE}
\begin{cases}
\frac{\partial}{\partial y} = \rho \frac{\partial}{\partial x}, \\
\frac{\partial}{\partial s} = \frac{\partial}{\partial t} - \frac{u^2}{2}\frac{\partial}{\partial x}.
\end{cases}
\end{align}
And the parametric representation of the solution of the short pulse equation \eqref{eqn:short pulse} is
\begin{align}
u = u(y,s), \ x = x(y_0,s) + \int_{y_0}^y \rho(\zeta,s) d\zeta.
\end{align}
where $y_0$ is a real constant.
Since the short pulse equation and the equivalent CD system are completely integrable and hence they have an infinite
number of conservation laws. The first two invariants of the SP equation are described by
\begin{align}
E_0 = \int u^2  dx, \ E_1 = \int \sqrt{1+u^2_x}\ dx,
\end{align}
and the corresponding conservation laws for the CD system are
\begin{equation}
H_0 = \int \rho u^2 dy, \ H_1 = \int \rho^2 + u^2_y \ dy.
\end{equation}

\subsubsection{$E_0$ conserved DG scheme}\label{sec:E0_SP}
To construct the discontinuous Galerkin method for the SP equation directly, we rewrite the SP equation \eqref{eqn:short pulse} as a first order system:
\begin{subnumcases}{\label{first_order_SP}}
   v_t = u + \omega_x, \\
   v = u_x, \\
   \omega = (\frac{1}{6}u^3)_x.
\end{subnumcases}
Then the local DG scheme for equations \eqref{first_order_SP} is formulated as follows: Find $u_h,v_h,\omega_h$ such that, for all test functions $\varphi, \phi, \psi \in V^k_h$ and $I'_j \in \mathcal{T}'_h$
\begin{subnumcases}{\label{eqn: SP_DG_scheme}}
 \label{eqn: SP_DG_scheme_1}  ((v_h)_t,\varphi)_{I'_j} = (u_h,\varphi)_{I'_j} + <\widehat{\omega_h},\varphi>_{I'_j} - (\omega_h, \varphi_x)_{I'_j}, \\
 \label{eqn: SP_DG_scheme_2}  (v_h,\phi)_{I'_j} = <\widehat{u_h}, \phi>_{I'_j} - (u_h, \phi_x)_{I'_j}, \\
 \label{eqn: SP_DG_scheme_3} (\omega_h, \psi)_{I'_j} = <\widehat{f(u_h)},\psi>_{I'_j} - (f(u_h), \psi_x)_{I'_j},
\end{subnumcases}
where $f(u) = \frac{1}{6}u^3$. The ``hat" terms in the scheme are the so-called ``numerical fluxes'', which are functions defined on the cell boundary from integration by parts and should be designed based on different guiding principles for different PDEs to ensure the stability and local solvability of the intermediate variables.  To ensure the scheme is $E_0$ conserved, the numerical fluxes we take are
\begin{subnumcases}{\label{eqn: flux}}
\widehat{\omega_h} = \average{\omega_h}, \ \widehat{u_h} = \average{u_h}, \label{eqn: flux1} \\
\widehat{f(u)} =  \begin{cases} \frac{\jump{F(u)}}{\jump{u}}, \; &\jump{u} \neq 0, \\ f(\average{u}),\; &\jump{u} = 0,\end{cases} \label{eqn: flux2}
\end{subnumcases}
where $ F(u) = \int^u f(\tau)d\tau$.

\begin{prop}(Energy conservation) The DG scheme \eqref{eqn: SP_DG_scheme} with the numerical fluxes \eqref{eqn: flux} for the short pulse equation \eqref{eqn:short pulse} satisfies the energy conservativeness
\begin{align}
\frac{d}{dt}E_0(u_h)  = 0.
\end{align}
\end{prop}
\begin{proof}
For the equation \eqref{eqn: SP_DG_scheme_2}, we take the time derivative and get
\begin{equation}\label{eqn:time derivate term}
((v_h)_t,\eta)_{I'_j} = <\widehat{(u_h)}_t, \eta>_{I'_j} - ((u_h)_t, \eta_x)_{I'_j}. \\
\end{equation}
Since \eqref{eqn:time derivate term}, and \eqref{eqn: SP_DG_scheme_1}-\eqref{eqn: SP_DG_scheme_3} hold for any test functions in $V^k_h$, we can choose
\begin{equation}
\varphi = (u_h)_t, \ \eta = -(u_h)_t, \ \psi = -u_h,
\end{equation}
and it follows that
\begin{align}
\label{eqn:SPDG_1}((v_h)_t,(u_h)_t)_{I'_j}& = (u_h, (u_h)_t)_{I'_j} + <\widehat{\omega_h},(u_h)_t>_{I'_j}- (\omega_h,(u_h)_{tx})_{I'_j}, \\
\label{eqn:SPDG_2}-((v_h)_t,(u_h)_t)_{I'_j} &= -<\widehat{(u_h)}_t,(u_h)_t>_{I'_j} + ((u_h)_t,(u_h)_{tx})_{I'_j}, \\
\label{eqn:SPDG_3}-(\omega_h,u_h)_{I'_j}& = -<\widehat{f(u_h)},(u_h)>_{I'_j}+ (f(u_h),(u_h)_{x})_{I'_j}.
\end{align}
To eliminate extra terms, we take test functions $\varphi = -\omega_h$ in \eqref{eqn: SP_DG_scheme_1}, $\eta = \omega_h$ in \eqref{eqn:time derivate term}, and then obtain
\begin{align}
\label{eqn:SPDG_4}-((v_h)_t,\omega_h)_{I'_j}& = -(u_h, \omega_h)_{I'_j} - <\widehat{\omega_h},\omega_h>_{I'_j}+ (\omega_h,(\omega_h)_{x})_{I'_j}, \\
\label{eqn:SPDG_5}((v_h)_t,\omega_h)_{I'_j}& = <\widehat{(u_h)}_t,\omega_h>_{I'_j} - ((u_h)_t,(\omega_h)_{x})_{I'_j}.
\end{align}
With these choices of test functions and summing up the five equations in \eqref{eqn:SPDG_1}-\eqref{eqn:SPDG_5}, we get
\begin{equation}\label{eqn:SPDG_energy}
\begin{split}
(u_h,(u_h)_t)_{I'_j} &+ <\widehat{\omega_h},(u_h)_t>_{I'_j} - (\omega_h,(u_h)_{tx})_{I'_j} - <\widehat{(u_h)}_t,(u_h)_t>_{I'_j} + ((u_h)_t,(u_h)_{tx})_{I'_j} \\
& + <\widehat{(u_h)}_t,\omega_h>_{I'_j} - ((u_h)_t,(\omega_h)_{x})_{I'_j} - <\widehat{f(u_h)},(u_h)>_{I'_j} + (f(u_h),(u_h)_{x})_{I'_j}\\
& - <\widehat{\omega_h},\omega_h>_{I'_j}+ (\omega_h,(\omega_h)_{x})_{I'_j} = 0.
\end{split}
\end{equation}
Now the equation \eqref{eqn:SPDG_energy} can be rewritten into following form
\begin{equation}\label{eqn:cell entropy}
(u_h,(u_h)_t)_{I'_j} + \Phi_{j+\frac{1}{2}} - \Phi_{j-\frac{1}{2}} + \Theta_{j-\frac{1}{2}} = 0,
\end{equation}
where the numerical entropy flux $\Phi$ is given by
\begin{align}
\Phi &= \widehat{\omega_h}(u_h^-)_t  - \omega_h^-(u_h^-)_t +\frac{1}{2}(u_h^-)_t^2 -\widehat{(u_h)}_t(u_h^-)_t \\
&+ \widehat{(u_h)}_t\omega_h^- -\frac{1}{2}(\omega^-_h)^2 +\widehat{\omega_h}\omega_h^-
     - \widehat{f(u_h)}u_h^- + F(u_h^-),
\end{align}
and the extra term $\Theta$ is defined as
\begin{equation}
\begin{split}
\Theta &=  -\widehat{\omega_h}\jump{(u_h)_t} - \widehat{(u_h)_t}\jump{w_h} + \jump{w_h(u_h)_t} + \widehat{f(u_h)}\jump{u_h} - \jump{F(u_h)}\\ &+(\widehat{(u_h)_t}-\average{(u_h)_t})\jump{(u_h)_t}+(-\widehat{\omega_h}+\average{\omega_h})\jump{\omega_h}=0,
\end{split}
\end{equation}
which vanishes due to choices of the conservative numerical fluxes \eqref{eqn: flux}.
Summing up the cell entropy equalities \eqref{eqn:cell entropy} with periodic or homogeneous Dirichlet boundary conditions, implies that
\begin{equation}
(u_h,(u_h)_t)_{I'} = 0.
\end{equation}
Thus, the DG scheme \eqref{eqn: SP_DG_scheme} for the short pulse equation is  $E_0$ conserved.
\end{proof}

The $E_0$ conserved DG scheme resolves the smooth solutions for the short pulse equation well, as we show in Section \ref{Numeical experiment}. Numerically, $E_0$ conserved DG scheme can achieve $(k+1)$-$th$ order of accuracy for even $k$, and $k$-$th$ order of accuracy for odd $k$. However, for the loop-soliton and cuspon-soliton solutions, this scheme can not be used due to the singularity of solutions. So we introduce the  DG schemes via hodograph transformations in the following sections.

\subsubsection{$H_0$ conserved DG scheme}\label{sec:H0_SP}
As we have mentioned, the short pulse equation can be converted into the coupled dispersionless (CD) system through the hodograph transformation.
To construct the local discontinuous Galerkin numerical method for the CD system, we first rewrite \eqref{eqn:CD system} as a first order system
\begin{subnumcases}{\label{H0_ODE}}
\rho_s + \gamma_y = 0, \label{eqn:H0_ODE1}\\
\omega_s = \rho u, \label{eqn:H0_ODE2}\\
\omega = u_y, \label{eqn:H0_ODE3}\\
\gamma = \frac{1}{2}u^2.  \label{eqn:H0_ODE4}
\end{subnumcases}
Then we can formulate the LDG numerical method as follows: Find  $ u_h $, $ \rho_h $, $\omega_h$, $\gamma_h\in V_h^k $  such that
\begin{subnumcases}{\label{eqn: CD_numerical-scheme}}
((\rho_h)_s, \phi)_{I_j} + <\widehat{\gamma_h}, \phi>_{I_j} - (\gamma_h, \phi_y)_{I_j} = 0, \\
((\omega_h)_{s}, \varphi)_{I_j} = (\rho_h u_h, \varphi)_{I_j},\\
 (\omega_h, \psi)_{I_j} = <\widehat{u_h}, \psi>_{I_j} - (u_h, \psi_y)_{I_j}, \label{eqn: CD_numerical-scheme_3}\\
(\gamma_h,\eta)_{I_j} = (\frac{1}{2}u_h^2, \eta)_{I_j}
\end{subnumcases}
for all test functions $ \phi $, $ \varphi $, $ \psi $, $\eta$ $ \in V_h^k $ and $I_j \in \mathcal{T}_h$.
To guarantee the conservativeness of $H_0$,  we adopt the central numerical fluxes
\begin{align}\label{eqn:upwind}
\widehat{\gamma_h} = \average{\gamma_h}, \ \widehat{u_h} = \average{u_h}.
\end{align}
Numerically we will see that the optimal  $(k+1)$-$th$ order of accuracy can be obtained for $u_h$, $\rho_h$ when $k$ is even, however, the numerical solutions $u_h$, $\rho_h$ have $k$-$th$ order of accuracy when $k$ is odd. If we modify numerical fluxes as below:
\begin{align}\label{flux:SP}
\widehat{\gamma_h} = \average{\gamma_h}-\alpha\jump{\rho_h}-\beta\jump{\gamma_h}, \ \widehat{u_h} = \average{u_h} + \mu\jump{u_h},
\end{align}
 then the scheme is dissipative on $H_0$ with the appropriate parameters  in Proposition \ref{prop:H0} and the optimal order of accuracy can be achieved numerically for this $H_0$ dissipative scheme.

\begin{prop}\label{prop:H0}($H_0$ conservation$/$dissipation)
 The semi-discrete DG numerical scheme \eqref{eqn: CD_numerical-scheme}, \eqref{eqn:upwind} can preserve quantity $H_0(\rho_h,u_h) = \int_I \rho_h u_h^2\ dy $ spatially. And the scheme \eqref{eqn: CD_numerical-scheme} with \eqref{flux:SP} composes a dissipative DG scheme on $H_0$ if the parameters in \eqref{flux:SP} satisfy the conditions
\begin{equation}
\alpha = 0, \ \beta \geq 0, \ \mu \geq 0 \ \text{and} \ \beta + \mu \neq 0.
\end{equation}

\end{prop}
\begin{proof}
First, we take time derivative of equation \eqref{eqn: CD_numerical-scheme_3}, and the test functions are chosen as $\phi = \gamma_h, \varphi = -(u_h)_s, \psi = (u_h)_s, \eta = -(\rho_h)_s$.  Then we have
\begin{align}
\label{eqn:one}   &((\rho_h)_s, \gamma_h)_{I_j} + <\widehat{\gamma_h}, \gamma_h>_{I_j} - (\gamma_h, ({\gamma_h})_y)_{I_j} = 0, \\
\label{eqn:second}-&((\omega_h)_{s}, (u_h)_s)_{I_j} + (\rho_h u_h, (u_h)_s)_{I_j} = 0 ,\\
\label{eqn:third} &((\omega_h)_s, (u_h)_s)_{I_j} - <\widehat{(u_h)_s}, (u_h)_s>_{I_j} + ((u_h)_s, (u_h)_{sy})_{I_j} = 0,\\
\label{eqn:fourth}-&(\gamma_h,(\rho_h)_s)_{I_j} +(\frac{1}{2}u_h^2, (\rho_h)_s)_{I_j} = 0.
\end{align}
Summing up all equalities \eqref{eqn:one}-\eqref{eqn:fourth}, we obtain
\begin{equation}
\begin{split}
&(\rho_h u_h, (u_h)_s)_{I_j} + (\frac{1}{2}u_h^2, (\rho_h)_s)_{I_j}+\\
& <\widehat{\gamma_h}, \gamma_h>_{I_j} - (\gamma_h, ({\gamma_h})_y)_{I_j}
- <\widehat{(u_h)_s}, (u_h)_s>_{I_j} + ((u_h)_s, (u_h)_{sy})_{I_j} = 0,
\end{split}
\end{equation}
which can be written as
\begin{align}\label{eqn:H0_cell_entropy1}
\frac{1}{2}\frac{d}{ds} \int_{I_j} \rho_h u_h^2 dy + \Phi_{j+\frac{1}{2}} - \Phi_{j-\frac{1}{2}} + \Theta_{j-\frac{1}{2}} = 0,
\end{align}
where the numerical entropy fluxes are given by
\begin{equation}
\Phi = \widehat{\gamma_h}\gamma_h^- - \frac{1}{2}(\gamma_h^-)^2 - \widehat{(u_h)_s}(u_h)_s^- + \frac{1}{2}((u_h)_s^-)^2
\end{equation}
and the extra term $\Theta$ is
\begin{equation}
\begin{split}
\Theta = (-\widehat{\gamma_h}+\average{\gamma_h})\jump{\gamma_h}+ (\widehat{(u_h)_s} - \average{(u_h)_s})\jump{(u_h)_s}.
\end{split}
\end{equation}
Therefore the choices of $\widehat{\gamma_h}, \widehat{u_h}$ in \eqref{flux:SP}
concern the conservativeness of the DG scheme. According to the parameters $\alpha,\beta,\mu $ , we give below two cases:
\begin{align}
\label{flux_theta_1}&\text{ $H_0$ conserved DG scheme $\alpha = \beta = \mu = 0$ :}  \; &\Theta = 0;   \\
\label{flux_theta_2}&\text{ $H_0$ dissipative DG scheme  $\alpha = 0, \beta = \frac{1}{2}, \mu =\frac{1}{2}$:} \; &\Theta = \frac{1}{2}(\jump{\gamma^2}+ \jump{u_s^2}) \geq 0    .
\end{align}
Summing up the cell entropy equalities \eqref{eqn:H0_cell_entropy1} and \eqref{flux_theta_1}, \eqref{eqn:H0_cell_entropy1} and \eqref{flux_theta_2}, respectively, then we get
\begin{equation}
\begin{split}
&\frac{1}{2}\frac{d}{ds} \int_I \rho_h u_h^2 \ dy  = 0,\ \text{ $H_0$ conserved DG scheme}, \\
&\frac{1}{2}\frac{d}{ds} \int_I \rho_h u_h^2 \ dy  \leq 0, \ \text{ $H_0$ dissipative DG scheme}. \label{eqn:H0_flu}
\end{split}
\end{equation}
\end{proof}

In the numerical  test Example \ref{ex1}, it shows that the dissipative scheme with parameters \eqref{flux_theta_2} can achieve the optimal convergence rate for both $u_h$ and $\rho_h$ no matter $k$ is odd or even. However, the order of accuracy for the  $H_0$ conserved DG scheme is $k$-$th$ for odd $k$, $(k+1)$-$th$ for even $k$.
The choices of these parameters are not unique, but the above numerical fluxes in the dissipative scheme can minimize the stencil as in \cite{Zhang2019CiCP}.

\subsubsection{$H_1$ conserved DG scheme}
In this section, we construct another discontinuous Galerkin scheme which preserves the quantity $H_1$ of the CD system \eqref{eqn:CD system} which links the Hamiltonian $E_1$ of the short pulse equation through the hodograph transformation.
First, we rewrite the CD system as a first order system
\begin{subnumcases}{\label{eqn:H1_CD}}
\rho_s + u\omega = 0, \label{eqn:H1_CD_1}\\
 w_{s} = \rho u, \label{eqn:H1_CD_2}\\
 \omega = u_y \label{eqn:H1_CD_3}.
\end{subnumcases}
Then the semi-discrete LDG numerical scheme can be constructed as: Find  $ u_h $, $ \omega_h $, $ \rho_h $ $\in V_h^k$ such that
\begin{subnumcases}{\label{eqn:numerical_conserved_DG_CD}}
((\rho_h)_s, \phi)_{I_j} + (u_h\omega_h,\phi)_{I_j}  = 0, \\
((\omega_h)_{s}, \varphi)_{I_j} = (\rho_h u_h, \varphi)_{I_j}, \\
(\omega_h, \psi)_{I_j} = <\widehat{u_h},\psi_y>_{I_j} - (u_h, \psi_y)_{I_j}, \label{eqn:numerical_conserved_DG_CD_3}
\end{subnumcases}
for all test functions $ \phi $, $ \varphi $, $ \psi $ $ \in V_h^k $ and $I_j \in \mathcal{T}_h$.  The numerical flux is taken as  $\widehat{u_h} = u_h^+$. Numerically, the optimal  $(k+1)$-$th$ order of accuracy can be obtained for both $u_h$, $\rho_h$. If we take ${\widehat{u_h} = \average{u_h}}$, then the accuracy is   $(k+1)$-$th$ order for $u_h$, $\rho_h$ when $k$ is even, and $k$-$th$ order of accuracy when $k$ is odd.
\begin{prop}($H_1$ conservation)
The semi-discrete DG numerical scheme  \eqref{eqn:numerical_conserved_DG_CD} can preserve the quantity $H_1(\rho_h,\omega_h) = \int_I (\rho_h^2 + \omega_h^2) dy$ spatially.
\end{prop}
\begin{proof}
By taking the test functions $\phi = \rho_h, \varphi = \omega_h$ in \eqref{eqn:numerical_conserved_DG_CD}, we obtain
\begin{align*}
&((\rho_h)_s, \rho_h)_{I_j} + (u_h\omega_h,\rho_h)_{I_j}  = 0, \\
&((w_h)_{s}, \omega_h)_{I_j} = (\rho_h u_h, \omega_h)_{I_j}.
\end{align*}
Summing up over all $I_j \in \mathcal{T}_h$, it implies that
\begin{align*}
\frac{1}{2}\frac{d}{ds} \int_{I} \rho_h^2 + \omega_h^2 \ dy = 0.
\end{align*}
\end{proof}
In what follows, we prepare to give the a priori error estimate for the $H_1$ conserved DG scheme. The standard $L^2$ projection of a function $\zeta$ with $k+1$ continuous derivatives into space $V_h^k$, is denoted by $\mathcal{P}$, that is, for each $I_j$
\begin{equation}\label{projection}
\begin{split}
&\int_{I_j}(\mathcal{P}\zeta - \zeta)\phi\ dx =0, \ \forall \phi \in P^{k}(I_j),
\end{split}
\end{equation}
and the special projections $\mathcal{P}^{\pm}$ into $V_h^k$  satisfy, for each $I_j$
\begin{align}\label{eqn:special_projection}
\int_{I_j}(\mathcal{P}^{+}\zeta - \zeta)\phi\ dx =0, \ \forall \phi \in P^{k-1}(I_j), \ \text{and} \ \mathcal{P}^{+}\zeta(y_{j-\frac{1}{2}}^+) = \zeta({y_{j-\frac{1}{2}}}),\\
\int_{I_j}(\mathcal{P}^{-}\zeta - \zeta)\phi\ dx =0, \ \forall \phi \in P^{k-1}(I_j),\ \text{and} \
\mathcal{P}^{-}\zeta(y_{j+\frac{1}{2}}^-) = \zeta({y_{j+\frac{1}{2}}}).
\end{align}
For the projections mentioned above, it is easy to show \cite{1975_Ciarlet_NH} that
\begin{equation}\label{projection error}
\norm{ \zeta^e}_{L^2(I)} +  h^{\frac{1}{2}} \norm{ \zeta^e}_{L^{\infty}(I)} + h^{\frac{1}{2}}\norm{ \zeta^e}_{L^2({\partial I})}  \leq Ch^{k+1},
\end{equation}
where $\zeta^e =\zeta -\mathcal{P}\zeta  $ or $\zeta^e =\zeta -\mathcal{P}^{\pm}\zeta  $ , and the positive constant $C$ only depends on $\zeta$. There is an inverse inequality we will use in the subsequent proof. For $\forall u_h \in V_h^k$, there exists a positive constant $\sigma$ (we call it the inverse constant), such that
\begin{equation}\label{eqn:inverse inequality}
\norm{u_h}_{L^2(\partial{I})} \leq \sigma h^{-\frac{1}{2}}\norm{u_h}_{L^2(I)},
\end{equation}
where $\norm{u_h}_{L^2(\partial{I})}  = \sum\limits_{j=1}^{N+1}\sqrt{((u_h)_{j+\frac{1}{2}}^-)^2 +((u_h)_{j-\frac{1}{2}}^+)^2}$.

First, we write the error equations of the $H_1$ conserved DG scheme as follows:
\begin{align}
\label{eqn:error_eqn1} ((\rho-\rho_h)_s,\varphi)_{I_j} &= -(u\omega - u_h\omega_h,\varphi)_{I_j},\\
\label{eqn:error_eqn2}((\omega-\omega_h)_s, \phi)_{I_j}& =(\rho u - \rho_hu_h, \phi)_{I_j},\\
(\omega - \omega_h,\psi )_{I_j} &= <\widehat{u - u_h},\psi>_{I_j} - (u-u_h,\psi_y)_{I_j},
\end{align}
and denote
\begin{equation}\label{def: xi and eta}
\begin{split}
 & \eta^{u} = u - \mathcal{P}^+u , \ \xi^{u} = \mathcal{P}^+u - u_h ,\\
 &\eta^{\rho} = \rho - \mathcal{P}\rho ,\ \  \xi^{\rho} = \mathcal{P}\rho - \rho_h,\\
 &\eta^{\omega} =\omega - \mathcal{P}\omega , \ \xi^{\omega} = \mathcal{P}\omega - \omega_h.
\end{split}
\end{equation}
To deal with the term $\xi^u$, we need to establish a relationship between $\xi^u$ and $\xi^{\omega}$ in following lemma.

\begin{lemma}\label{lemma:H1_conserved} The $\xi^u,\xi^{\omega},\eta^{\omega}$ are defined in \eqref{def: xi and eta}, then there exists a positive constant $C_{\sigma, p}$ independent of $h$ but depending on inverse constant $\sigma$ and Poincar\'{e} constant $C_p$, such that
\begin{equation}\label{eqn:4.17}
\norm{\xi^u}_{L^2(I)} \leq  C_{\sigma, p}(\norm{\xi^{\omega}}_{L^2(I)} + \norm{\eta^{\omega}}_{L^2(I)}).
\end{equation}
\end{lemma}
\begin{proof}
By Poincar\'{e} -Friedrichs inequality in Chapter $10$ of \cite{2007_brenner_SSBM}, we have
\begin{align*}
\norm{\xi^u}_{L^2(I)} &\leq C_p \Big[\norm{\xi^u_y}_{L^2(I)} + h^{-\frac{1}{2}}\norm{\jump{\xi^u}}_{L^2(\partial I)}            \Big]
\end{align*}
where
\begin{align*}
\norm{\xi^u_y}_{L^2(I)} = \Big(\sum\limits_{I_j\in \mathcal{T}_h} \int_{I_j}(\xi^u_y)^2 dy\Big)^{\frac{1}{2}},\; \; \norm{\jump{\xi^u}}_{L^2(\partial I)}  = \Big(\sum\limits_{j=1}^{N+1} \jump{\xi^u}_{j-\frac{1}{2}}^2\Big)^{\frac{1}{2}}.
\end{align*}
The inequality (4.17) in \cite{Wang_2015_SJNA} gives
\begin{align*}
\Big[ \norm{\xi^u_y}_{L^2(I)} + h^{-\frac{1}{2}}\norm{\jump{\xi^u}}_{L^2(\partial I)}             \Big]
        \leq C_{\sigma}(\norm{\xi^{\omega}}_{L^2(I)}+\norm{\eta^{\omega}}_{L^2(I)} ),
\end{align*}
which implies that
\begin{align*}
\norm{\xi^u}_{L^2(I)} \leq  C_{\sigma, p}(\norm{\xi^{\omega}}_{L^2(I)} + \norm{\eta^{\omega}}_{L^2(I)}).
\end{align*}
\end{proof}

\begin{theorem}\label{thm:H1_conserved}
Assume that the system \eqref{eqn:H1_CD} with Dirichlet boundary condition has a smooth solution $u, \rho, \omega $. Let $u_h, \rho_h, \omega_h$ be the numerical solution of the semi-discrete DG scheme \eqref{eqn:numerical_conserved_DG_CD}. And there exists that initial conditions $u_h^0, \omega_h^0, \rho_h^0$ satisfy the  following approximation property
\begin{equation}
\norm{u^0 -u_h^0}_{L^2(I)}  + \norm{\rho^0 -\rho_h^0}_{L^2(I)}  + \norm{\omega^0 -\omega_h^0}_{L^2(I)}  \leq Ch^{k+1}.
\end{equation}
Then for regular triangulations of $I = (y_L, y_R)$, and  the finite element
space $V^k_h$ with $k \geq 0$, there holds the following error estimate
\begin{equation}
\norm{u-u_h}_{L^2(I)} + \norm{\omega-\omega_h}_{L^2(I)}  + \norm{\rho - \rho_h}_{L^2(I)}  \leq Ch^{k+1},
\end{equation}
where the positive constant $C$ depends on the final time $T$ and the exact solutions.
\end{theorem}

\begin{proof}
We rewrite the error equation \eqref{eqn:error_eqn1}, \eqref{eqn:error_eqn2} as
\begin{equation}
\begin{split}
((\xi^{\rho}+ \eta^{\rho})_s,\varphi)_{I_j} & = (-u\omega + u_h\omega_h,\varphi)_{I_j} \\
                                   & = (-\omega(\xi^u+\eta^u)- u_h(\xi^{\omega}+ \eta^{\omega}), \varphi)_{I_j},\\
((\xi^{\omega}+ \eta^{\omega})_s,\phi)_{I_j}& = (\rho u- \rho_hu_h,\phi)_{I_j} \\
                                   & =  (\rho(\xi^u+\eta^u)_{I_j} + u_h(\xi^{\rho}+\eta^{\rho}), \phi)_{I_j}.\\                                     \end{split}
\end{equation}
Taking the test functions $\varphi = \xi^{\rho}, \phi = \xi^{\omega}$, we have
\begin{equation}
\begin{split}
(\xi^{\rho}_s, \xi^{\rho})_{I_j} = &(-\eta^{\rho}_s, \xi^{\rho})_{I_j} - (\omega(\xi^u+\eta^u)+ u_h(\xi^{\omega}+ \eta^{\omega}), \xi^{\rho})_{I_j},\\
(\xi^{\omega}_s, \xi^{\omega})_{I_j} = &(-\eta^{\omega}_s, \xi^{\omega})_{I_j} + (\rho(\xi^u+\eta^u) + u_h(\xi^{\rho}+\eta^{\rho}), \xi^{\omega})_{I_j}.
\end{split}
\end{equation}
Summing up  over all interval $I_j$ and omitting the subscript $I$,  we obtain
\begin{equation}\label{error_1}
\begin{split}
(\xi^{\rho}_s, \xi^{\rho})+(\xi^{\omega}_s, \xi^{\omega}) &= (-\eta^{\rho}_s - \omega(\xi^u+\eta^u)-u_h\eta^{\omega} , \xi^{\rho}) + (-\eta^{\omega}_s + \rho(\xi^u+\eta^u)+\eta^{\rho}u_h, \xi^{\omega})\\
&     =(-\eta^{\rho}_s - \mathcal{P}\omega(\xi^u+\eta^u)-u\eta^{\omega} , \xi^{\rho}) + (-\eta^{\omega}_s + \mathcal{P}\rho(\xi^u+\eta^u)+u\eta^{\rho}, \xi^{\omega})\\
& = (A,\xi^{\rho}) + (B,\xi^{\omega}) - (\mathcal{P}\omega\xi^u,\xi^{\rho})+ (\mathcal{P}\rho\xi^u,\xi^{\omega}).
\end{split}
\end{equation}
By the Cauchy-Schwarz and arithmetic-geometric mean inequalities, the equation becomes
\begin{equation}\label{error_2}
\begin{split}
\frac{d}{dt}(\norm{\xi^\omega}_{L^2(I)}^2 + \norm{\xi^\rho}_{L^2(I)}^2)   \leq K + \frac{1}{2}( 1+ C_{\omega,\rho})(\norm{\xi^{\omega}}_{L^2(I)}^2 + \norm{\xi^{\rho}}_{L^2(I)}^2) + C_{\omega,\rho}\norm{\xi^u}_{L^2(I)}^2
\end{split}
\end{equation}
where
\begin{align}
&A = -\eta^{\rho}_s - \mathcal{P}\omega\eta^u - u\eta^{\omega},\ B = -\eta^{\omega}_s + \mathcal{P}\rho\eta^u+u\eta^{\rho},\\
&K = \frac{1}{2}(\norm{A}_{L^2(I)}^2+ \norm{B}_{L^2(I)}^2), \ C_{\omega, \rho} =\max(c_{\rho},c_{\omega}).
\end{align}
Here, we need to interpret the constants we mentioned above. We denote by $C_{\ast}$ all positive constants independent of $h$, which depends on the subscript ${\ast}$.

Using Lemma \ref{lemma:H1_conserved} in the above inequality \eqref{error_2}, we get
\begin{equation}\label{error_3}
\begin{split}
\frac{d}{dt}(\norm{\xi^\omega}_{L^2(I)}^2 + \norm{\xi^\rho}_{L^2(I)}^2) &\leq \tilde{K} + \tilde{C_{\sigma,\omega,\rho}}(\norm{\xi^\omega}_{L^2(I)}^2 + \norm{\xi^\rho}_{L^2(I)}^2),
\end{split}
\end{equation}
where
\begin{equation}
\tilde{K} = \frac{1}{2}(\norm{A}_{L^2(I)}^2+ \norm{B}_{L^2(I)}^2) + C_{\sigma,\omega,\rho}\norm{\eta^\omega}_{L^2(I)}^2 , \ \tilde{C_{\sigma,\omega,\rho}} = \frac{1}{2}(1 +C_{\omega,\rho} + C_{\sigma,\omega,\rho}).
\end{equation}
By the Gronwall inequality, we have
\begin{align*}
\norm{\xi^\omega}_{L^2(I)}^2 + \norm{\xi^\rho}_{L^2(I)}^2 \leq  Ch^{2k+2}.
\end{align*}
 Lemma \ref{lemma:H1_conserved} also implies that
\begin{align*}
\norm{\xi^u}_{L^2(I)}^2  \leq  Ch^{2k+2}.
\end{align*}
Then Theorem \ref{thm:H1_conserved} follows by the triangle inequality and the interpolating property.
\end{proof}

\subsubsection{Integration DG scheme}\label{subsubsec:integration}
Instead of the scheme \eqref{eqn:numerical_conserved_DG_CD_3} to solve $u_h$ from $\omega_h$, we can also integrate the equation $u_y = \omega$ directly referring to \cite{Xu2010_JCM}. We give an integration DG scheme defined as follows: Find $u_h \in V_h^{k+1}$, $\rho_h, \omega_h \in V_h^{k}$, such that, for all test functions $\phi$, $\varphi \in V_h^{k}$ and $I_j\in\mathcal{T}_h$

\begin{subnumcases}{\label{eqn: H1_integration_numerical_scheme}}
((\rho_h)_s, \phi)_{I_j} + (u_h\omega_h,\phi)_{I_j}  = 0, \\
((\omega_h)_{s}, \varphi)_{I_j} = (\rho_h u_h, \varphi)_{I_j}, \\
 u_h(y,s})\mid_{I_j}  = u_h({y_{j+\frac{1}{2}},s) - \int_y^{y_{j+\frac{1}{2}}}\omega_h(\xi,s) \ d\xi
\end{subnumcases}
with the boundary condition $u_h(y_{N+\frac{1}{2}},s) = u(y_R,s)$. Here, $ u_h $ is not longer in $V_h^k$ space but in $V_h^{k+1}$ space and continuous. Numerically, this scheme can obtain the optimal order of accuracy, i.e. $(k+2)$-$th$ order for $u_h$, and $(k+1)$-$th$ order for $\rho_h$.
\begin{remark}
It is notable that our integration DG scheme \eqref{eqn: H1_integration_numerical_scheme} is based on the $H_1$ conserved DG scheme \eqref{eqn:numerical_conserved_DG_CD}. Actually, following the $H_0$ conserved DG scheme \eqref{eqn: CD_numerical-scheme}, we have another integration DG scheme:
Find  $ u_h \in V_h^{k+1}$, $ \rho_h $, $\omega_h$, $\gamma_h\in V_h^k $  such that
\begin{subnumcases}{}
((\rho_h)_s, \phi)_{I_j} + <\widehat{\gamma_h}, \phi>_{I_j} - (\gamma_h, \phi_y)_{I_j} = 0, \\
((\omega_h)_{s}, \varphi)_{I_j} = (\rho_h u_h, \varphi)_{I_j},\\
(\gamma_h,\eta)_{I_j} = (\frac{1}{2}u_h^2, \eta)_{I_j},\label{eqn:L2-projection}\\
u_h(y,s})\mid_{I_j}  = u_h({y_{j+\frac{1}{2}},s) - \int_y^{y_{j+\frac{1}{2}}}\omega_h(\xi,s) \ d\xi,
\end{subnumcases}
for all test functions $ \phi $, $ \varphi $, $\eta$ $ \in V_h^k $ and $I_j \in \mathcal{T}_h$.
The only difference between those two integration DG schemes is the $L^2$ projection \eqref{eqn:L2-projection} . Numerically, there is little difference on accuracy and conserved quantities $H_0$, $H_1$. In the following sections, the integration DG scheme we mention refers to the numerical scheme \eqref{eqn: H1_integration_numerical_scheme}.
\end{remark}
To prove the error estimate of this integration scheme, we introduce following lemma to build the relationship between $\xi^u$ and $\xi^\omega$.

\begin{lemma}\label{lemma:integration}
In this lemma, $u,\omega$ is the exact solution of CD system \eqref{eqn:H1_CD} with Dirichlet boundary, $\omega$ is the derivative of $u$ with respect to $y$. The numerical solutions $u_h, \omega_h$ of the integration DG scheme \eqref{eqn: H1_integration_numerical_scheme} for CD system satisfy $(u_h)_y = \omega_h$, and $\xi^u,\xi^{\omega},\eta^{u},\eta^{\omega}$ are defined in \eqref{def: xi and eta}, then we have the following relationship:
\begin{equation}\label{eqn:integration_lemma}
\norm{\xi^u_y}_{L^2(I)} + h^{-\frac{1}{2}}\norm{\jump{\xi^u}}_{L^2(I)} \leq C_{\sigma}(\norm{\xi^\omega}_{L^2(I)}+\norm{\eta^\omega}_{L^2(I)} + h^{-1}\norm{\eta^u}_{L^2(I)})
\end{equation}
where the positive constant $C_\sigma$ depends on the inverse constant $\sigma$.
\end{lemma}
\begin{proof}
First we write the error equation
\begin{equation}\label{eqn:H1_error_equation}
\xi^u_y+\eta_y^u = \xi^\omega+\eta^\omega.
\end{equation}
Here we take a test function $\phi$ as follows:
\begin{equation}
\phi(y)_{I_j} = \xi^u_y(y) - (\xi_y^u)_{j+\frac{1}{2}}^-L_{k+1}(\zeta)
\end{equation}
where $L_k$ is the standard Legendre polynomial of degree $k-1$ in $[-1,1]$, $\zeta = 2(y-y_j)/h_j$. We have $L_k(1) = 1$ and $L_k$ is orthogonal to any polynomials with degree at most $k-1$. Therefore we obtain some relevant properties $\phi^-_{j+\frac{1}{2}} = 0$, $(\xi_y^u, \phi)_{I_j} = (\xi_y^u, \xi_y^u)_{I_j}$.
Then we multiply the error equation \eqref{eqn:H1_error_equation} by test function $\phi(y)$ and integrate it over $I_j$, which follows that
\begin{align}
(\xi^u_y, \phi)_{I_j} + (\eta^u_y, \phi)_{I_j} = (\xi^\omega, \phi)_{I_j}+(\eta^\omega, \phi)_{I_j}, \label{eqn:lemma2_11}\\
(\xi^u_y, \xi^u_y)_{I_j} + (\eta^u)^-_{j+\frac{1}{2}}\phi_{j+\frac{1}{2}}^- - (\eta^u)^+_{j-\frac{1}{2}}\phi_{j-\frac{1}{2}}^+ -(\eta^u, \phi_y)_{I_j} = (\xi^\omega, \phi)_{I_j}+(\eta^\omega, \phi)_{I_j},\\
(\xi^u_y, \xi^u_y)_{I_j} = (\xi^\omega, \phi)_{I_j}+(\eta^\omega, \phi)_{I_j}.
\end{align}
Since $\norm{L_k(\zeta)}_{L^2(I_j)} \leq Ch^{\frac{1}{2}}_{j}$, we use Cauchy-Schwarz inequality and the inverse property \eqref{eqn:inverse inequality}, and obtain
\begin{align}
\norm{\xi^u_y}^2_{L^2(I_j)}&\leq (\norm{\xi^u_y}_{L^2(I_j)}+\abs{(\xi_y^u)_{j+\frac{1}{2}}^-}\norm{L_k(\zeta)}_{L^2(I_j)})(\norm{\xi^\omega}^2_{L^2(I_j)} +\norm{\eta^\omega}^2_{L^2(I_j)}) \nonumber \\
&\leq C_{\sigma}\norm{\xi^u_y}_{L^2(I_j)}(\norm{\xi^\omega}^2_{L^2(I_j)} +\norm{\eta^\omega}^2_{L^2(I_j)}).
\end{align}
Hence we arrive at
\begin{equation}
\norm{\xi^u_y}_{L^2(I_j)}\leq C_{\sigma}(\norm{\xi^\omega}^2_{L^2(I_j)} +\norm{\eta^\omega}^2_{L^2(I_j)}).
\end{equation}
Next, for the boundary term $\jump{\xi^u}_{j-{\frac{1}{2}}}$, the deduction process is as follows:
\begin{align}
&\jump{\xi^u}_{j-\frac{1}{2}} = \jump{\mathcal{P}^+u-u_h}_{j-\frac{1}{2}} = \jump{\mathcal{P}^+u}_{j-\frac{1}{2}}
= \jump{\mathcal{P}^+u-u}_{j-\frac{1}{2}} = \jump{\eta^u}_{j-\frac{1}{2}}.
\end{align}
Here the equalities hold due to the continuity of $u$ and $u_h$. Taking account of the projection error \eqref{projection error}, we have
\begin{equation}
\begin{split}
&\norm{\jump{\xi^u}}_{L^2(I_j)} = \norm{\jump{\eta^u}}_{L^2(I_j)} \leq C h^{-\frac{1}{2}}\norm{\eta^u}_{L^2(I_j)}
\end{split}
\end{equation}
Finally, by summing over all cells $I_j$ it follows the result \eqref{eqn:integration_lemma}\end{proof}

Next, we imitate Theorem \ref{thm:H1_conserved} and give the error estimate for the integration DG scheme \eqref{eqn: H1_integration_numerical_scheme}.
\begin{theorem}\label{thm:integration}
Assume that the system \eqref{eqn:H1_CD} with Dirichlet boundary condition has a smooth solution $u, \rho, \omega $. Let $u_h, \rho_h, \omega_h$ be the numerical solution of the semi-discrete DG scheme \eqref{eqn: H1_integration_numerical_scheme}. And there exists that initial conditions $u_h^0, \omega_h^0, \rho_h^0$ satisfy the  following approximation property
\begin{align}
 &\norm{u^0 -u_h^0}_{L^\infty(I)} + \norm{\rho^0 -\rho_h^0}_{L^2(I)}  + \norm{\omega^0 -\omega_h^0}_{L^2(I)}  \leq Ch^{k+1}.
\end{align}
Then for a regular tessellation of $I = (y_L, y_R)$, and  the finite element
spaces $V^k_h$ and $V^{k+1}_h$  with $k \geq 0$, there holds the following error estimates
\begin{align}
 &\norm{u-u_h}_{L^\infty(I)} + \norm{\omega-\omega_h}_{L^2(I)}  + \norm{\rho - \rho_h}_{L^2(I)}  \leq Ch^{k+1}
 \end{align}
where the positive constant $C$ depends on the final time $T$ and the exact solutions.
\end{theorem}

\begin{proof}
First, we write the error equations of the integration DG scheme \eqref{eqn: H1_integration_numerical_scheme} as follows:
\begin{subnumcases}{\label{eqn:integration error_eqn}}
((\rho-\rho_h)_s,\varphi)_{I_j} = -(u\omega - u_h\omega_h,\varphi)_{I_j},\label{eqn:integration error_eqn1}\\
((\omega-\omega_h)_s, \phi)_{I_j} =(\rho u - \rho_hu_h, \phi)_{I_j},\label{eqn:integration error_eqn2}\\
(u - u_h)(y) = (u - u_h)(y_{j+\frac{1}{2}}) - \int_y^{y_j+\frac{1}{2}} (\omega -\omega_h)(\zeta)\ d\zeta .\label{eqn:integration error_eqn3}
\end{subnumcases}
Then the process of this proof is similar to the proof of Theorem \ref{thm:H1_conserved}. We can adopt the proof of Theorem \ref{thm:H1_conserved} until estimate \eqref{error_2}. The relationship between $\xi^u$ and $\xi^w$ for the integration DG scheme which has been given in Lemma \ref{lemma:integration} guarantees that the proof can be continued. Under this circumstance, $u\in V^{k+1}_h$ and $\norm{\eta^u} < Ch^{k+2}$. So we can still have the same result as Theorem \ref{thm:H1_conserved},
\begin{align*}
\norm{\xi^\omega}_{L^2(I)}^2 + \norm{\xi^\rho}_{L^2(I)}^2 \leq  Ch^{2k+2}.
\end{align*}
Then followed by the triangle inequality and the interpolating property, there holds
\begin{align}
 \norm{\omega-\omega_h}_{L^2(I)}  + \norm{\rho - \rho_h}_{L^2(I)}  \leq Ch^{k+1}.
\end{align}
 Here, we can obtain the $L^\infty$ error estimate for $u_h$. For $\forall y \in [y_L, y_R]$, we apply Cauchy-Schwarz inequality on error equation \eqref{eqn:integration error_eqn3}, and  the following estimate holds
\begin{equation}\label{eqn:error_u}
\abs{(u - u_h)(y)} =  \abs{(u - u_h)(y_R) + \int_y^{y_R} (\omega-\omega_h)(\xi)d\xi}
                   \leq C\norm{\omega-\omega_h}_{L^2(I)}.
\end{equation}
The boundary term vanishes in \eqref{eqn:error_u} due to the boundary condition $u(y_R,s) = u_h(y_{N+\frac{1}{2}},s)$. Thus we arrive at $\norm{u-u_h}_{L^\infty} \leq Ch^{k+1}$.
\end{proof}

\subsection{The sine-Gordon equation}\label{SG and SP}
  There exists some short pulse type equations which are failed to be transformed into the corresponding CD systems to solve, e.g. some examples in Section \ref{Extension}. In this section, we can adopt an alternative approach by linking the short pulse equation \eqref{eqn:short pulse} with the sine-Gordon equation.

  First, we consider the sine-Gordon equation
\begin{equation}\label{eqn:sine-Gordon}
z_{ys} = \sin{z}, \ z\in\mathbb{R},
\end{equation}
where $s$ denotes the time coordinate, $y$ is the spatial scale and $y \in I = [y_L,y_R]$. The hodograph transformation between the short pulse equation and the sine-Gordon equation is as follows:
\begin{equation}\label{eqn:hodograph_SG_SPE}
\begin{cases}
\frac{\partial}{\partial y} = \cos z \frac{\partial}{\partial x}, \\
\frac{\partial}{\partial s} = \frac{\partial}{\partial t} - \frac{(z_s)^2}{2}\frac{\partial}{\partial x}.
\end{cases}
\end{equation}
 The parametric representation of the solution of the short pulse equation \eqref{eqn:short pulse} is
\begin{align}
u = z_s(y,s), \ x = x(y_0,s) + \int_{y_0}^y \cos z\ dy.
\end{align}
The sine-Gordon equation has a conserved quantity $H_2 = \int_I z_y^2 \ dy$ which can be preserved discretely in the following DG scheme.

\subsubsection{DG scheme for sine-Gordon equation }\label{subsection:SGSP}

In this subsection, we develop the DG scheme for  the sine-Gordon equation \eqref{eqn:sine-Gordon}. We divide the sine-Gordon equation into these first-order equations
\begin{equation}
\begin{cases}
&\omega_s = \eta,\\
&\eta = \sin{z}, \\
&\omega = z_y.
\end{cases}
\end{equation}
The semi-discrete DG numerical method for sine-Gordon equation is defined as follows: Find $z_h, \eta_h, \omega_h \in V_h^k$, such that, for all test functions $\varphi$, $\phi$, $\psi \in V_h^k$ and $I_j \in \mathcal{T}_h$
\begin{subnumcases}{\label{eqn:numerical_SG}}
((\omega_h)_{s},\varphi)_{I_j} = (\eta_h,\varphi)_{I_j}, \label{eqn:numerical_SG_1}\\
(\eta_h,\phi)_{I_j} = (\sin{z_h},\phi)_{I_j}, \label{eqn:numerical_SG_2}\\
(\omega_h,\psi)_{I_j} = <\widehat{z_h}, \psi>_{I_j} - (z_h,\psi_y)_{I_j},\label{eqn:numerical_SG_3}
\end{subnumcases}
where the numerical flux is $\widehat{z_h} = z_h^+$. Numerically, this scheme can achieve the optimal order of accuracy for $z_h$.

To maintain the quantity $H_2$ conserved, we can also choose a special numerical flux $\widehat{z_h}$
\begin{equation}\label{eqn: SG_flux}
\begin{split}
&\widehat{z_h} =  \begin{cases} \frac{\jump{z_h\eta_h}- \jump{\cos{z_h}}}{\jump{\eta_h}}, \; &\jump{\eta_h} \neq 0, \\ z_h. \; &\jump{\eta_h} = 0, \end{cases}
\end{split}
\end{equation}
which makes the scheme $H_2$ conservative. But  in the actual computation, this conservative scheme will increase the complexity for solving a nonlinear system. For the convenience of the simulation, we adopt the numerical flux $\widehat{z_h} = z_h^+$ for the DG scheme \eqref{eqn:numerical_SG} in our numerical tests.

\subsubsection{Integration DG scheme for sine-Gordon equation }\label{subsection:integration SG}
We can also solve $\omega = z_y$  by integration directly as the scheme \eqref{eqn: H1_integration_numerical_scheme} in
Section \ref{subsubsec:integration}. The semi-discrete integration DG numerical scheme for sine-Gordon equation \eqref{eqn:sine-Gordon} is formulated as: Find $z_h \in V_h^{k+1}$, $\eta_h, \omega_h \in V_h^k$, such that, for all test functions $\varphi$, $\phi$ $\in V_h^k$ and $I_j \in \mathcal{T}_h$,
\begin{subnumcases}{\label{eqn:numerical_SG_integration}}
((\omega_h)_{s},\varphi)_{I_j} = (\eta_h,\varphi)_{I_j}, \label{eqn:numerical_SG_integration1}\\
(\eta_h,\phi)_{I_j} = (\sin{z_h},\phi)_{I_j}, \label{eqn:numerical_SG_integration2}\\
z_h(y,s)\mid_{I_j}  = z_h(y_{j+\frac{1}{2}},s) - \int_y^{y_{j+\frac{1}{2}}}\omega_h(\xi,s) \ d\xi,\label{eqn:numerical_SG_integration2}
\end{subnumcases}
with boundary condition $z_h(y_{N+\frac{1}{2}},s) = z(y_R,s)$. Similarly, this scheme can achieve the optimal $(k+2)$-$th$ order of accuracy for $z_h$ numerically.

Since $ u = z_s $, when we obtain $z_h$ in each time level, $u^n_h$ can be computed by a fourth order approximation
\begin{equation}\label{ut}
u^n_h = \frac{2z_h^{n+1}+3z_h^n -6z_h^{n-1}+z_h^{n-2}}{6\Delta s}.
\end{equation}
Here, we just give one of approximation methods to get $u_h$ by $z_h$ as an example, and this discretization method is not unique.

\section{Extensions to other cases}\label{Extension}
In this section, we  consider some generalized short pulse type equations. Similar numerical schemes including $H_0$, $H_1$ conserved DG schemes and integration DG scheme in Section \ref{subsec:SP} for the corresponding CD systems, DG scheme and integration DG scheme in Section \ref{SG and SP} for sine-Gordon type equations can be also adopted to solve these generalized short pulse type equations via hodograph transformations. For simplicity, we just introduce these generalized equations and the conservative quantities considered in the numerical schemes.

\subsection{The coupled short pulse equation}\label{subsec:CSP}
The short pulse equation can be generalized to the coupled short pulse equations
\begin{equation}\label{eqn:CSP1}
\left.
\begin{array}{r}
 u_{xt} = u + \frac{1}{2}(uvu_x)_x , \\
  \ v_{xt} = v + \frac{1}{2}(uvv_x)_x,
\end{array}
\right\} \ u, v\in\mathbb{R},
\end{equation}
 which can be converted into a coupled CD system
\begin{equation}\label{eqn:CCD1}
\rho_s + \frac{1}{2}(uv)_y = 0, \ u_{ys} = \rho u, \ v_{ys} = \rho v,
\end{equation} proposed by Konno and Kakuhata in \cite{Kakuhata_1996_JPSJ}. 
For the coupled short pulse equation \eqref{eqn:CSP1}, we set $u, v \in \mathbb{C}$ and $v = u^*$ which denotes the complex conjugate of $u$. Then we have complex short pulse equation derived from \cite{Feng_2015_PDNP}
\begin{equation}\label{eqn: complex_SPE1}
\begin{split}
 &u_{xt} = u + \frac{1}{2}(\abs{u}^2u_x)_x, \ u\in \mathbb{C},
\end{split}
\end{equation}
which is related to the complex CD system
\begin{equation}\label{eqn:ComplexCD1}
\rho_s + \frac{1}{2}\abs{u}^2_y = 0, \ u_{ys} = \rho u, \ u^*_{ys} = \rho u^*.
\end{equation}
And according to the real-valued case, we also give the complex form of the coupled short pulse equation,
\begin{equation}\label{eqn:BCCSP}
\left.
\begin{array}{r}
  u_{xt} = u + \frac{1}{2}((\abs{u}^2 + \abs{v}^2)u_x)_x, \\
  v_{xt} = v + \frac{1}{2}((\abs{v}^2 + \abs{u}^2)v_x)_x,
\end{array}
\right\} \ u, v \in\mathbb{C}.
\end{equation}
Through the corresponding hodograph transformation, it can be transformed into
\begin{equation}\label{eqn:BCCD}
 \rho_s + \frac{1}{2}(\abs{u}^2)_y + \frac{1}{2}(\abs{v}^2)_y = 0, \
 u_{sy} = \rho u, \
 v_{sy} = \rho v.
\end{equation}

Similar to the short pulse equation \eqref{eqn:short pulse}, the corresponding $H_0$, $H_1$ conserved DG schemes and integration DG scheme can be constructed to solve above three short pulse type equations \eqref{eqn:CSP1}, \eqref{eqn: complex_SPE1}, \eqref{eqn:BCCSP}, respectively. Here, we list the conserved quantities $H_0$, $H_1$ in Table \ref{tab:conservedtab}.
\begin{table}[H]
\begin{tabular}{|c|c|c|c|}
\hline
           & Coupled CD \eqref{eqn:CCD1}  & Complex CD \eqref{eqn:ComplexCD1} & Coupled CD in complex form \eqref{eqn:BCCD} \\\hline
     $H_0$ &$\int\rho uv dy$ & $\int\rho \abs{u}^2 dy $ & $ \int \rho (\abs{u}^2 + \abs{v}^2) \ dy$ \\
     $H_1$ &$\int\rho^2 + u_y v_y dy$ &$\int\rho^2 + \abs{u_y}^2 dy$ & $\int \rho^2 + \abs{u_y}^2 + \abs{v_y}^2 dy$\\\hline
\end{tabular}
\caption{\label{tab:conservedtab} The conserved quantities $H_0$, $H_1$ for three generalized CD systems.}
\end{table}

 \subsection{The modified short pulse equation}
For nonlinear wave equation
\begin{equation}
u_{xt} = u + au^2u_{xx} + buu_x^2,
\end{equation}
if its coefficient ratio $a/b$ equals $1$ instead of $\frac{1}{2}$, then we obtain the modified short pulse equation after rescaling the variable $u$,
\begin{equation}\label{eqn:MSP}
\ u_{xt} = u + \frac{1}{2}u(u^2)_{xx}, \ u\in\mathbb{R}.
\end{equation}
It can be converted to a modified CD system  \cite{Matsuno_2016_JMP}
\begin{equation}\label{eqn:MCD}
\rho_s + (u^2)_y = 0, \ u_{ys} = (2\rho-1)u,
\end{equation}
for which we can build $H_0$, $H_1$ conserved and integration DG schemes.

Next, some generalized modified short pulse equation will be introduced.  It is worth to mention that there is one type generalized modified short pulse equation in Section \ref{secsec:link SG},  which can not  be transformed into the CD system but only the sine-Gordon equations through the hodograph transformation.

 \subsubsection{The generalized modified short pulse systems}\label{secsec:link CD}
The modified short pulse equation 
\begin{equation}\label{eqn:CMSP}
\left.
\begin{array}{r}
  u_{xt}  = u + \frac{1}{2}v(u^2)_{xx}, \\
  v_{xt} = v + \frac{1}{2}u(v^2)_{xx},
\end{array}
\right\}\ u,v \in\mathbb{R},
\end{equation}
 connects with the coupled modified CD system
\begin{equation}\label{eqn:ComplexMCD}
\rho_s + (uv)_y = 0, \ u_{ys} = (2\rho-1)u, \ v_{ys} = (2\rho-1)v.
\end{equation}
When $u, v \in \mathbb{C}$ and $v = u^*$  in \eqref{eqn:CMSP},
 the defocusing type of complex modified short pulse equation mentioned in \cite{Shen_2017_JNMP} is
\begin{equation}\label{eqn:MCSP_defocusing}
u_{xt} = u - \frac{1}{2}u^*(u^2)_{xx}, \ u\in\mathbb{C},
\end{equation}
and the corresponding CD system is
\begin{equation}\label{eqn:Complex_defocusing_MCD}
\rho_s - \abs{u}^2_y = 0, \ u_{ys} = (2\rho-1)u, \ u^*_{ys} = (2\rho-1)u^* .
\end{equation}
Similarly, we can develop the corresponding $H_0$, $H_1$ conserved DG schemes and integration DG scheme to solve the modified short pulse type equations \eqref{eqn:MSP},
\eqref{eqn:CMSP}, \eqref{eqn:MCSP_defocusing} via the hodograph transformations. The conserved quantities $H_0$, $H_1$ are contained in Table \ref{tab:conservedtab2}.
\begin{table}[H]
\begin{tabular}{|c|c|c|c|}
\hline
           & MCD \eqref{eqn:MCD}  & Complex MCD \eqref{eqn:ComplexMCD} & Defocusing complex MCD \eqref{eqn:Complex_defocusing_MCD} \\\hline
     $H_0$ &$\int (2\rho - 1) u^2 dy$ & $\int(2\rho-1)uv\ dy $ & $ \int (2\rho - 1)\abs{u}^2 dy$ \\
     $H_1$ &$\int \rho^2 + u^2_y \ dy$ &$\int \rho^2 + u_yv_y dy$ & $\int \rho^2 + \abs{u_y}^2 dy$\\\hline
\end{tabular}
\caption{\label{tab:conservedtab2} The conserved quantities $H_0$, $H_1$ for three generalized MCD systems.}
\end{table}

\subsubsection{Sine-Gordon type equations}\label{secsec:link SG}
 The modified short pulse equation \eqref{eqn:MSP} is linked with the sine-Gordon equation
 \begin{equation}
 z_{ys} = \sin{z}\cos{z}
 \end{equation}
which is equivalent to $z_{ys} =\frac{1}{2} \sin 2z $.
Similarly, we can develop DG schemes for the sine-Gordon equation to solve the modified short pulse equation as in Section \ref{SG and SP}.

We consider another integrable generalization of the modified short pulse equation, so called the novel coupled short pulse equation proposed by Feng in \cite{Feng_2012_JPMT}
\begin{equation}\label{New coupled SP_1}
\left.
\begin{array}{r}
   u_{xt} = u + \frac{1}{6}(u^3)_{xx} + \frac{1}{2}v^2u_{xx},\\
 v_{xt} = v + \frac{1}{6}(v^3)_{xx} + \frac{1}{2}u^2v_{xx},
\end{array}
\right\} \ u, v\in\mathbb{R},
\end{equation}
 which is failed to be transformed to the CD system. But  it can be transformed to the following two-component sine-Gordon system
\begin{equation}\label{eqn:Coupled_sine-Gordon_1}
z_{ys} = \sin{z},\ \tilde{z}_{ys} = \sin \tilde{z},
\end{equation}
 by the hodograph transformation
\begin{equation}\label{eqn:hodograph_SG_SPE}
\begin{cases}
&\frac{\partial}{\partial y} = \frac{1}{2}(\cos z + \cos \tilde{z})\frac{\partial}{\partial x}, \\
&\frac{\partial}{\partial s} = \frac{\partial}{\partial t} - \frac{u^2 + v^2}{2}\frac{\partial}{\partial x}.
\end{cases}
\end{equation}
Hence, we can develop the DG and integration DG schemes for  this coupled system \eqref{eqn:Coupled_sine-Gordon_1}  as we did for the sine-Gordon equation in  Section \ref{SG and SP}.

\section{Numerical experiments}\label{Numeical experiment}
In this section we will provide several numerical experiments to illustrate the accuracy and capability of the DG methods. Time discretization is the fourth order explicit Runge-Kutta method in \cite{Gottlieb_2001_SIAM}. This time discretization method may not ensure the conservativeness of fully discretization schemes. However, we will not address the issue of time discretization conservativeness in this paper.  In order not to repeat, we always choose one of $H_0$, $H_1$ conserved and integration DG schemes to profile the numerical solution in the subsequent numerical experiments.

\begin{example}
In this example, smooth solutions of traveling waves with periodic boundary condition for the short pulse equation \eqref{eqn:short pulse} are used to test the $E_0$ conserved DG method in Section \ref{sec:E0_SP}. We list one of the exact solutions derived in \cite{Matsuno_2008_JMP}, which is given by the exact solution of  the sine-Gordon equation \eqref{eqn:sine-Gordon} in the coordinate $(y,s)$ and need to be transformed to the coordinate  $(x,t)$ in the following test,
\begin{equation}\label{sol:periodic}
\begin{split}
&u(y,s) = \frac{2\kappa}{a}cn(\eta,\kappa), \\
&x(y,s) = x_0 + \frac{1}{a^2}(1-2\kappa^2)s +  \frac{1}{a}(-\eta + 2E(\eta,\kappa)) + d,\\
&\eta = ay - \frac{s}{a} + \eta_0,
\end{split}
\end{equation}
where $cn(\eta,\kappa)$ is Jacobi's $cn$ function.
The period is computed by $T_p =\frac{4}{a}\abs{-K(\kappa)+2E(\kappa)} $,   where $K(\kappa)$ and $E(\kappa)$ are the complete elliptic integrals of the first and second kinds, respectively \cite{Abramowitz_1972_Dover}.

The $L^2$ and $L^{\infty}$ errors and the convergence rates for $u_h$ are contained in Table \ref{tab:SPDG_E0_conserved}. For even $k$, we see the optimal order of accuracy, however, for odd $k$, we can only have suboptimal order of accuracy.  Figure \ref{fig:period} illustrates a typical periodic solution at $T = 1$ with the parameters $\kappa = 0.65, a = 1.3,x_0 = \eta_0 = d = 0$. This solution represents a periodic wavetrain traveling to the right with a constant velocity $V= 0.0917$.
\begin{table}[!htp]
\begin{center}
\begin{tabular}{|c|cccc|cccc|}
  \hline
                     &  \multicolumn{4}{c|}{$P^1$} & \multicolumn{4}{c|}{$P^2$}\\\hline
             N       &$ \norm{u-u_h}_{L^2}$    & order   &$ \norm{u-u_h}_{L^\infty}$   & order
                     &$ \norm{u-u_h}_{L^2}$    & order   &$ \norm{u-u_h}_{L^\infty}$   & order\\\hline

              40  &2.05E-02 &-- &9.17E-02 &--
                   &2.10E-05 &-- &1.12E-04 &--\\
              80  &1.09E-02 &0.91 &5.18E-02 &0.82
                   &2.30E-06 &3.19 &1.22E-05 &3.19\\
             160  &5.49E-03 &0.99 &2.90E-02 &0.84
                   &2.41E-07 &3.26 &1.40E-06 &3.13\\
             320  &2.76E-03 &0.99 &1.54E-02 &0.91
                   &3.19E-08 &2.92 &1.90E-07 &2.88\\\hline
\end{tabular}
\end{center}
\caption{\label{tab:SPDG_E0_conserved} Periodic solution \eqref{sol:periodic} of the short pulse equation \eqref{eqn:short pulse}: $E_0$ conserved DG scheme with computational domain $[0,T_p]$, at time $ T = 1$. The parameters $\kappa = 0.65, a = 1.3,x_0 = \eta_0 = d = 0$.} 
\end{table}
\begin{figure}[!htp]
\begin{center}
\begin{tabular}{cc}
\includegraphics[width=0.45\textwidth]{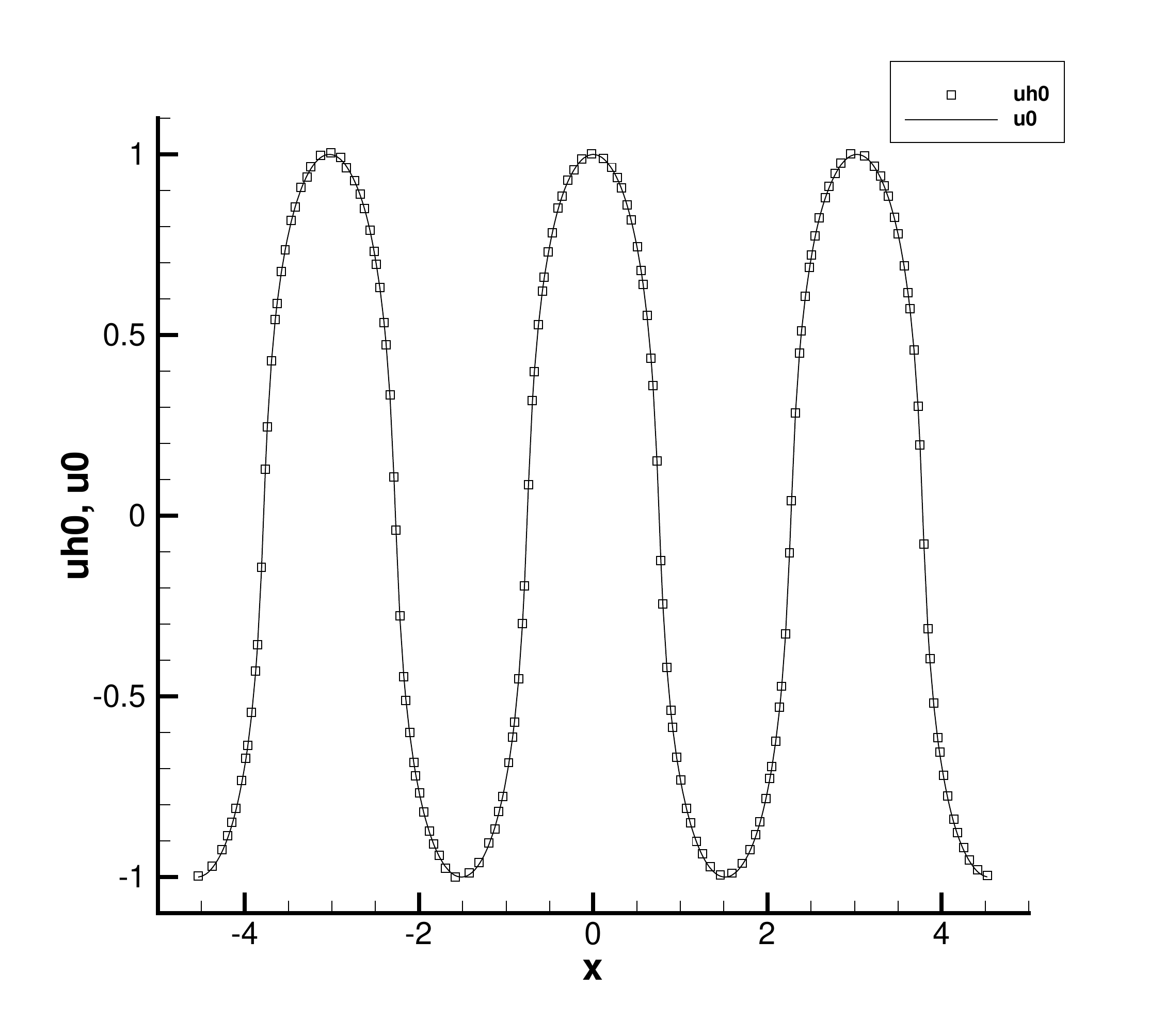}&\includegraphics[width=0.45\textwidth]{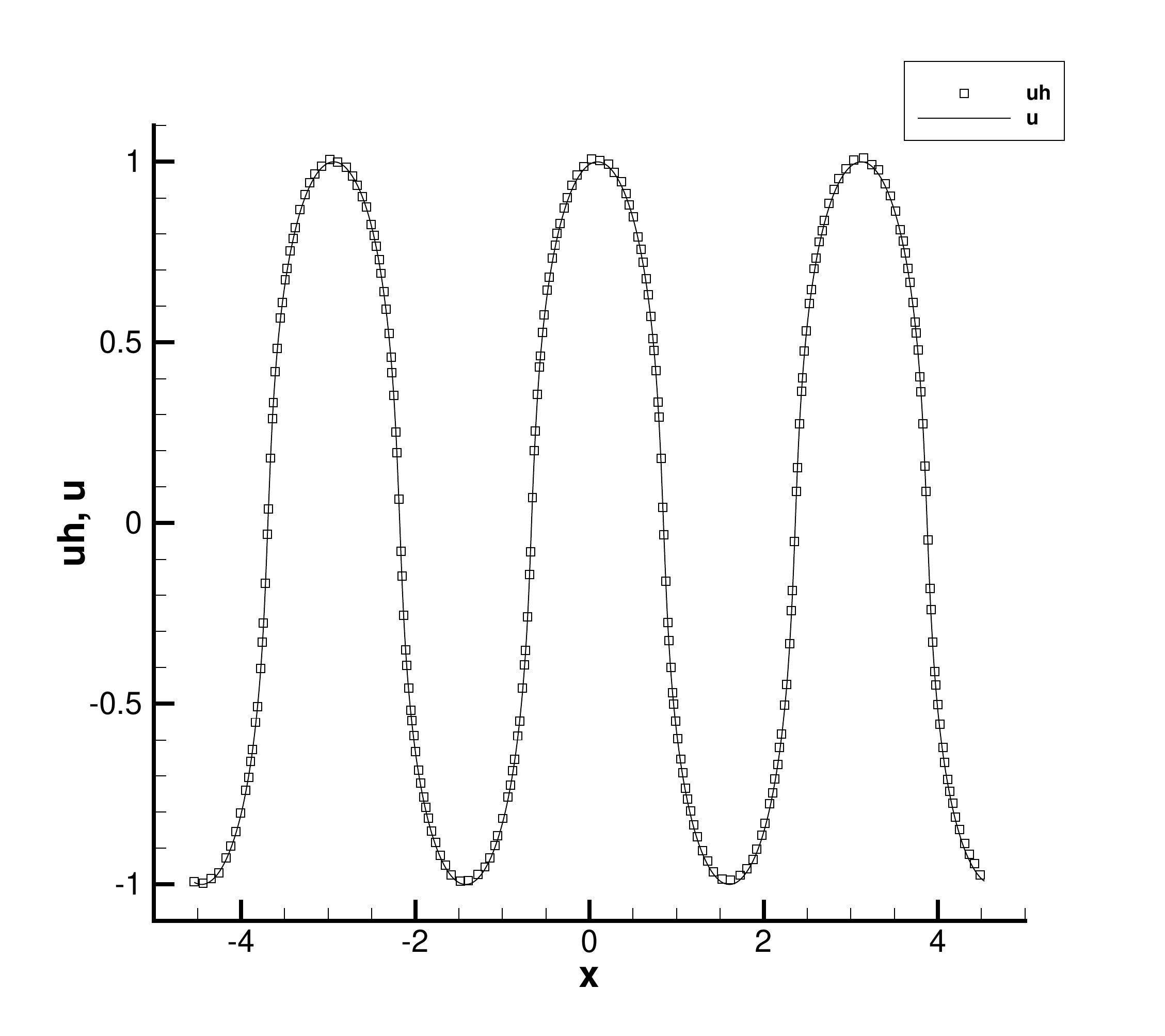}\\
(a) t = 0& (b) t = 1
\end{tabular}
\end{center}
\caption{\label{fig:period} Periodic solution \eqref{sol:periodic} of the short pulse equation \eqref{eqn:short pulse}: $E_0$ conserved DG scheme with computational domain $[-1.5T_p,1.5T_p]$ and $N = 160$ cells, $P^2$ elements at time $T = 0, 1$. The parameters $\kappa = 0.65, a = 1.3,x_0 = \eta_0 = d = 0$. }
\end{figure}
\end{example}

\begin{example}\label{ex1} This example is devoted to solve some loop-solition solutions of the short pulse equation \eqref{eqn:short pulse} by linking with CD system. We list the general determinant form of solution:
\begin{equation}\label{soliton_solution}
\begin{cases}
&u = \frac{g}{f}, \\
&x = x(y_0,s) + \int_{y_0}^y \rho(\zeta,s) d\zeta = y - 2(\ln f)_s, \ t = s,
\end{cases}
\end{equation}

\[\ f =
\left|\begin{array}{cc}
    A_I&    I   \\
    -I &    B_I  \\
\end{array}\right|,
\ g =
 \left|\begin{array}{ccc}
       A_I&       I &                \mathbf{e}^T \\
       -I &     B_I &                \mathbf{0}^T \\
\mathbf{0}&    -\mathbf{\alpha}^T &             0 \\
\end{array}\right|
\]
where $A_I, B_I, I \in \mathbb{R}^{m\times m}$, $m$ is integer denoting the number of soliton, and $I$ is an identity matrix, $\mathbf{0}, \mathbf{e}, \mathbf{\alpha} \in \mathbb{R}^{m}$ are $m$-component row vectors whose elements are defined, respectively, by
\begin{equation}\label{matrix_element}
\begin{split}
&a_{ij} = \frac{1}{2(\frac{1}{p_i}+ \frac{1}{p_j})}e^{\xi_i + \xi_j}, \ b_{ij} = \frac{\alpha_i\alpha_j}{2(\frac{1}{p_i}+ \frac{1}{p_j})}, \\
&\mathbf{e} = (e^{\xi_1}, e^{\xi_2}, \ldots, e^{\xi_m}),\ \mathbf{\alpha} = (\alpha_1,\alpha_2,\ldots,\alpha_m),
\text{with} \ \xi_i = p_iy + \frac{s}{p_i} + y_{i0}, \ i = 1, 2, \ldots, m.
\end{split}
\end{equation}
We fix the constant of integration in $x$ by choosing $ y_0 =0 $, and $p_i$, $\alpha_i $, $y_{i0}$ $\in \mathbb{C}$ are constants.

\newsavebox{\tablebox}
  \begin{lrbox}{\tablebox}
  \begin{tabular}{|c|c|cccc|cccc|}
  \hline
           & N       &$ \norm{u-u_h}_{L^2}$    & order   &$ \norm{u-u_h}_{L^\infty}$   & order
                     &$ \norm{\rho-\rho_h}_{L^2}$    & order   &$ \norm{\rho-\rho_h}_{L^\infty}$   & order\\\hline

                    &  40  &1.17E-02 & -- &8.00E-02 &--
                     &1.84E-02 &-- &2.30E-01 & --\\

  $H_0$                  &  80  &3.97E-04 &4.88 &2.37E-03 &5.08
                     &6.68E-04 &4.78 &1.07E-02 &4.43 \\

  dissipative                  & 160  &4.92E-05 &3.01 &2.78E-04 &3.09
                     &7.75E-05 &3.11 &1.17E-03 &3.19 \\

  DG scheme                  & 320  &6.85E-06 &2.85 &3.85E-05 &2.85
                     &1.03E-05 &2.91 &1.50E-04 &2.97 \\\hline

                    &  40  &7.92E-03 &-- &3.79E-02 &--
                     &6.59E-04 &-- &7.83E-03 &-- \\

 $H_0$                  &  80  &8.46E-05 &6.55 &6.17E-04 &5.94
                     &5.70E-05 &3.53 &9.23E-04 &3.08 \\

 conserved                   & 160  &1.04E-05 &3.03 &7.61E-05 &3.02
                     &7.45E-06 &2.93 &1.25E-04 &2.89 \\

 DG scheme                   & 320  &1.32E-06 &2.98 &9.60E-06 &2.99
                     &9.42E-07 &2.98 &1.56E-05 &3.00 \\\hline

                    &  40  &1.03E-02 &-- &4.67E-02 &--
                     &3.27E-04 &-- &3.97E-03 &-- \\

  $H_1$                  &  80  &9.88E-05 &6.70 &6.97E-04 &6.07
                     &3.41E-05 &3.26 &5.81E-04 &2.77 \\

  conserved                  & 160  &1.16E-05 &3.09 &8.26E-05 &3.08
                     &4.23E-06 &3.01 &7.94E-05 &2.87 \\

  DG scheme                  & 320  &1.43E-06 &3.02 &1.02E-05 &3.01
                     &5.29E-07 &3.00 &1.01E-05 &2.97 \\\hline

                     &  40  &1.77E-05 &-- &1.94E-04 &--
                     &2.32E-04 &-- &2.80E-03 &-- \\

   Integration                 &  80  &9.16E-07 &4.27 &1.75E-05 &3.48
                     &2.92E-05 &2.99 &4.59E-04 &2.61 \\

    DG scheme                & 160  &5.41E-08 &4.08 &1.17E-06 &3.90
                     &3.70E-06 &2.98 &6.07E-05 &2.92 \\

                    & 320  &3.33E-09 &4.02 &7.45E-08 &3.98
                     &4.65E-07 &2.99 &7.72E-06 &2.98 \\\hline

\end{tabular}
\end{lrbox}

 \begin{table}[!ht]
 \captionsetup{font={scriptsize}}
 \caption{\label{tab:SP_H1_conserved} 1-soliton solution of the CD system \eqref{eqn:CD system}: The computational domain $[-10,10]$, $P^2 $ elements, at time $T = 10$. The parameters $p_1 = 1.0, \alpha_1 = 4.0$.} 
 \centering
   \begin{threeparttable}
\scalebox{0.9}{\usebox{\tablebox}}
    \end{threeparttable}
\end{table}

The $ L^2, L^\infty $ error order of 1-soliton solution for CD system are calculated numerically and reported in Table \ref{tab:SP_H1_conserved}. We compare four kinds of DG schemes at time $T = 10$ with uniform meshes in $[-10,10]$. The $H_0$ dissipative DG numerical scheme with $\alpha = 0.1, \beta = 0, \mu = 0.5$ in \eqref{flux:SP} can reach the optimal $(k+1)$-$th$ order of accuracy for $u_h$ and $\rho_h$, so can $H_1$ conserved DG scheme. And the integration DG scheme has the optimal $(k+2)$-$th$ order of accuracy for $u_h$. It is notable that the $H_0$ conserved DG scheme can only reach the optimal error order when $k$ is even, but the suboptimal $k$-$th$ order of accuracy when $k$ is odd.
For the same order of accuracy, we notice that $H_0$ dissipative DG scheme is less accurate than other three schemes. For variable $\rho_h$, $H_1$ conserved DG scheme is more accurate than $H_0$ conserved DG scheme, and there is little difference for variable $u_h$. In these four schemes, the integration DG scheme is most accurate one.

\begin{table}[H]
\begin{tabular}{c|cc|cc|cc}
\hline

             &  \multicolumn{2}{c|}{$ H_0 $ conserved DG scheme} & \multicolumn{2}{c|}{$ H_1 $ conserved DG scheme}& \multicolumn{2}{c}{Integration DG scheme}   \\\hline
      $P^k$  &$\Delta H_0$ &$\Delta H_1$&  $\Delta H_0$& $\Delta H_1$  & $\Delta H_0$ &$\Delta H_1$\\\hline
      $P^2$  &1.96E-06 &2.96E-06&1.94E-06&1.23E-07&4.38E-07&1.24E-07 \\
      $P^3$  &1.51E-07 &1.16E-08&4.92E-08&2.18E-08&7.61E-08&5.88E-08 \\\hline

\end{tabular}
\caption{\label{tab:energy-change} The time evolution of conserved quantities for 1-soliton solution for the CD system \eqref{eqn:CD system}, with computational domain $[-20,20]$ and $N = 160$ cells at time $ T = 10.$}
\end{table}

The changes of quantity $H_0$ and $H_1$  with time are contained in Table \ref{tab:energy-change}. The quantities $\Delta H_0$, $\Delta H_1$ are defined as
\begin{equation}\label{ex1_delta_H}
\begin{split}
&\Delta H_0 =\sum\limits_{j=1}^{N+1} \abs{ \Big( \int_{I_j} \rho_h u_h^2\Big{|}_{t=T} \ dy -\int_{I_j} \rho_0u_0^2 \ dy\Big)}, \\
&\Delta H_1 = \sum\limits_{j=1}^{N+1}\abs{ \Big( \int_{I_j} \rho_h^2 + \omega_h^2\Big{|}_{t=T} \ dy -\int_{I_j} \rho_0^2 + \omega_0^2 \ dy\Big)}.
\end{split}
\end{equation}
where $\rho_0, u_0, \omega_0 $ are the initial conditions. Even though the fully discrete schemes may not be conservative, the two conserved quantities change slightly. To reduce the fluctuation of $H_0$ and $H_1$, we can increase the accuracy of temporal and spatial discretizations.

Next we show the capability of these DG schemes to simulate the singular soliton solutions. For the 2-soliton solution of the corresponding CD system, according to the choice of parameters, the solutions of the short pulse equation \eqref{eqn:short pulse} can be divide into three classes: two-loops-soliton solution in Figure \ref{fig:2loop_CD}, loop-antiloop-soliton solution in Figure \ref{fig:antiloop_CD}, two smooth-soliton solution (so called breather solution) in Figure \ref{fig:Breather_CD}. The DG schemes conduct accurate approximations for these soliton solutions.

\begin{figure}[!htp]
\begin{minipage}{0.49\linewidth}
  \centerline{\includegraphics[width=1.1\textwidth]{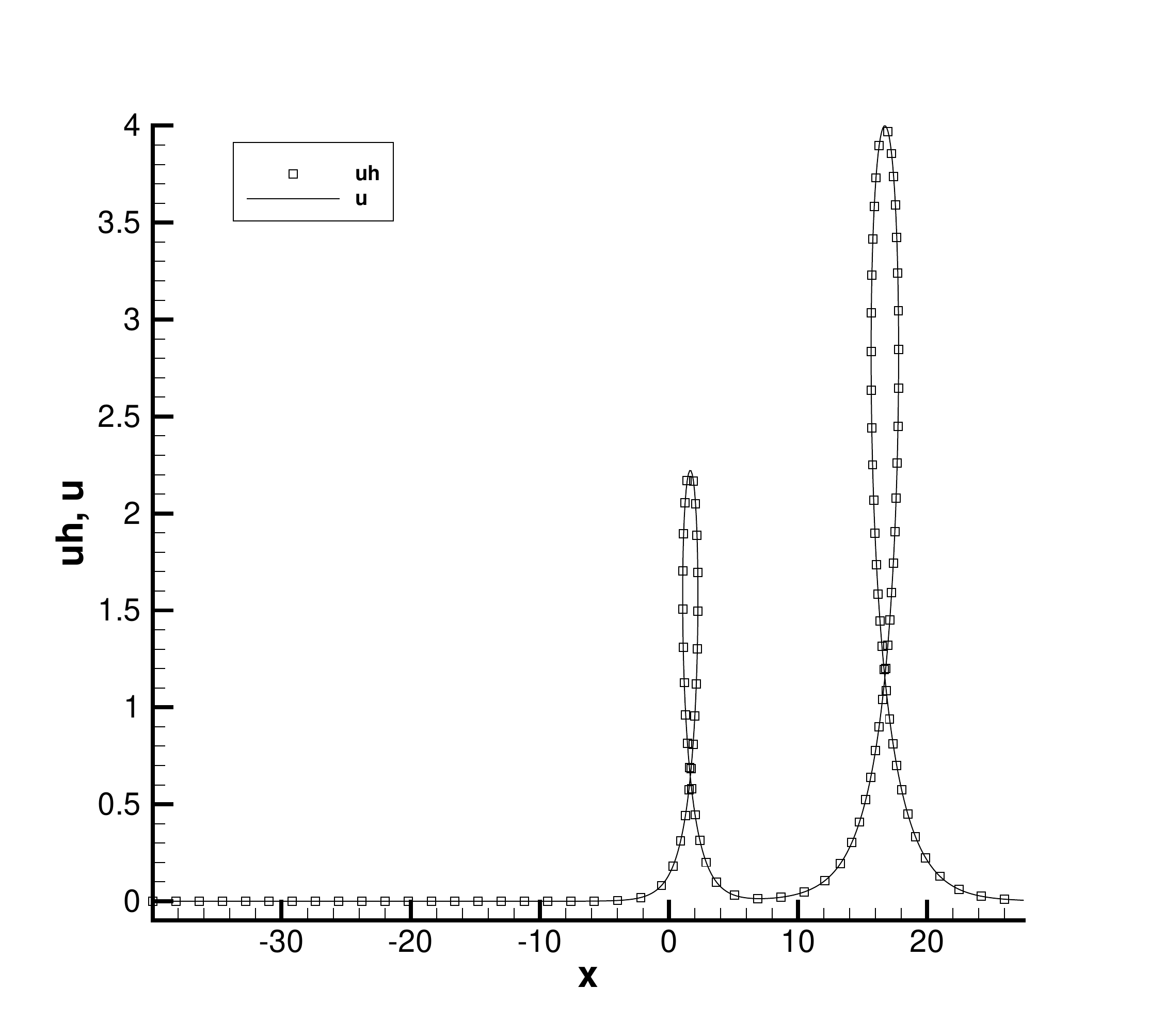}}
  \centerline{(a) t = 0.0}
\end{minipage}
\hfill
\begin{minipage}{0.49\linewidth}
  \centerline{\includegraphics[width=1.1\textwidth]{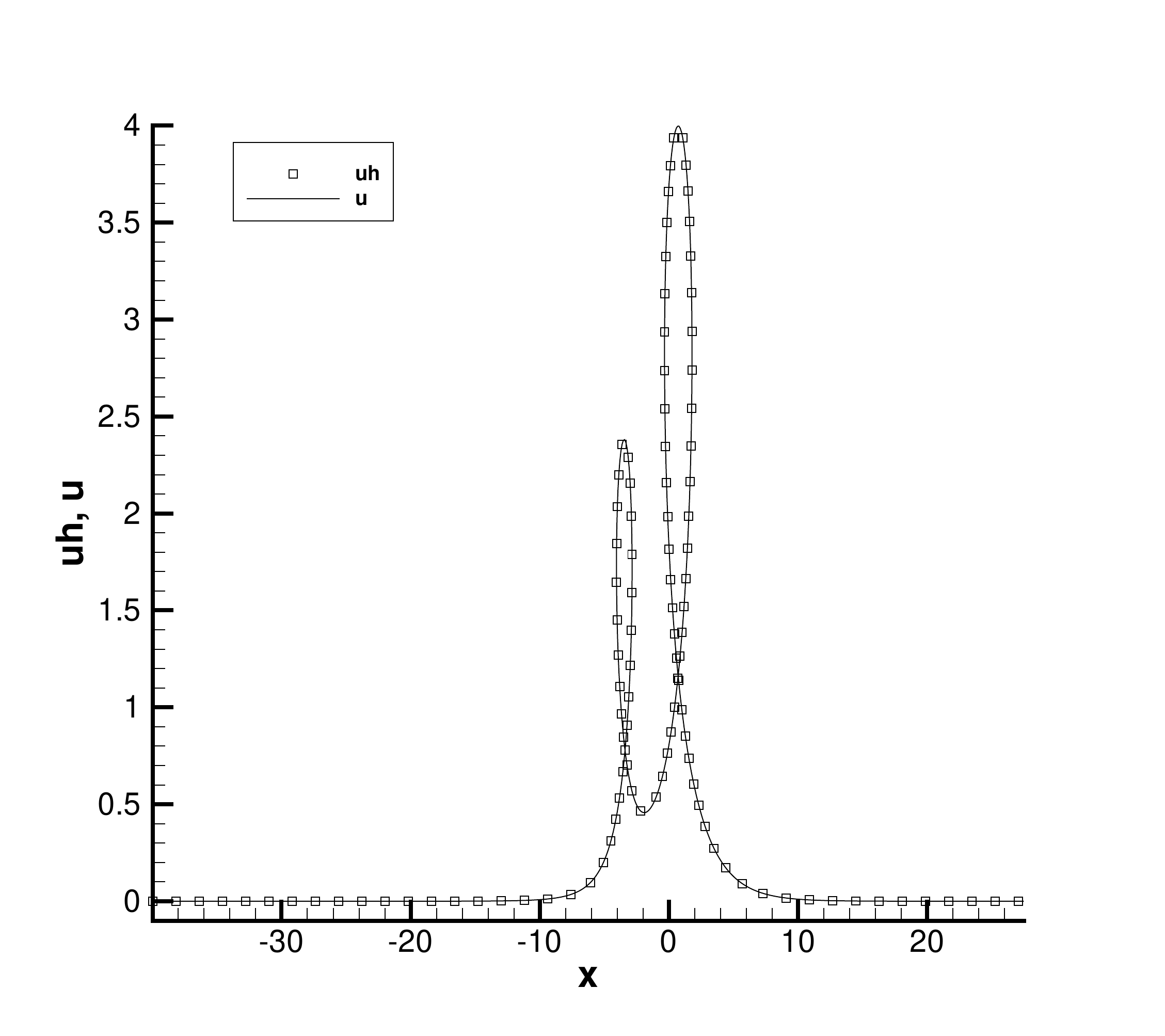}}
  \centerline{(b) t = 4.0}
\end{minipage}
\vfill
\begin{minipage}{0.49\linewidth}
  \centerline{\includegraphics[width=1.1\textwidth]{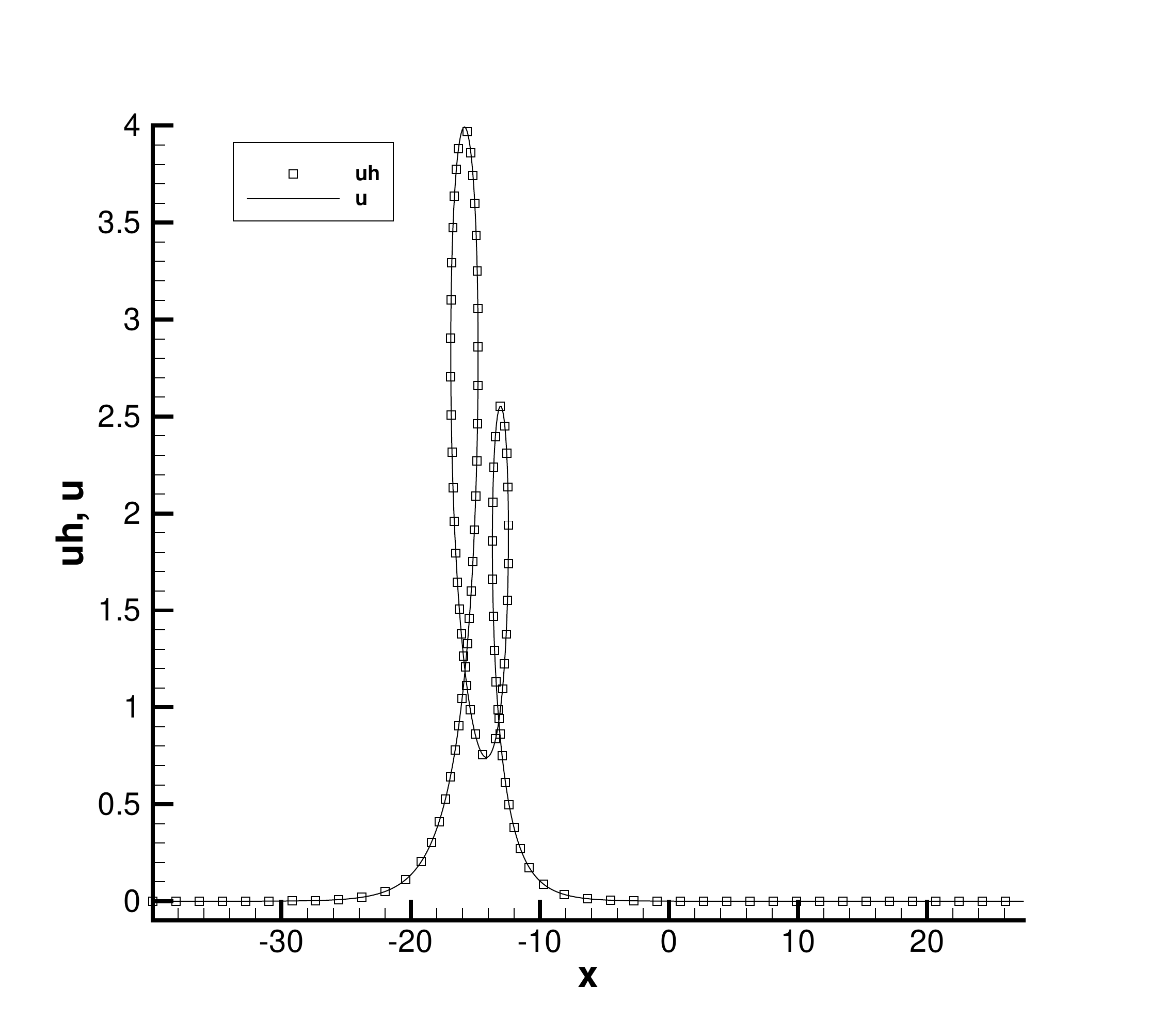}}
  \centerline{(c) t = 8.0}
\end{minipage}
\hfill
\begin{minipage}{0.49\linewidth}
  \centerline{\includegraphics[width=1.1\textwidth]{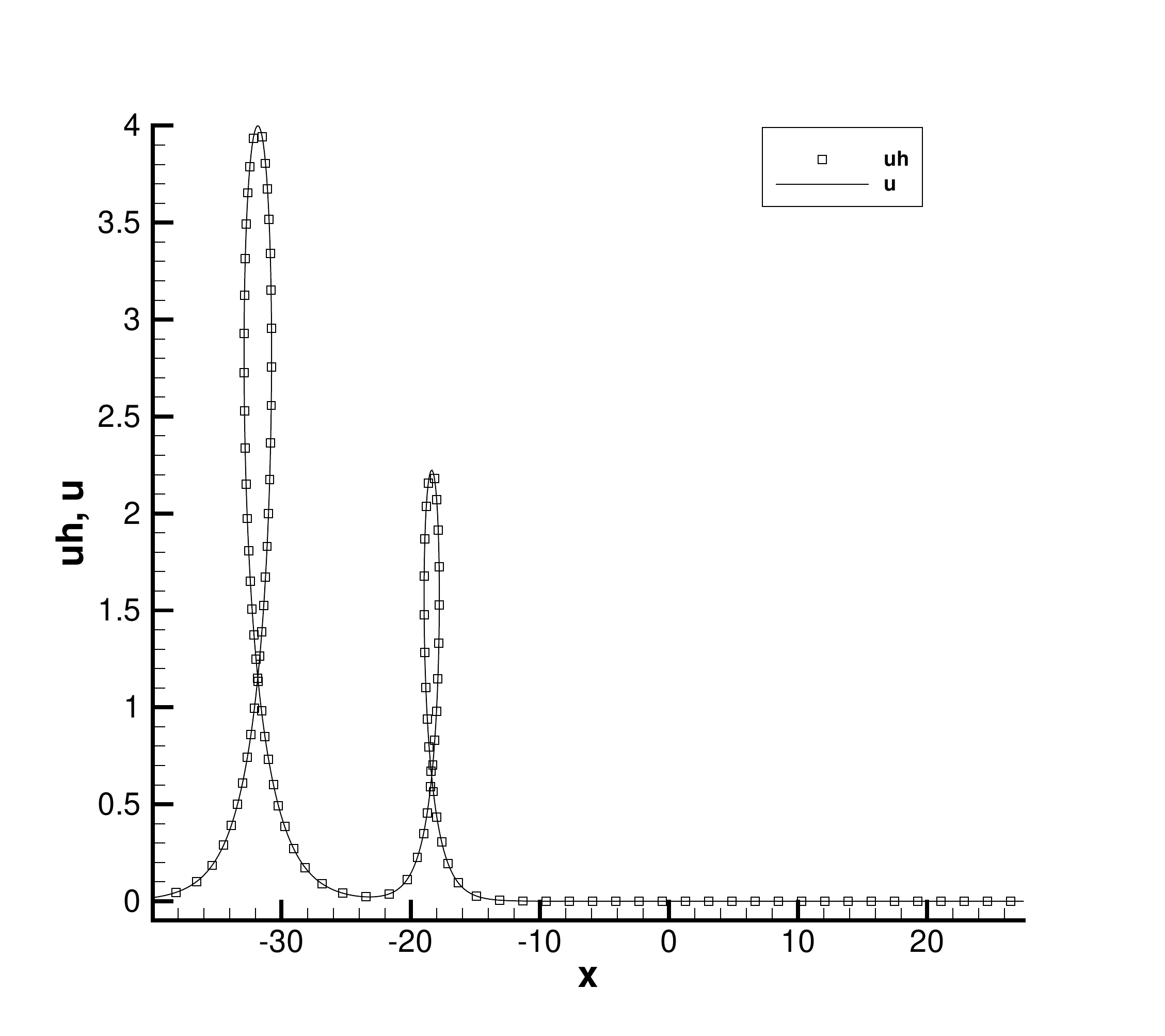}}
  \centerline{(d) t =12.0}
\end{minipage}
\caption{Two-loops-soliton solution of the short pulse equation \eqref{eqn:short pulse}: $H_0$ conserved DG scheme with $N = 160$ cells, $P^2$ elements. The parameters $\alpha_1$ = $ e^{-2} $, $\alpha_2$ = $ e^{-8} $, $p_1 = 0.9$, $p_2 = 0.5$. }
\label{fig:2loop_CD}
\end{figure}

\begin{figure}[!htp]
\begin{minipage}{0.49\linewidth}
  \centerline{\includegraphics[width=1.1\textwidth]{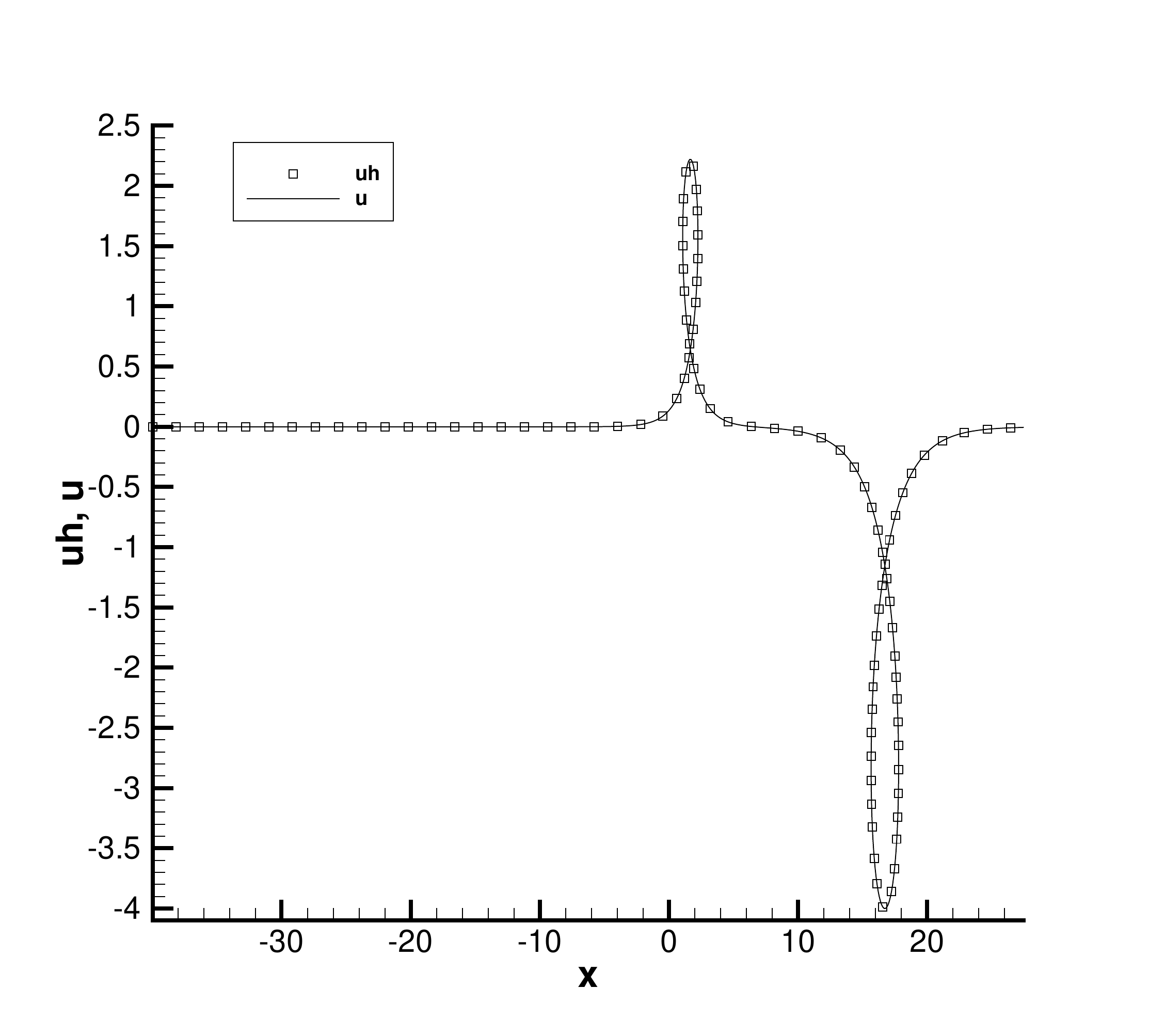}}
  \centerline{(a) t = 0.0}
\end{minipage}
\hfill
\begin{minipage}{0.49\linewidth}
  \centerline{\includegraphics[width=1.1\textwidth]{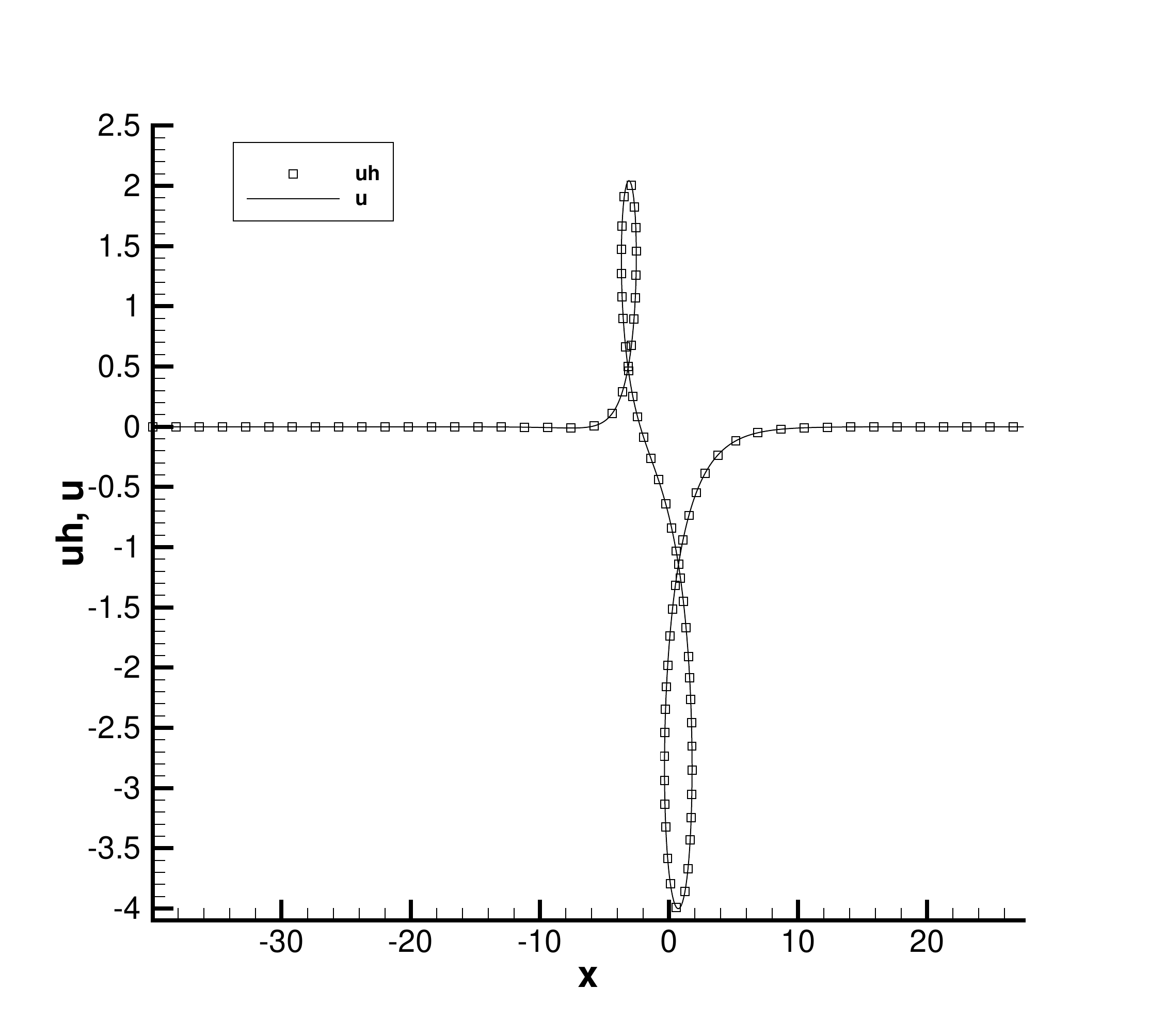}}
  \centerline{(b) t = 4.0}
\end{minipage}
\vfill
\begin{minipage}{0.49\linewidth}
  \centerline{\includegraphics[width=1.1\textwidth]{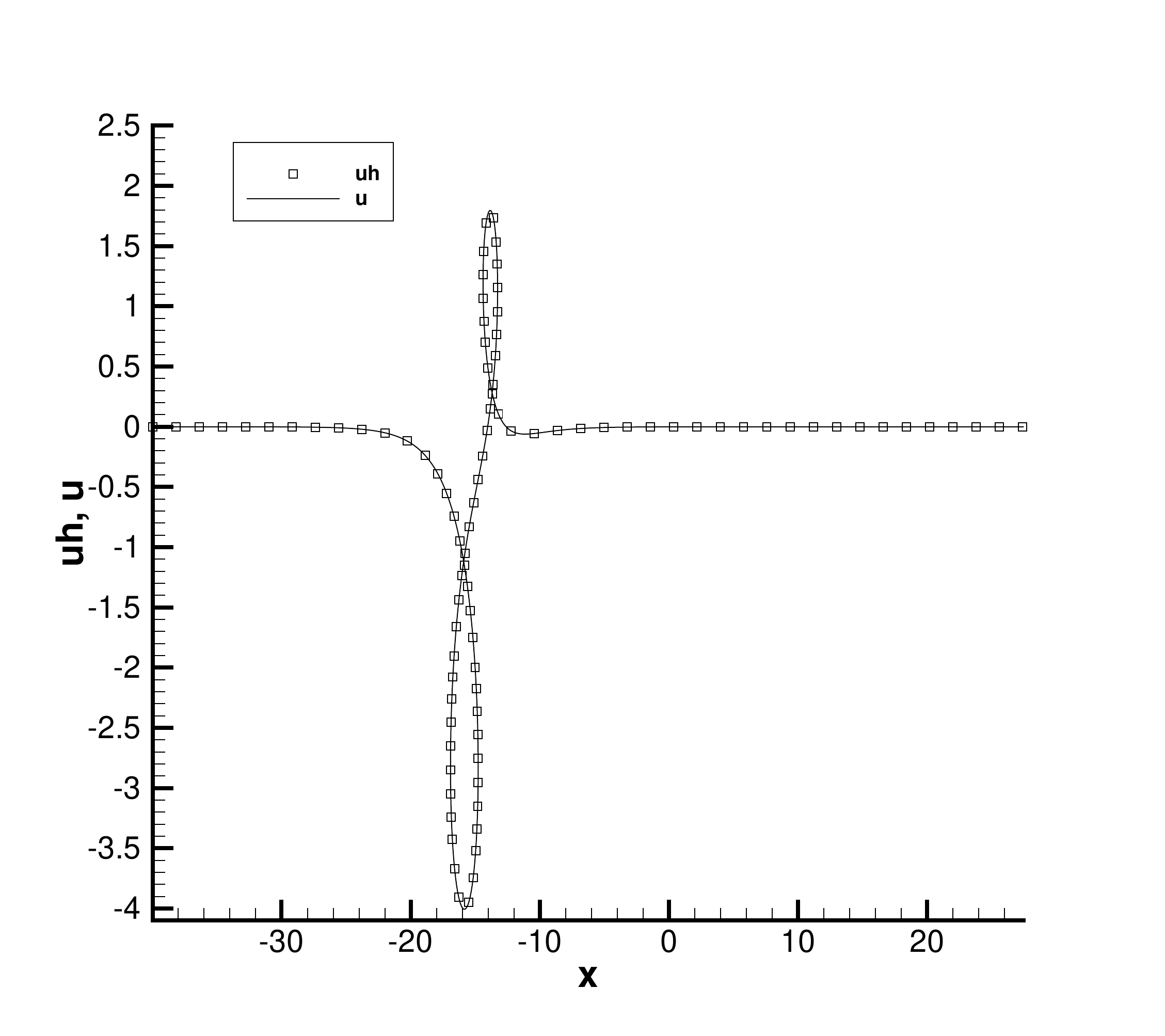}}
  \centerline{(c) t = 8.0}
\end{minipage}
\hfill
\begin{minipage}{0.49\linewidth}
  \centerline{\includegraphics[width=1.1\textwidth]{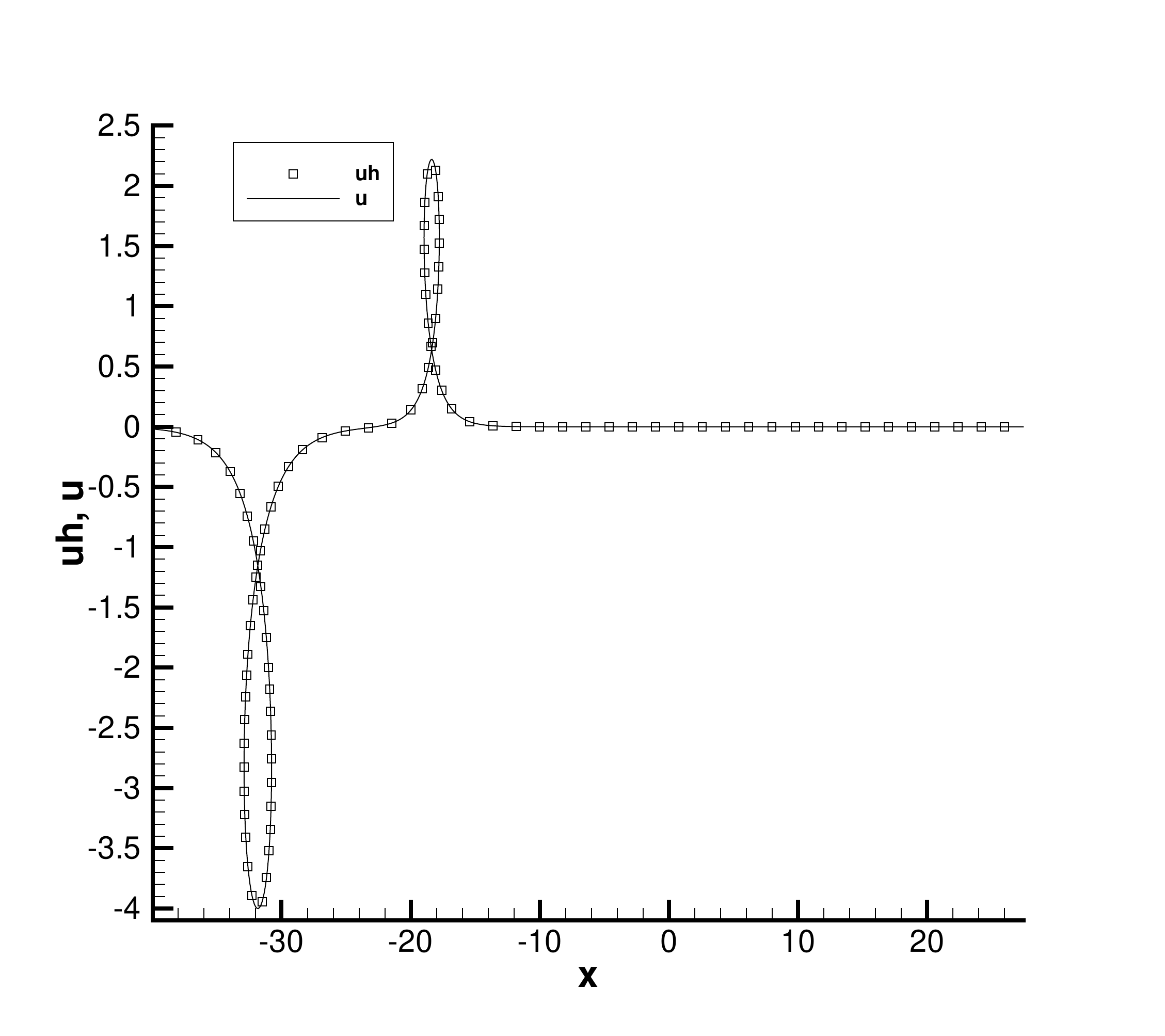}}
  \centerline{(d) t =12.0}
\end{minipage}
\caption{ Loop-antiloop-soliton solution of the short pulse equation \eqref{eqn:short pulse}: $H_1$ conserved DG scheme with $N = 160$ cells, $P^2$ elements. The parameters $\alpha_1$ = $ e^{-2} $, $\alpha_2$ = $ -e^{-8} $, $p_1 = 0.9$, $p_2 = 0.5$.}
\label{fig:antiloop_CD}
\end{figure}

\begin{figure}[!htp]
\begin{minipage}{0.49\linewidth}
  \centerline{\includegraphics[width=1.1\textwidth]{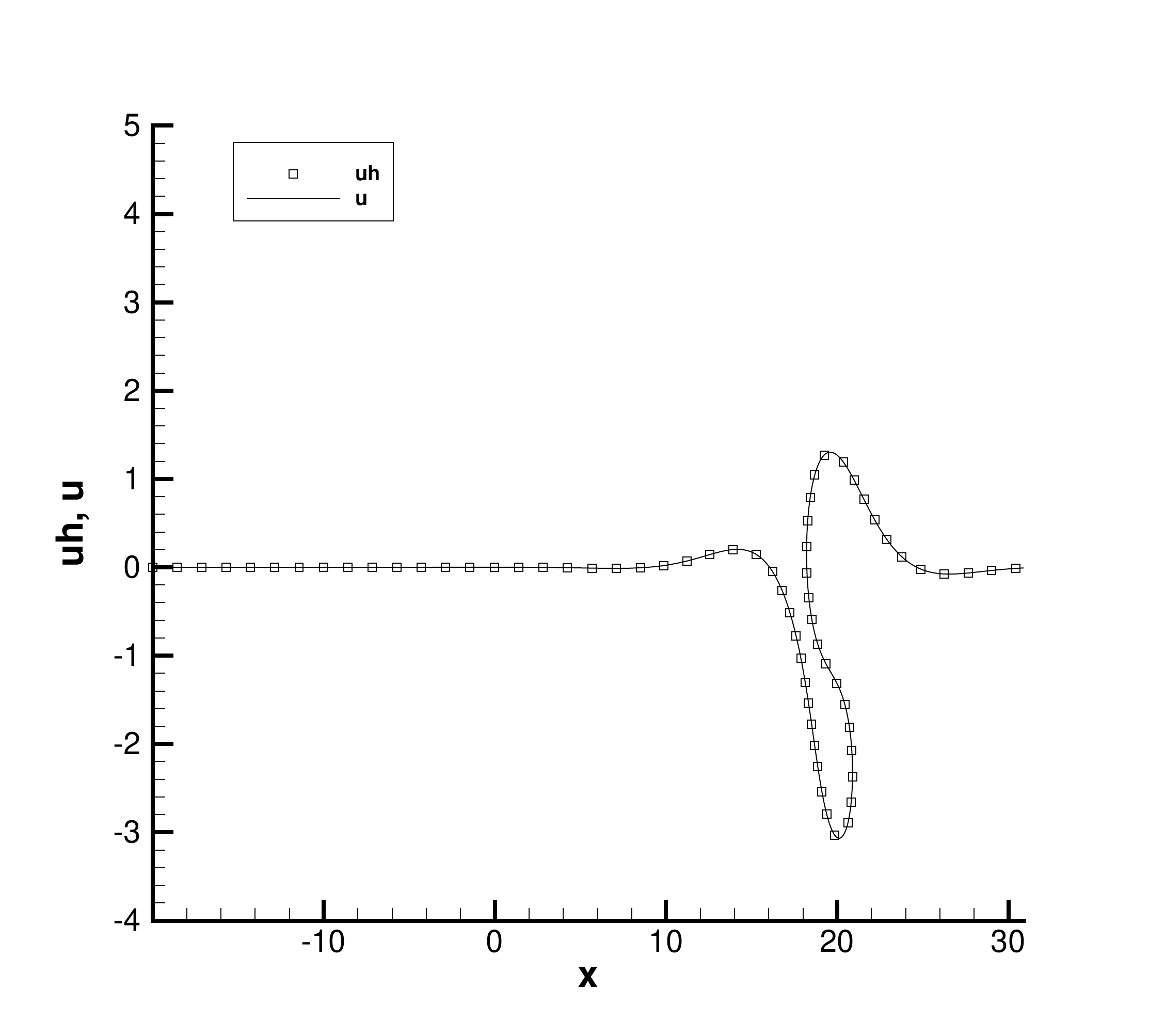}}
  \centerline{(a) t = 0.0}
\end{minipage}
\hfill
\begin{minipage}{0.49\linewidth}
  \centerline{\includegraphics[width=1.1\textwidth]{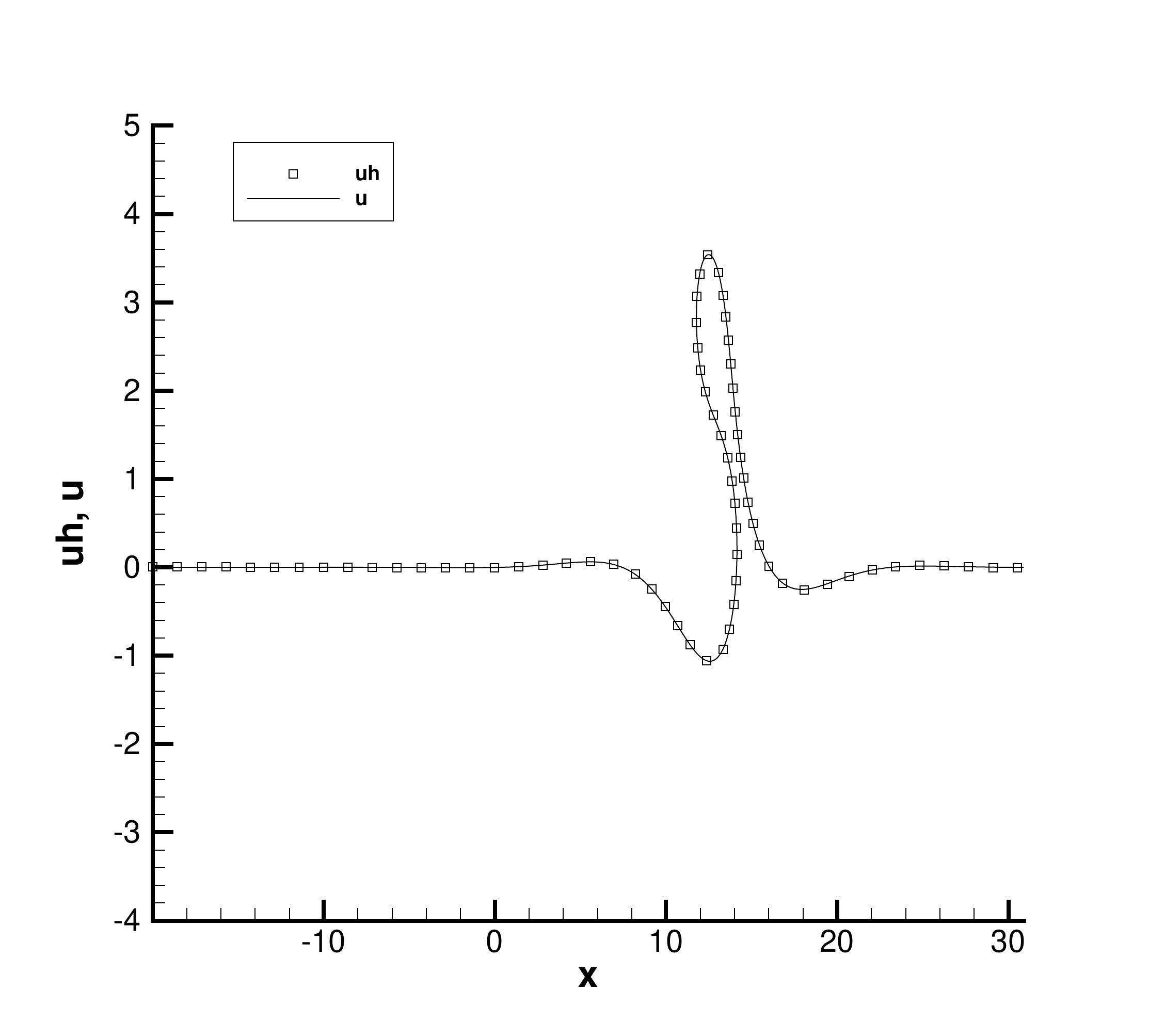}}
  \centerline{(b) t = 2.0}
\end{minipage}
\vfill
\begin{minipage}{0.49\linewidth}
  \centerline{\includegraphics[width=1.1\textwidth]{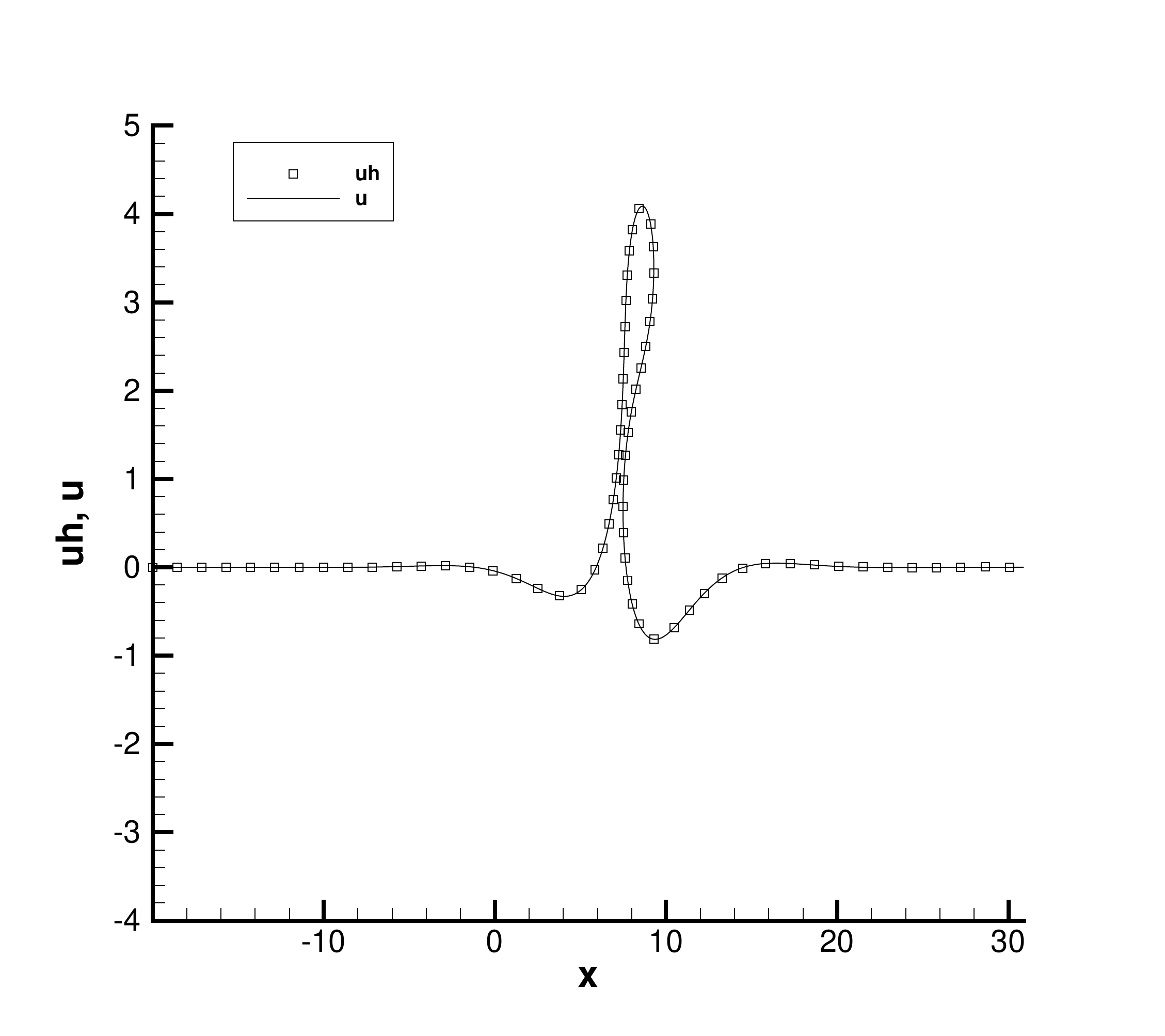}}
  \centerline{(c) t = 4.0}
\end{minipage}
\hfill
\begin{minipage}{0.49\linewidth}
  \centerline{\includegraphics[width=1.1\textwidth]{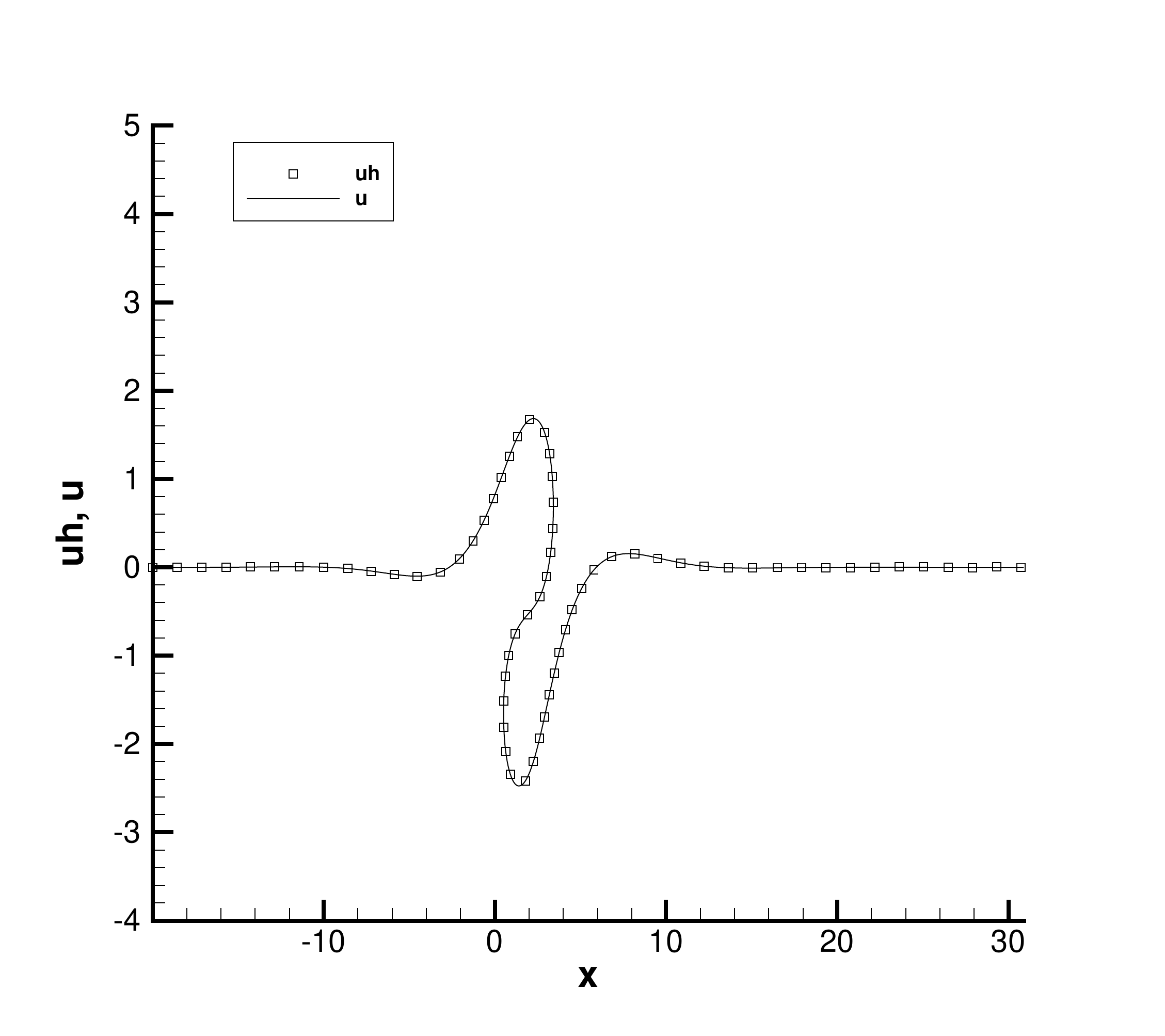}}
  \centerline{(d) t = 6.0}
\end{minipage}
\caption{ Breather solution of the short pulse equation \eqref{eqn:short pulse}: Integration DG scheme with $N = 160$ cells, $P^2$ elements. The parameters $\alpha_1$ = $ e^{-8}(1+i) $, $\alpha_2$ = $ e^{-8}(1-i) $, $p_1 = 0.4+0.44i$, $p_2 = 0.4-0.44i$. }
\label{fig:Breather_CD}
\end{figure}

\end{example}

\begin{example}
In this example, we test the DG schemes for the sine-Gordon equation
\begin{equation}\label{exeqn: SG}
z_{ys} = \sin{z}.
\end{equation}
with the 1-soliton solution:
\begin{equation}\label{eqn:loop_SG}
z(y,s) = 4\arctan(\exp(y+s)).
\end{equation}
In Table \ref{tab:SG}, we give the errors and convergence rates of two DG schemes. The DG scheme \eqref{eqn:numerical_SG} can reach optimal $(k+1)$-$th$ order of accuracy. And the integration DG scheme \eqref{eqn:numerical_SG_integration} can also achieve the optimal order of accuracy, which is  $(k+2)$-$th$ order for $u_h$.
\begin{table}[H]
\begin{center}
\begin{tabular}{|c|c|cccc|cccc|}
  \hline
           &         &  \multicolumn{4}{c|}{DG scheme \eqref{eqn:numerical_SG}} & \multicolumn{4}{c|}{integration DG scheme \eqref{eqn:numerical_SG_integration}}\\\hline
           & N       &$ \norm{z-z_h}_{L^2}$    & order   &$ \norm{z-z_h}_{L^\infty}$   & order
                     &$ \norm{z-z_h}_{L^2}$    & order   &$ \norm{z-z_h}_{L^\infty}$   & order\\\hline

    $P^1$      &  40  &1.95E-03 &-- &2.56E-02 &--
                      &1.53E-04 &-- &1.65E-03 &--\\

               &  80  &4.90E-04 &1.99 &7.17E-03 &1.83
                      &1.94E-05 &2.99 &2.28E-04 &2.85\\

              & 160  &1.23E-04 &2.00 &1.84E-03 &1.96
                     &2.43E-06 &3.00 &2.98E-05 &2.94\\

              & 320  &3.07E-05 &2.00 &4.64E-04 &1.99
                     &3.04E-07 &3.00 &3.78E-06 &2.98\\\hline


      $P^2$ &  40  &1.41E-04 &3.21 &1.32E-03 &3.09
                   &7.91E-06 &4.22 &7.96E-05 &4.15\\

            &  80  &1.77E-05 &2.99 &1.84E-04 &2.84
                   &4.99E-07 &3.99 &5.61E-06 &3.83\\

            & 160  &2.22E-06 &3.00 &2.37E-05 &2.96
                   &3.13E-08 &4.00 &3.62E-07 &3.95\\

            & 320  &2.77E-07 &3.00 &2.99E-06 &2.99
                   &1.96E-09 &4.00 &2.28E-08 &3.99\\\hline


\end{tabular}
\end{center}
\caption{\label{tab:SG}1-soliton solution \eqref{eqn:loop_SG} of the sine-Gordon equation \eqref{exeqn: SG}: The computational domain $[-5,5]$, at time $T = 1$.} 
\end{table}
We define
\begin{equation}
\Delta H_2 = \sum\limits_{j=1}^{N+1}\Big( \int_{I_j}\omega_h^2\Big{|}_{t=T} \ dy -\int_{I_j} \omega_0^2 \ dy\Big)
\end{equation}
where $\omega_h$ is defined in \eqref{eqn:numerical_SG_3}. Compared with DG scheme \eqref{eqn:numerical_SG}, the fluctuation of $H_2$ in integration DG scheme is very slight, as shown in Table \ref{ex:SG_change} .
\begin{table}[H]
\begin{tabular}{cc|cc}
\hline
      $\Delta H_2$            && DG scheme\eqref{eqn:numerical_SG} & integration DG scheme\eqref{eqn:numerical_SG_integration}    \\\hline
     $P^1$        &&6.39E-03  &   6.85E-07               \\
     $P^2$        &&7.12E-06  &   4.93E-10	                 \\
     $P^3$        && 7.75E-09 &   4.90E-10	                 \\\hline

\end{tabular}
\caption{\label{ex:SG_change} The time evolution of conserved quantity $H_2$ for 1-soliton solution \eqref{eqn:loop_SG} for the sine-Gordon equation \eqref{exeqn: SG}, with the computational domain $[-30,30]$ and $N = 320$ cells at time $ T = 10$.}
\end{table}

In \cite{Sakovich_2006_JPA}, two-loops-soliton, loop-antiloop-soliton and breather solutions can be found for the short pulse equation \eqref{eqn:short pulse}.  We use these solutions to validate our numerical schemes.
We plot figures of the antiloop-loop and breather solutions for the short pulse equation \eqref{eqn:short pulse} in Figure \ref{fig:antiloop-loop_SG}, \ref{fig:Breather_SG}. For those two cases, the initial data of the sine-Gordon equation is required to be continuous. Thus we can get the temporal and spatial derivatives $z_s, z_y$ of $z$. The initial condition for the two-loops-soliton solution is discontinuous, so we give the initial derivatives piecewisely.   The integration DG scheme  obtains the accurate solution as shown in Figure \ref{fig:2loop_SG}. These three kinds of solutions can be resolved very well compared with the exact solutions in \cite{Sakovich_2006_JPA}.
\begin{figure}[H]
\begin{center}
\begin{tabular}{ccc}
\includegraphics[width=0.33\textwidth]{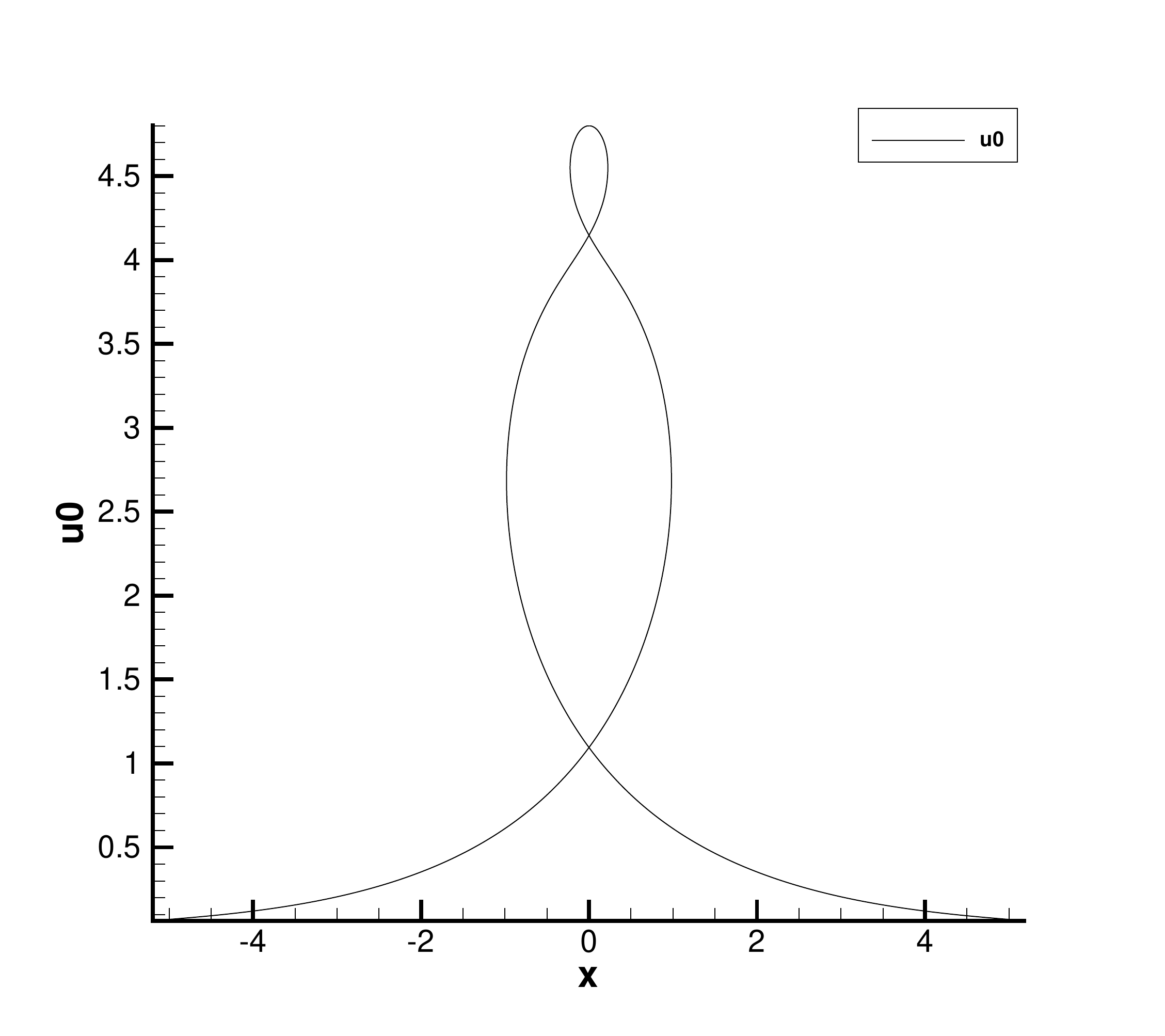}&\includegraphics[width=0.33\textwidth]{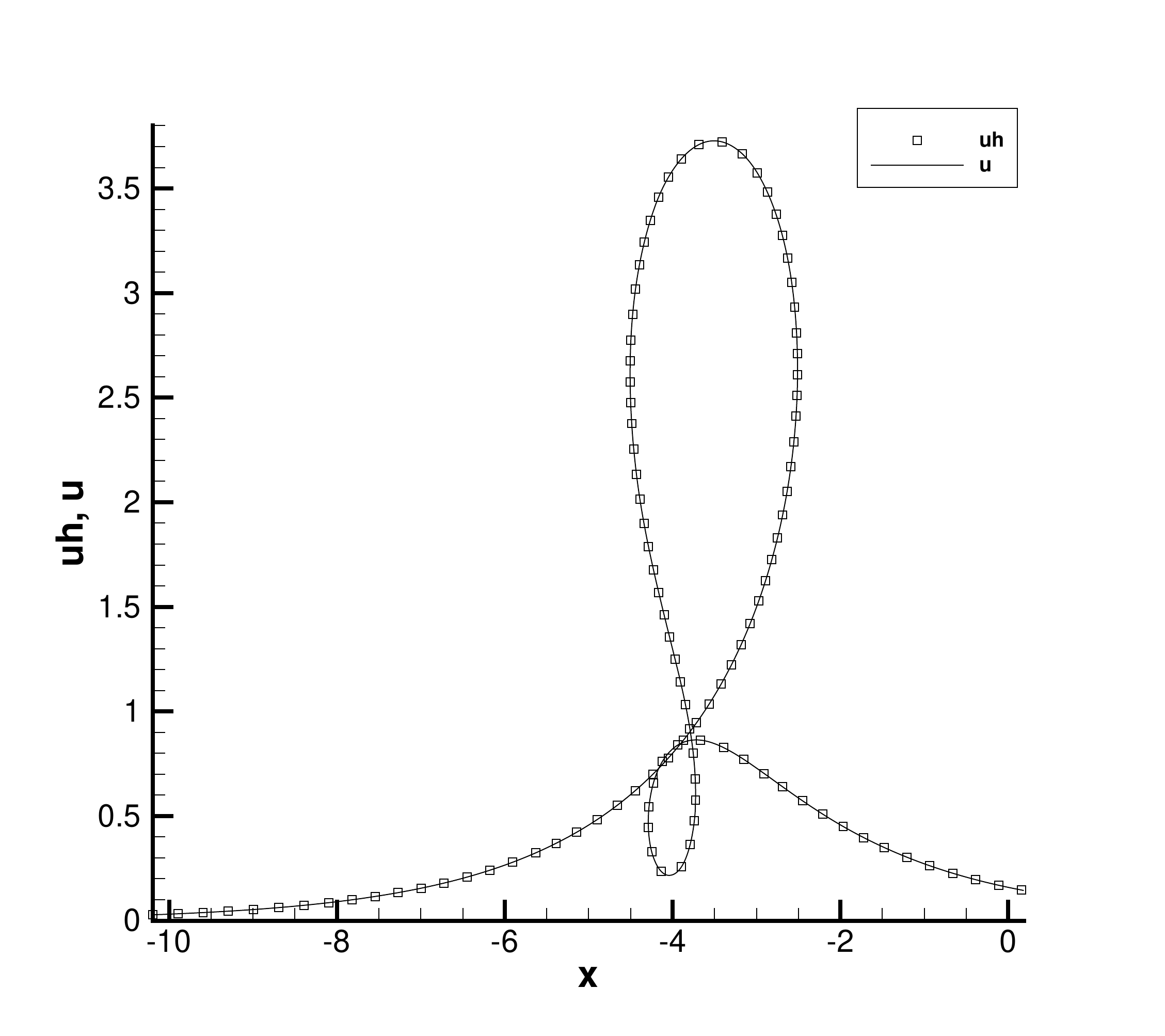}&\includegraphics[width=0.33\textwidth]{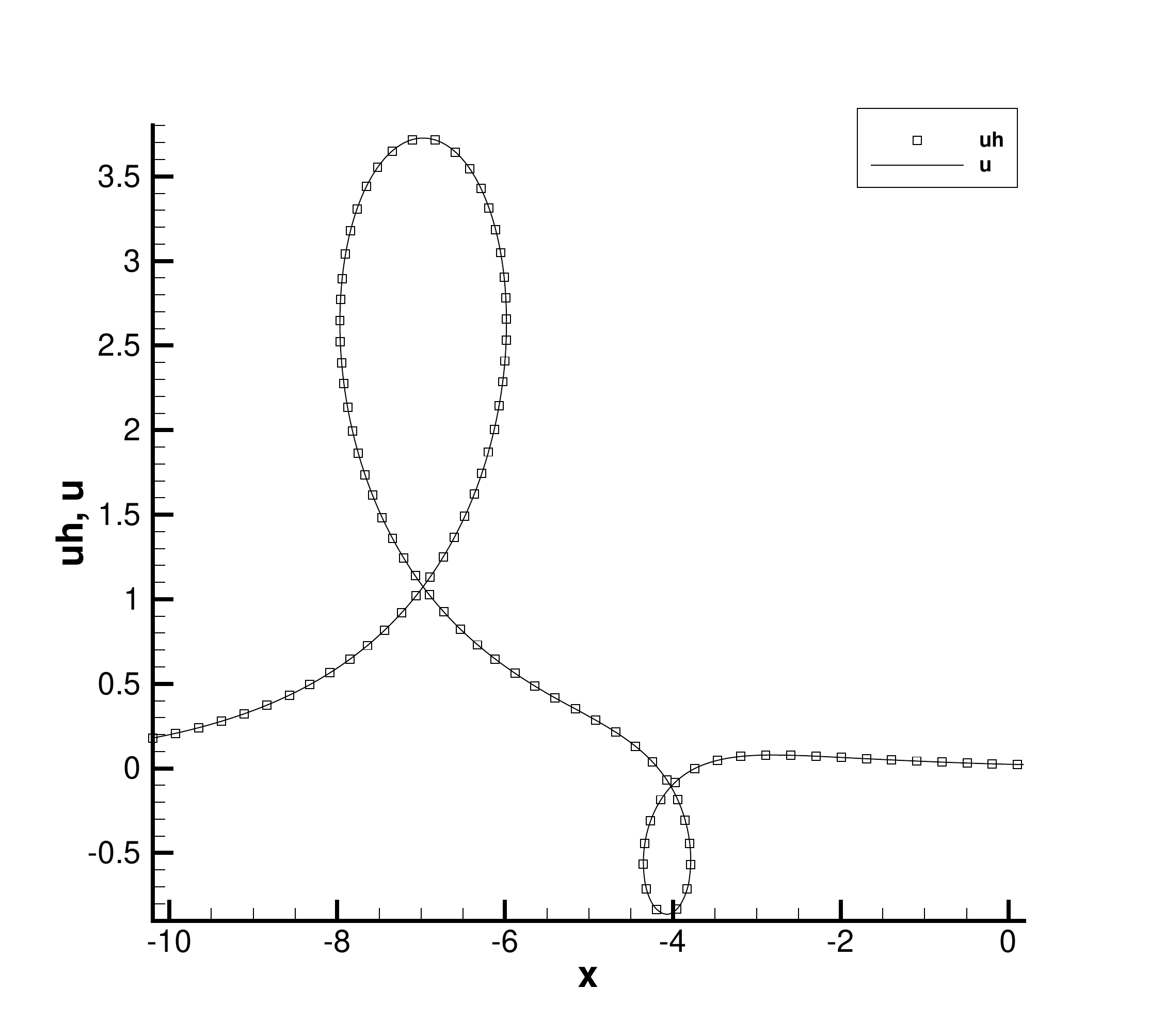}\\
(a) t = 0.0 & (b) t = 1.0  &(c) t = 2.0
\end{tabular}
\end{center}
\caption{\label{fig:antiloop-loop_SG} Antiloop-loop-soliton solution of the short pulse equation \eqref{eqn:short pulse}: Integration DG scheme with $N = 160 $ cells, $P^2$ elements.}
\end{figure}

\begin{figure}[H]
\begin{center}
\begin{tabular}{ccc}
\includegraphics[width=0.33\textwidth]{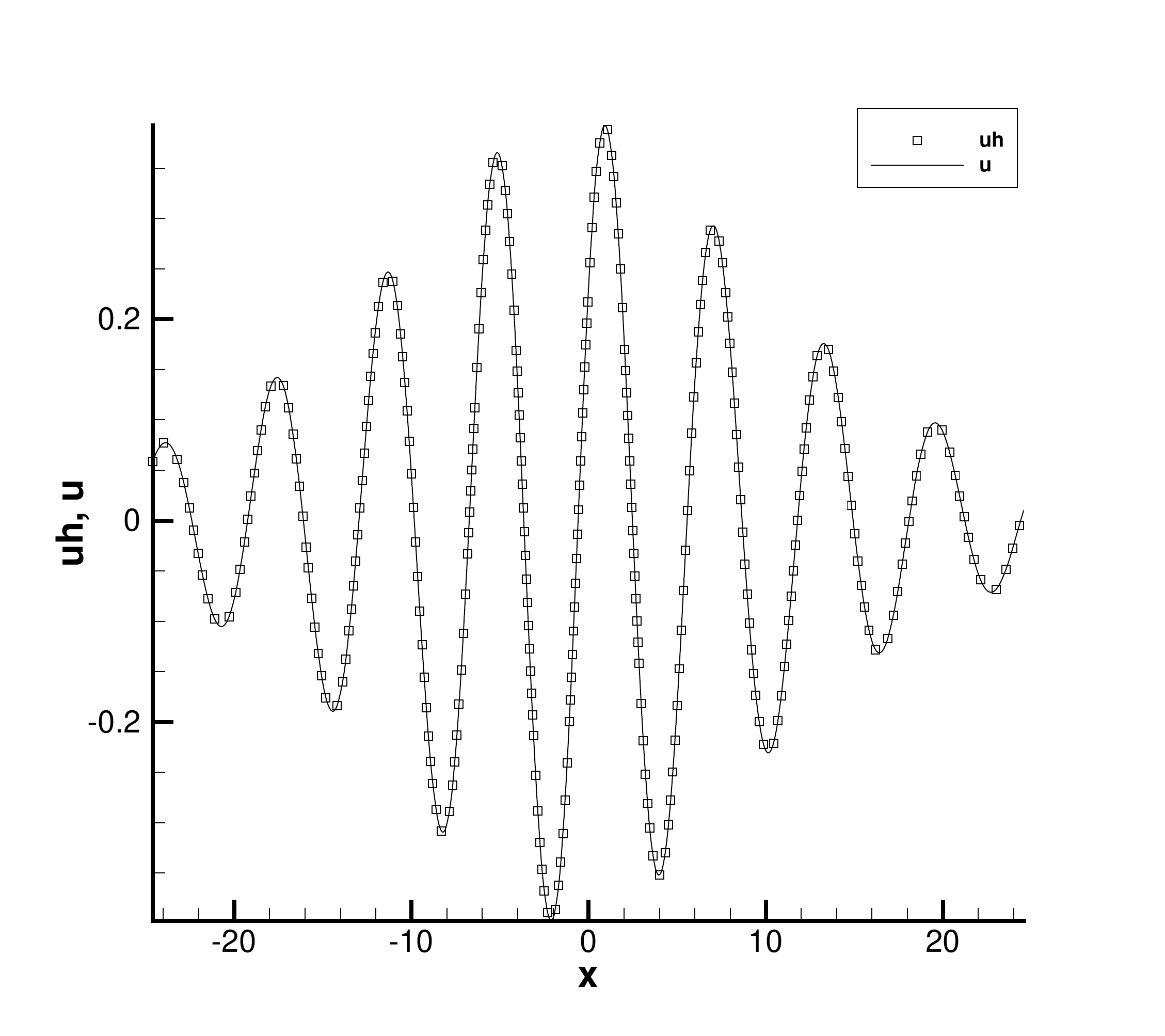}&\includegraphics[width=0.33\textwidth]{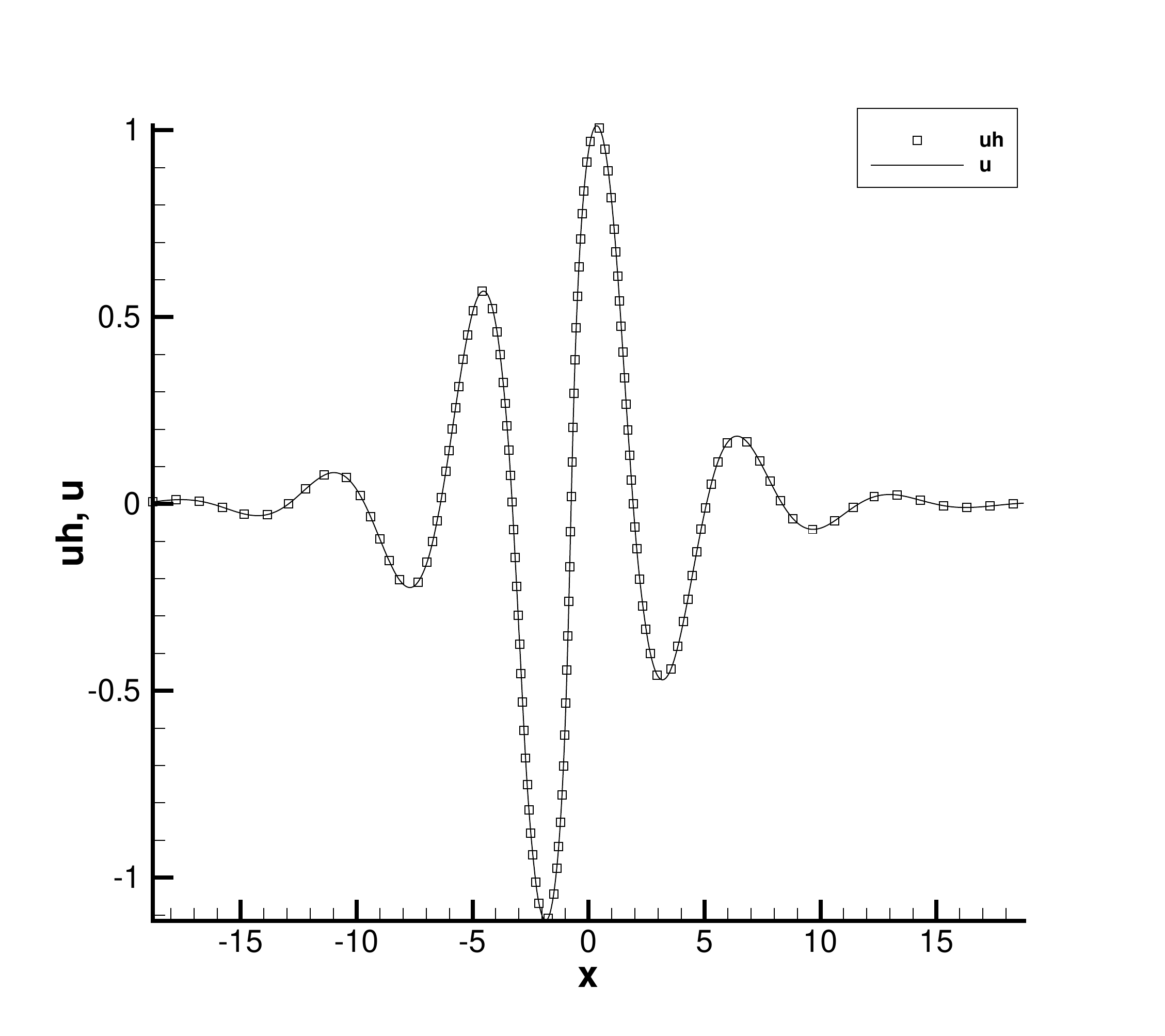}&\includegraphics[width=0.33\textwidth]{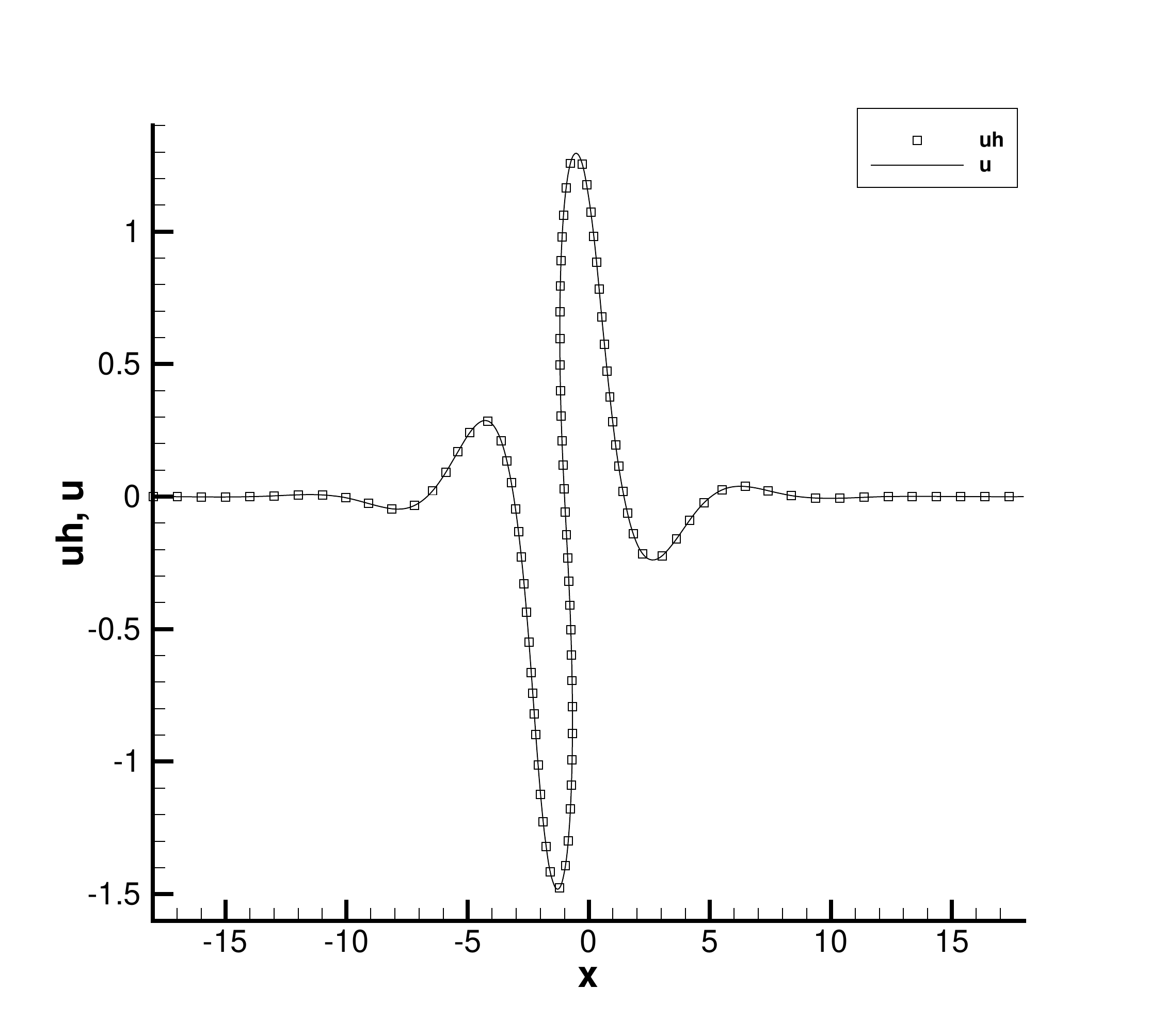}\\
\end{tabular}
\end{center}
\caption{\label{fig:Breather_SG} Breather solution of the short pulse equation \eqref{eqn:short pulse}: DG scheme \eqref{eqn:numerical_SG} with $N = 160 $ cells, $P^2$ elements at time $T = 1$.    }
\end{figure}
\begin{figure}[htp]
\begin{center}
\begin{tabular}{ccc}
\includegraphics[width=0.33\textwidth]{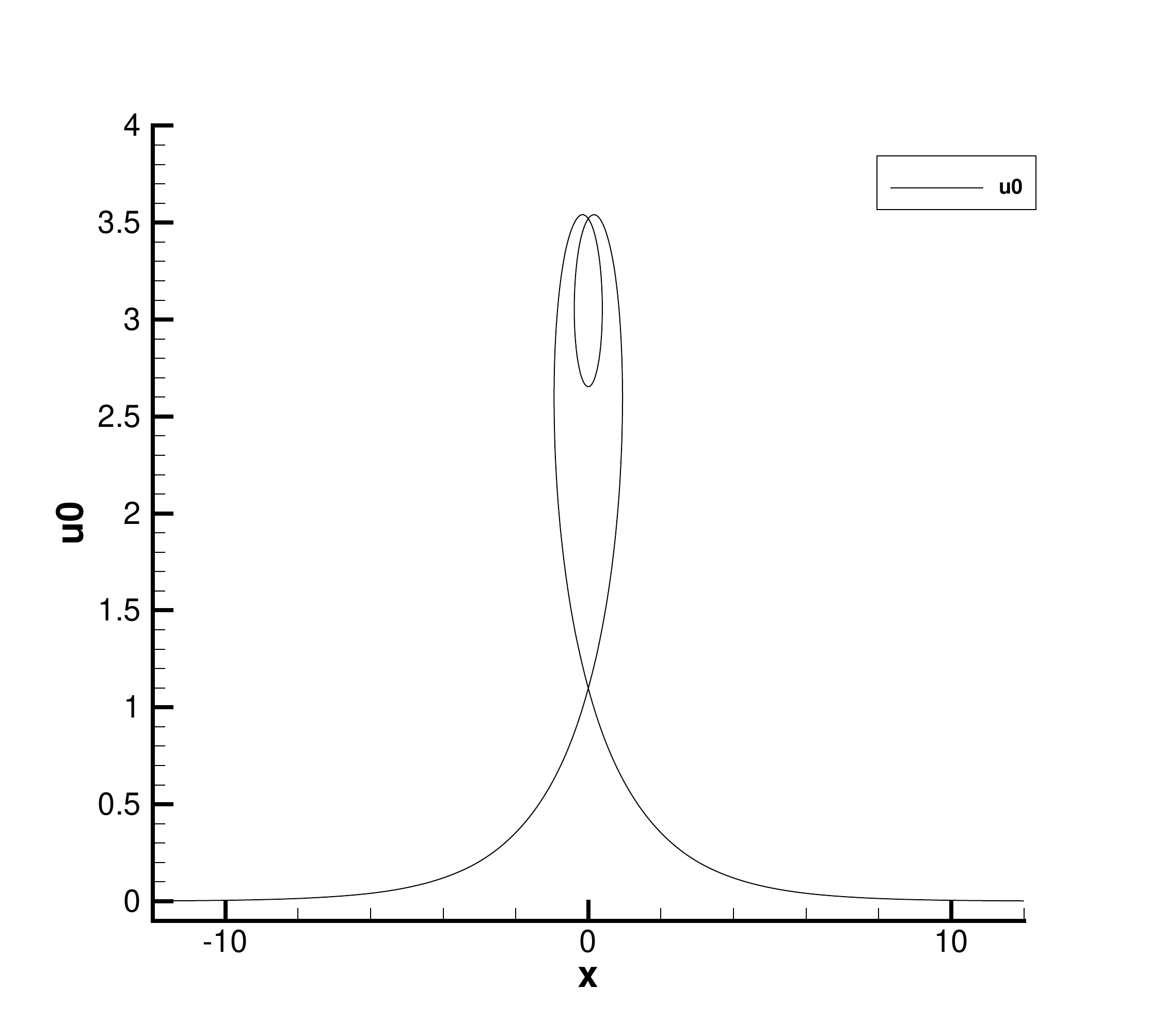}&\includegraphics[width=0.33\textwidth]{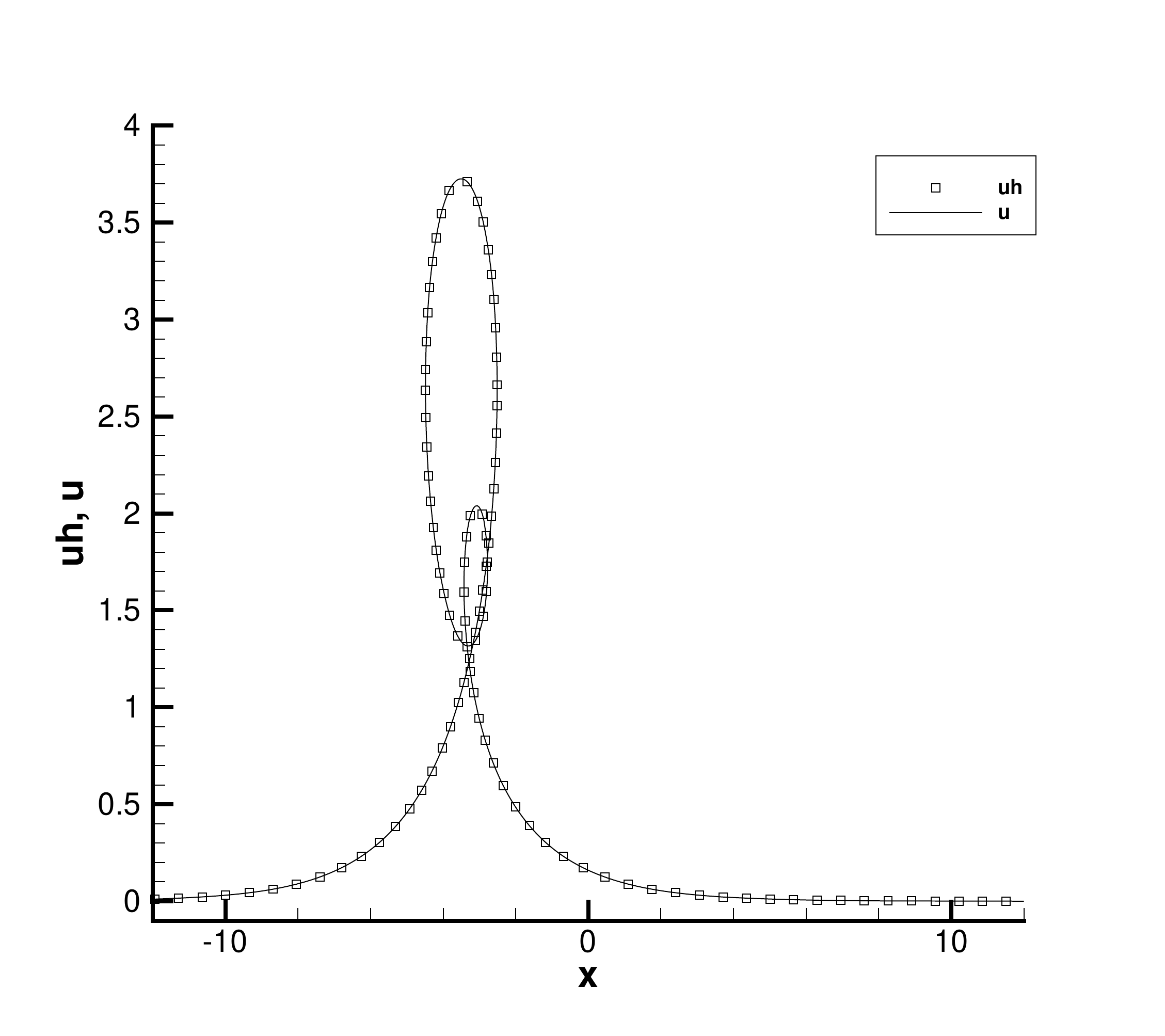}&\includegraphics[width=0.33\textwidth]{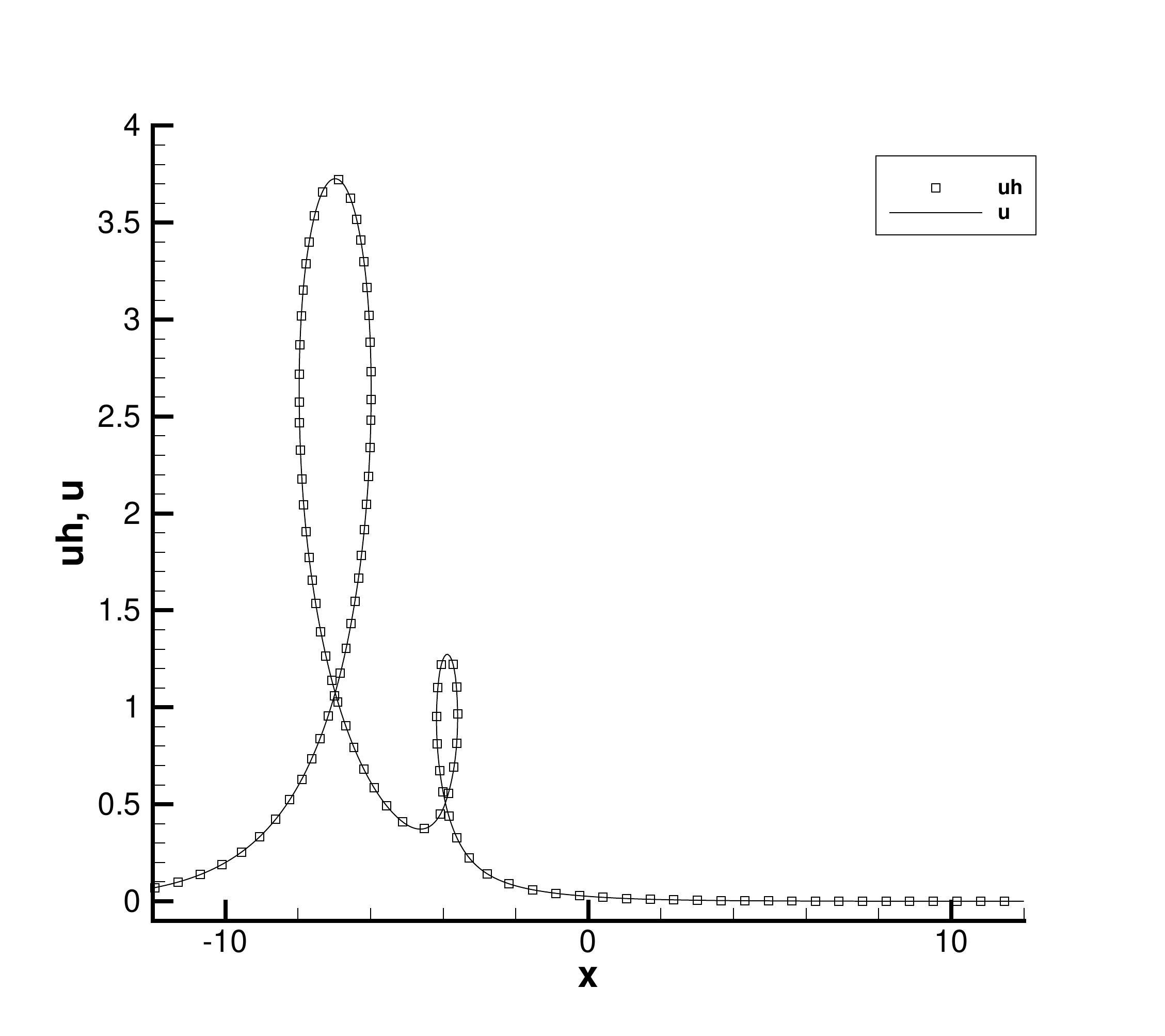}\\
(a) t = 0.0 & (b) t = 1.0  &(c) t = 2.0
\end{tabular}
\end{center}
\caption{\label{fig:2loop_SG} Two-loops-soliton solution of the short pulse equation \eqref{eqn:short pulse}: Integration DG scheme with $N = 160 $ cells, $P^2$ elements.}
\end{figure}
\end{example}

\begin{example}\label{ex2}
In this example, we consider the complex short pulse equation
\begin{equation}\label{eqn: complex_SPE1_ex}
\begin{split}
 &u_{xt} = u + \frac{1}{2}(\abs{u}^2u_x)_x,\ u\in \mathbb{C}
\end{split}
\end{equation}
 which is linked with the complex CD system \eqref{eqn:ComplexCD1}. The exact solution can be also expressed as the determinant form \eqref{soliton_solution}, the
elements of $A_I$ and $B_I$ are distinct from Example \ref{ex1}:
\begin{equation}\label{matrix_element}
\begin{split}
&a_{ij} = \frac{1}{2(\frac{1}{p_i}+ \frac{1}{p_j^*})}e^{\xi_i + \xi^*_j}, \ b_{ij} = \frac{\alpha_i^*\alpha_j}{2(\frac{1}{p^*_i}+ \frac{1}{p_j})}. \\
&\text{with} \ \xi_i = p_iy + \frac{s}{p_i} + y_{i0}, \ \xi^*_i = p_i^*y + \frac{s}{p_i^*} + y^*_{i0},\ i = 1, 2, \ldots, m,
\end{split}
\end{equation}
where $*$ denotes the complex conjugate, and $p_i$, $\alpha_i $, $y_{i0}$ $\in \mathbb{C}$ are constants. The shape of solution depends on the choice of parameters $p_i$. For 2-soliton solution, if $p_1 = p_2 \in \mathbb{C}$, then it degenerates to 1-soliton solution. Otherwise, it has two solitons. We first take the parameters as $p_1 = 0.5 + i$, $p_2 = 0.8 + 2i$, i.e. two smooth-soliton, so called breather solution. We plot the moduli $\abs{u}$ and $\abs{u_h}$ at $T = 0, 25,35,60$ in Figure \ref{fig:2-soliton_complex_CD_abs}. The computational domain is $[-50,50]$ with cells $N = 320$. In Figure \ref{fig:2-soliton_complex_CD_abs2}, an interaction of loop-soliton and cuspon-soliton is shown with meshes in $[-20,20]$ and parameters $p_1 = 0.9 + 0.5i$, $p_2 = 2.0 + 2i$. We can see clearly that the moving soliton interaction is resolved very well comparing with \cite{Feng_2014_PJMI}.
\begin{figure}[htp]
\begin{minipage}{0.49\linewidth}
  \centerline{\includegraphics[width=1.1\textwidth]{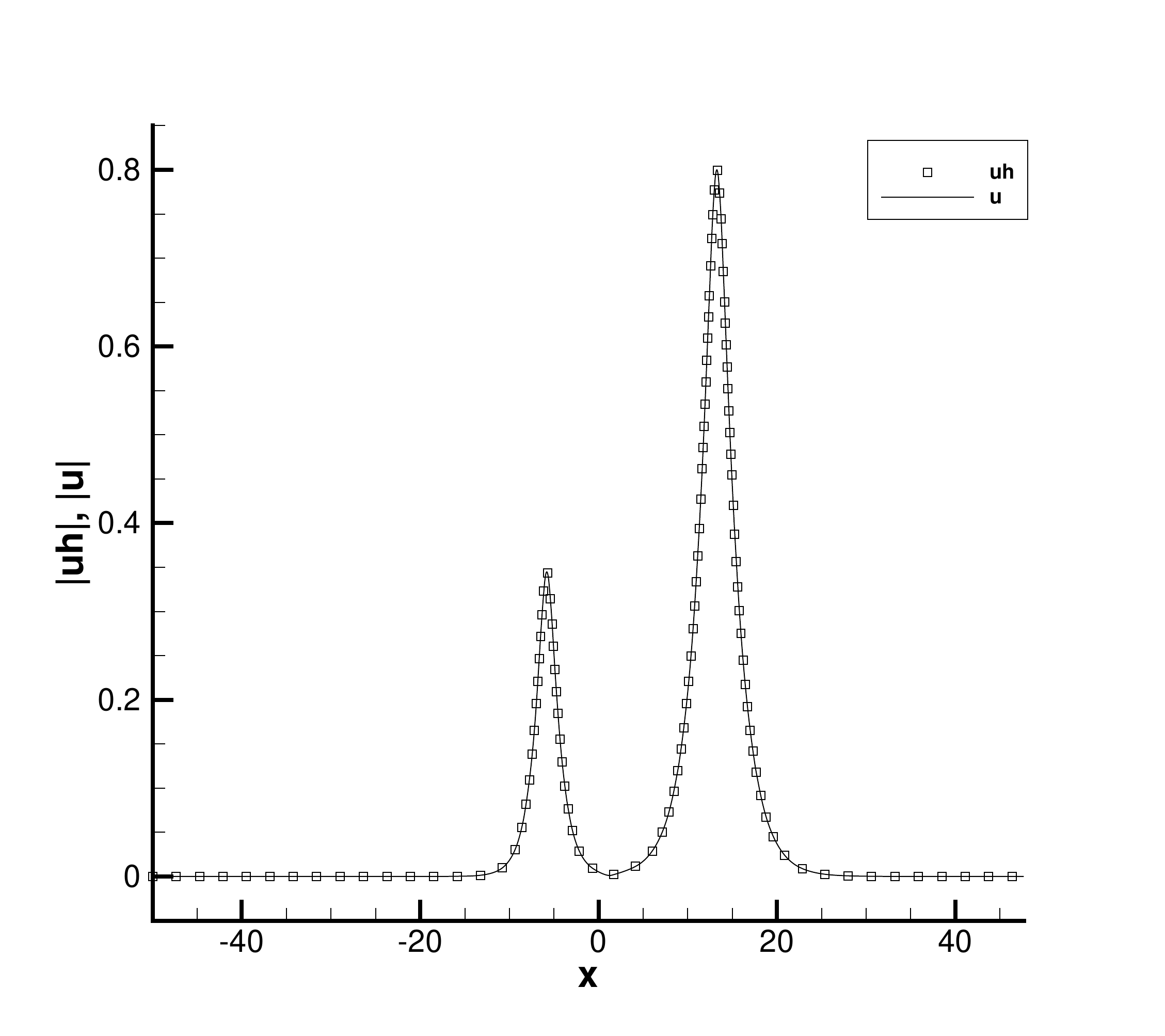}}
  \centerline{(a) t = 0.0}
\end{minipage}
\hfill
\begin{minipage}{0.49\linewidth}
  \centerline{\includegraphics[width=1.1\textwidth]{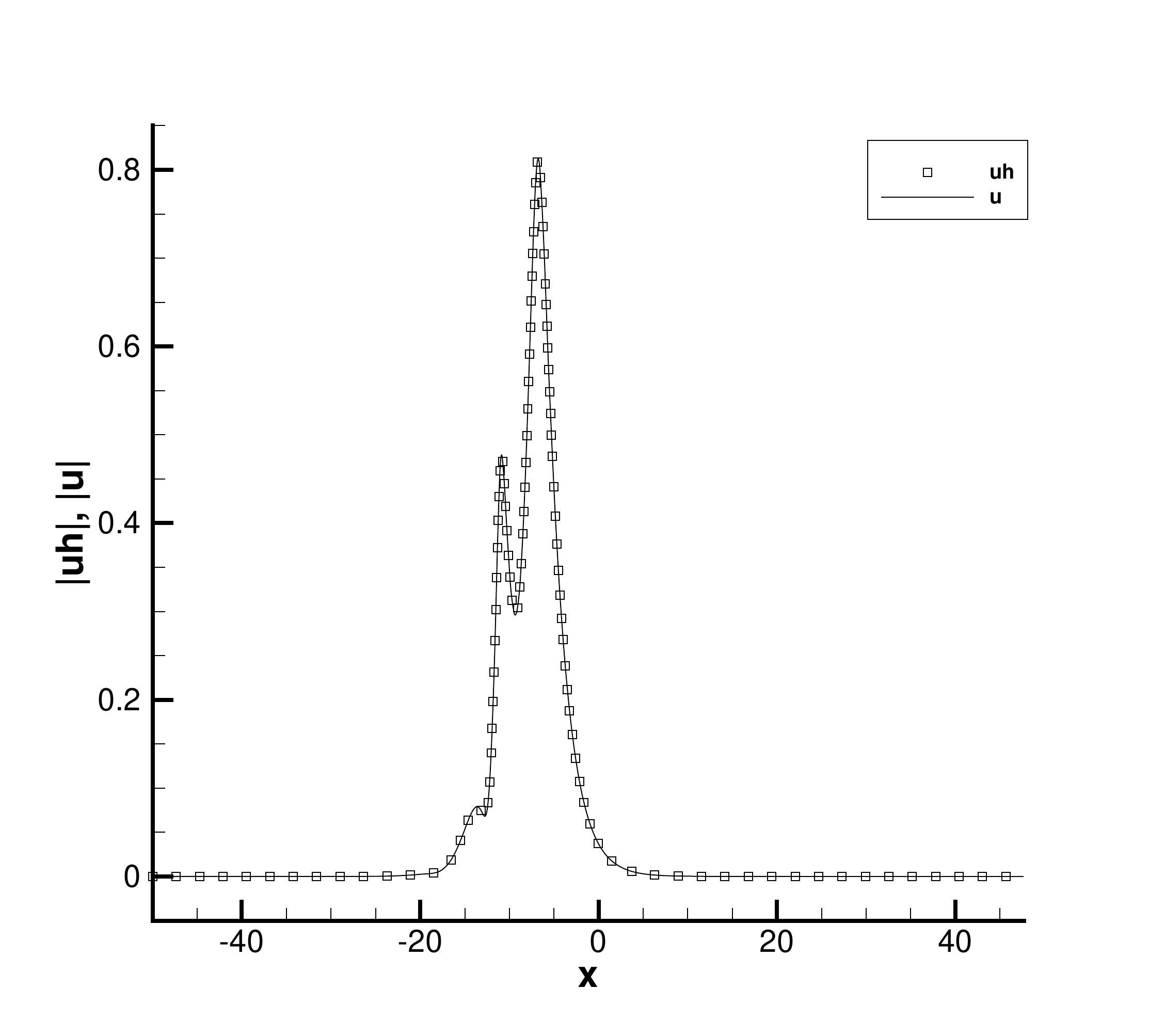}}
  \centerline{(b) t = 25.0}
\end{minipage}
\vfill
\begin{minipage}{0.49\linewidth}
  \centerline{\includegraphics[width=1.1\textwidth]{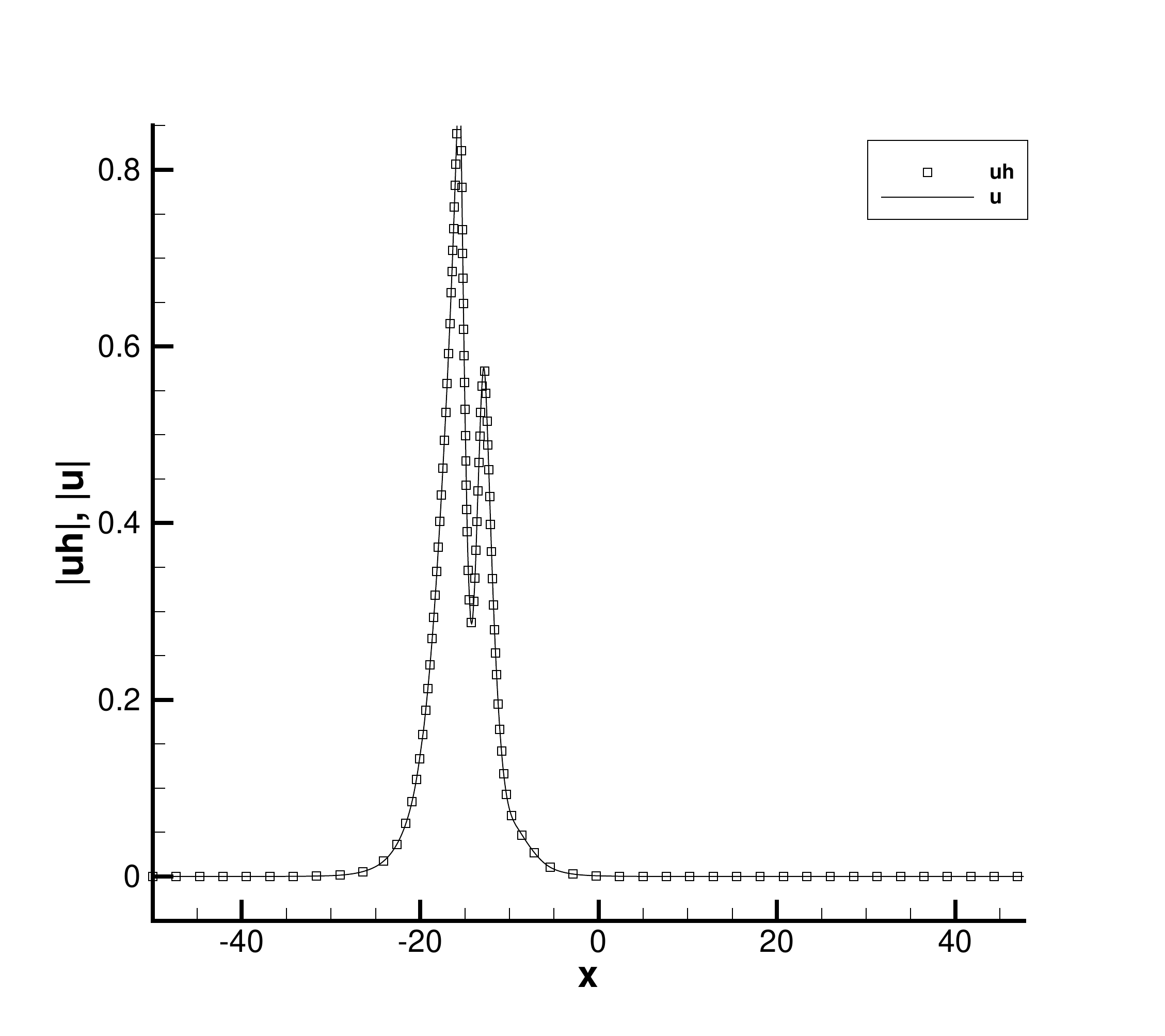}}
  \centerline{(c) t = 35.0}
\end{minipage}
\hfill
\begin{minipage}{0.49\linewidth}
  \centerline{\includegraphics[width=1.1\textwidth]{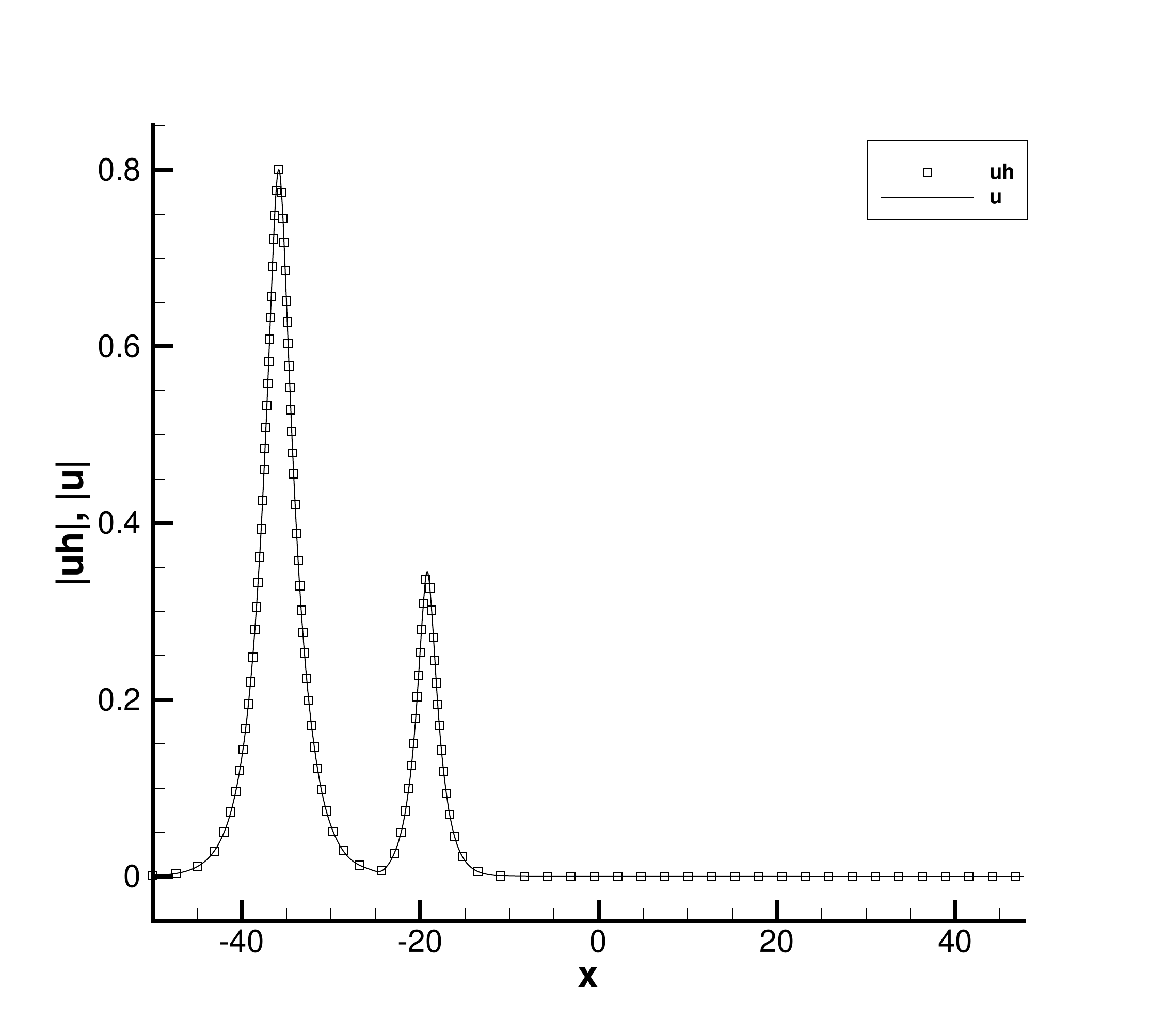}}
  \centerline{(d) t = 60.0}
\end{minipage}
\caption{Breather solution for the complex short pulse equation \eqref{eqn: complex_SPE1_ex}: $H_1$ conserved DG scheme with $N = 320$ cells, $P^2$ elements. The parameters $\alpha_1 =  e^{-6}$, $\alpha_2 = e^{4} $, $p_1 = 0.5 + i$, $p_2 = 0.8 + 2i$. }
\label{fig:2-soliton_complex_CD_abs}
\end{figure}

\begin{figure}[htp]
\begin{minipage}{0.49\linewidth}
  \centerline{\includegraphics[width=1.1\textwidth]{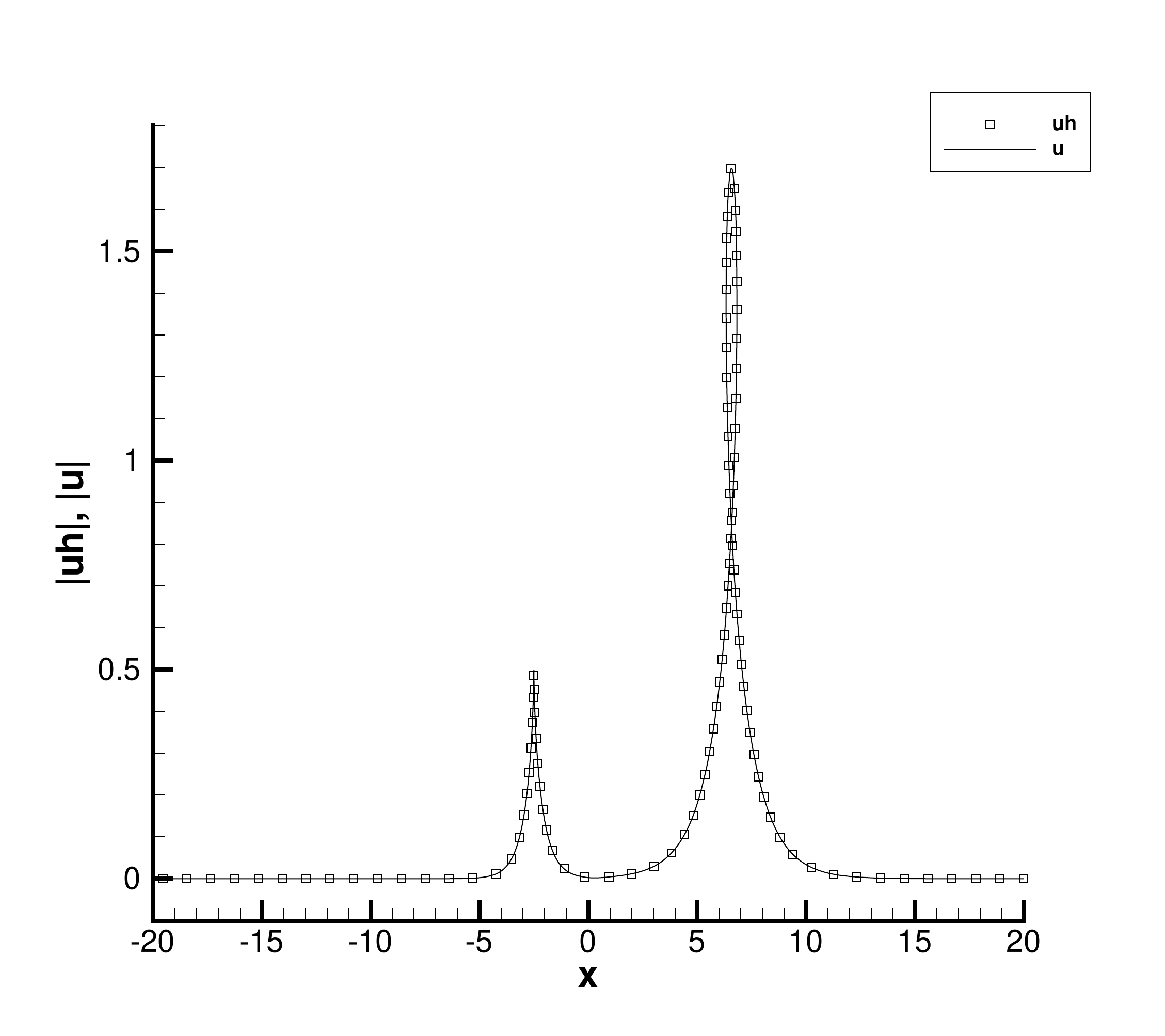}}
  \centerline{(a) t = 0.0}
\end{minipage}
\hfill
\begin{minipage}{0.49\linewidth}
  \centerline{\includegraphics[width=1.1\textwidth]{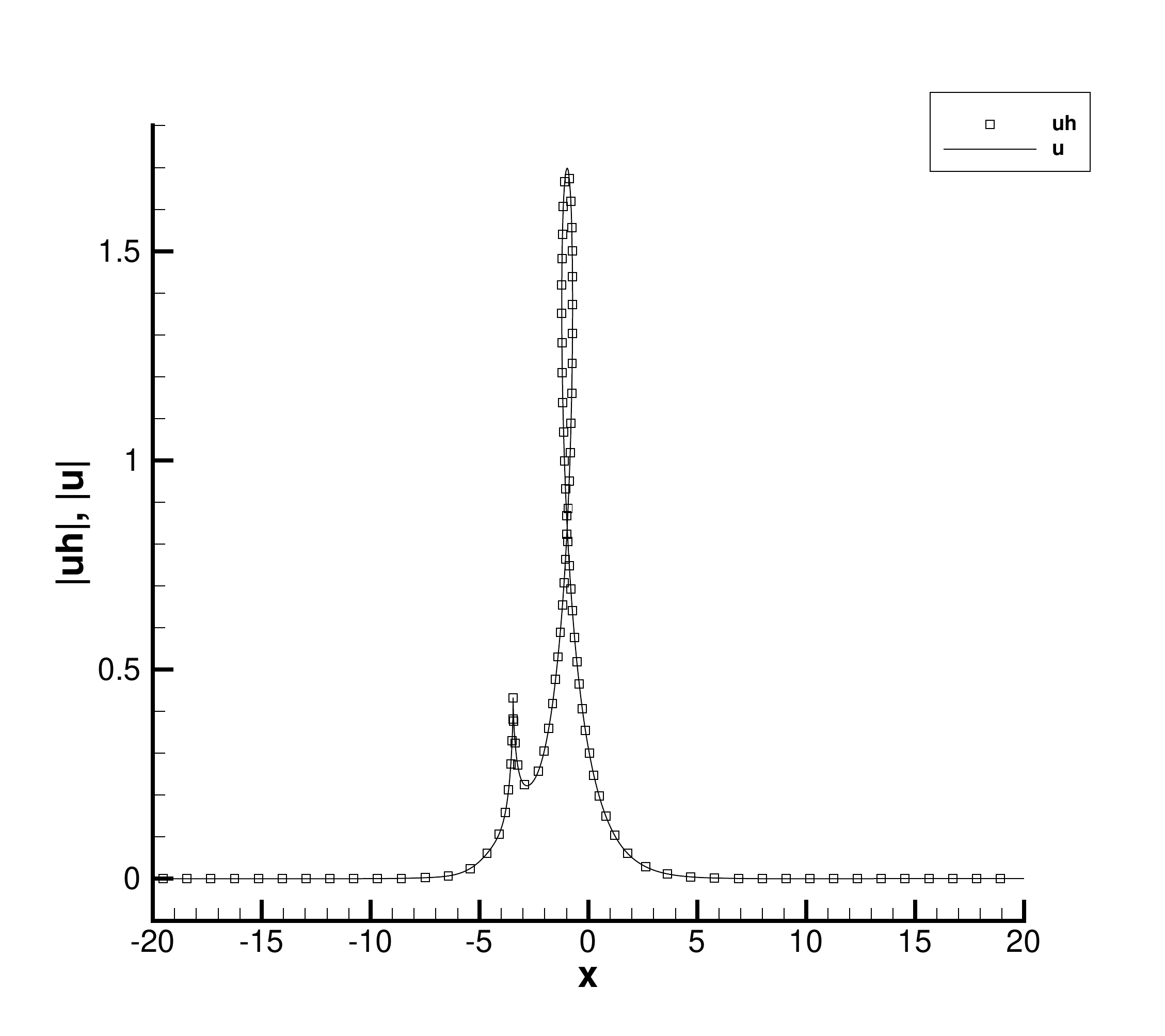}}
  \centerline{(b) t = 8.0}
\end{minipage}
\vfill
\begin{minipage}{0.49\linewidth}
  \centerline{\includegraphics[width=1.1\textwidth]{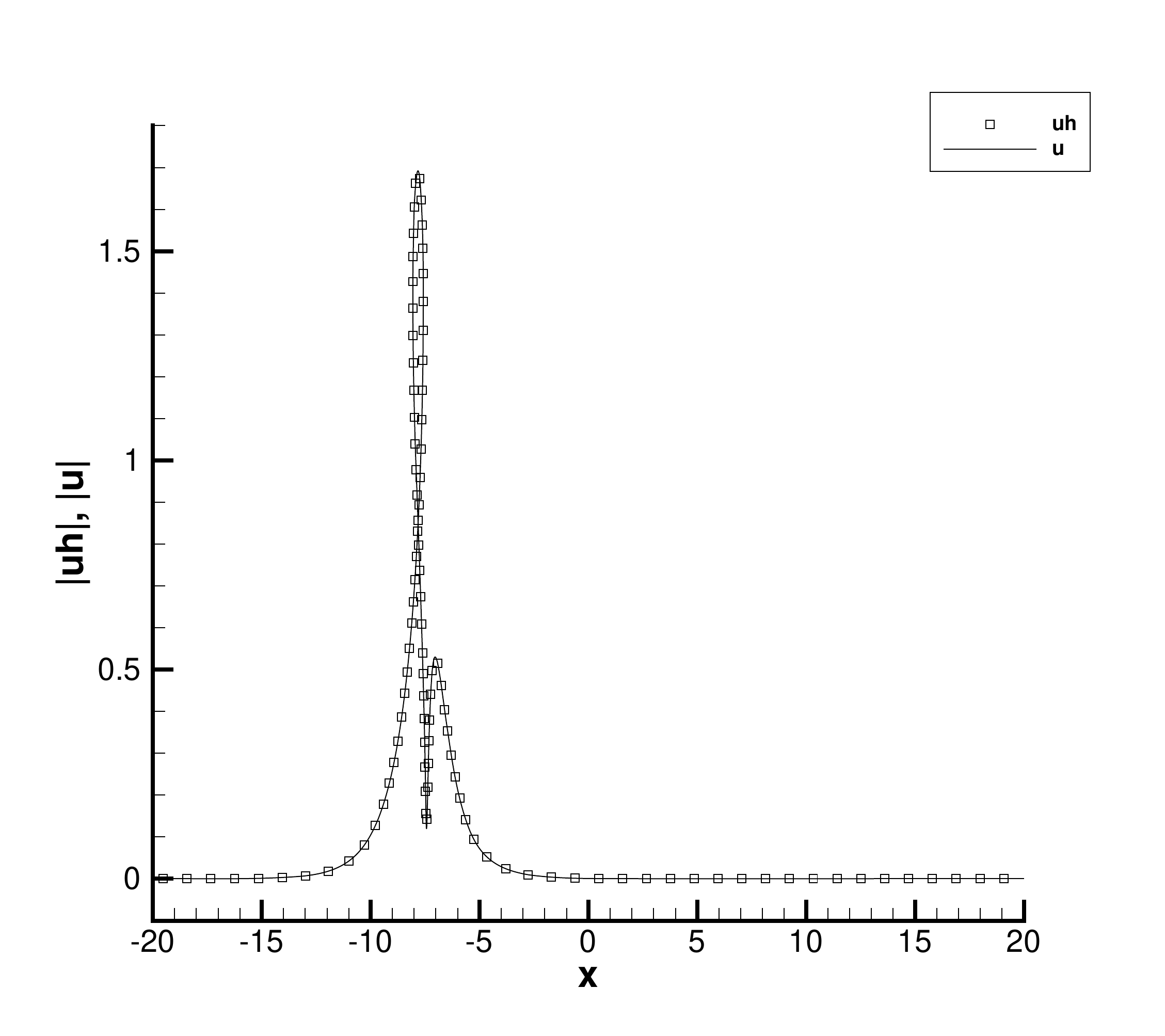}}
  \centerline{(c) t = 15.0}
\end{minipage}
\hfill
\begin{minipage}{0.49\linewidth}
  \centerline{\includegraphics[width=1.1\textwidth]{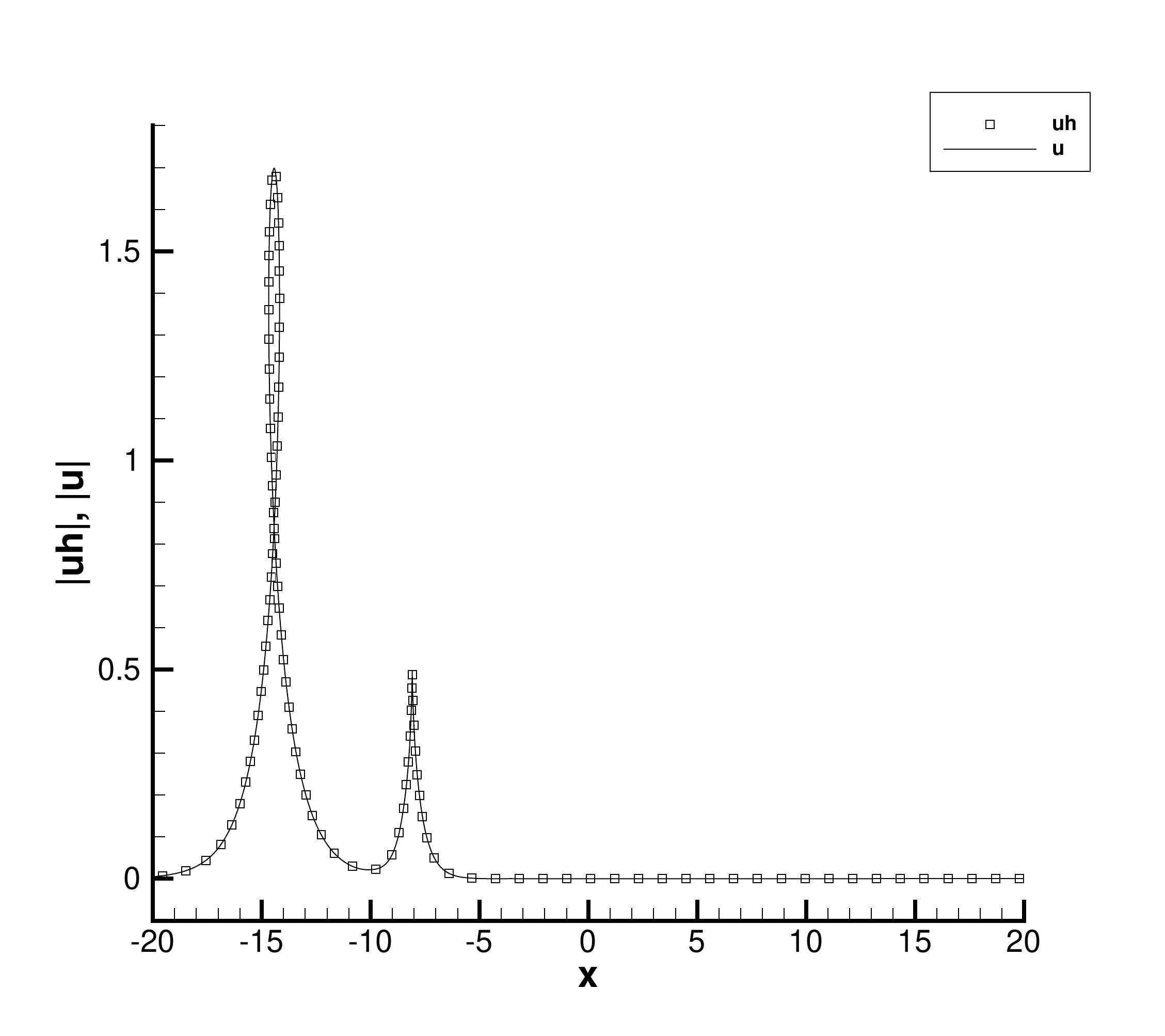}}
  \centerline{(d) t = 22.0}
\end{minipage}
\caption{ Loop-cuspon-soliton solution for the complex short pulse equation \eqref{eqn: complex_SPE1_ex}: $H_1$ conserved DG scheme with $N = 320$ cells, $P^2$ elements. The parameters $\alpha_1 =  e^{-6}$, $\alpha_2 = e^{4} $, $p_1 = 0.9 + 0.5i$, $p_2 = 2.0 + 2i$. }
\label{fig:2-soliton_complex_CD_abs2}
\end{figure}

Similar to the short pulse equation \eqref{eqn:short pulse}, the conserved quantities $H_0$, $H_1$ of the complex CD system \eqref{eqn:ComplexCD1} are contained in Table \ref{ex:csp_change}. Increasing the degree $k$ of piecewise polynomial space can reduce the change of conserved quantities. The fluctuations of $\Delta H_0$, $\Delta H_1$ for integration DG scheme is the most slight among these three DG methods.
\begin{table}[H]
\begin{tabular}{cc|cc|cc|cc}
\hline
                  &&\multicolumn{2}{c|}{$ H_0 $ conserved DG scheme} &   \multicolumn{2}{c|}{$ H_1 $ conserved DG scheme} & \multicolumn{2}{c}{ integration DG scheme}  \\\hline
     $P^k$        &&   $\Delta H'_0$  & $\Delta H'_1$ &  $\Delta H'_0$  & $\Delta H'_1$ &  $\Delta H'_0$  & $\Delta H'_1$\\\hline
     $P^1$        &&   2.43E-02	     &     1.67E-01    &  1.28E-02 &9.15E-08   & 2.05E-04 & 3.25e09    \\
     $P^2$        &&   3.56E-06	     &     1.63E-06    &  3.70E-06 & 3.29E-09  & 5.64E-07 &7.68E-10    \\
      $P^3$       &&   1.04E-07	     &     2.19E-07    &  6.98E-08 & 7.70E-10  & 1.28E-09 &9.93E-12    \\\hline

\end{tabular}
\caption{\label{ex:csp_change} The time evolution of conserved quantities for breather solution of the complex CD system \eqref{eqn:ComplexCD1}, with the computational domain $[-50,50]$ and $N = 320$ cells at time $ T = 10$. The parameters $\alpha_1 =  e^{-6}$, $\alpha_2 = e^{4} $, $p_1 = 0.5 + i$, $p_2 = 0.8 + 2i$. }
\end{table}

\end{example}

\begin{example}
The coupled short system
\begin{equation}\label{ex:eqnCCSP}
\left.
\begin{array}{r}
u_{xt} = u + \frac{1}{2}((\abs{u}^2 + \abs{v}^2)u_x)_x,\\
v_{xt} = v + \frac{1}{2}((\abs{v}^2 +  \abs{u}^2)v_x)_x,
\end{array}
\right\} \ u, v\in \mathbb{C}
\end{equation}
is integrable and admits $N$-soliton solution. The corresponding CD system \eqref{eqn:BCCD} with the exact solutions was developed in \cite{Feng_2015_PDNP}. Here, we perform the simulation of 2-smooth-soliton solution at time $T = 0, 25, 45, 65$, in which  the collision is elastic as shown in Figure \ref{fig:CCSP}. Compared with the results in \cite{Feng_2015_PDNP}, all DG schemes resolve the collision well. Hence, we  only show the result of $H_0$ conserved DG scheme in Figure \ref{fig:CCSP}.

\begin{figure}[!htp]
\begin{minipage}{0.48\linewidth}
  \centerline{\includegraphics[width=1.1\textwidth]{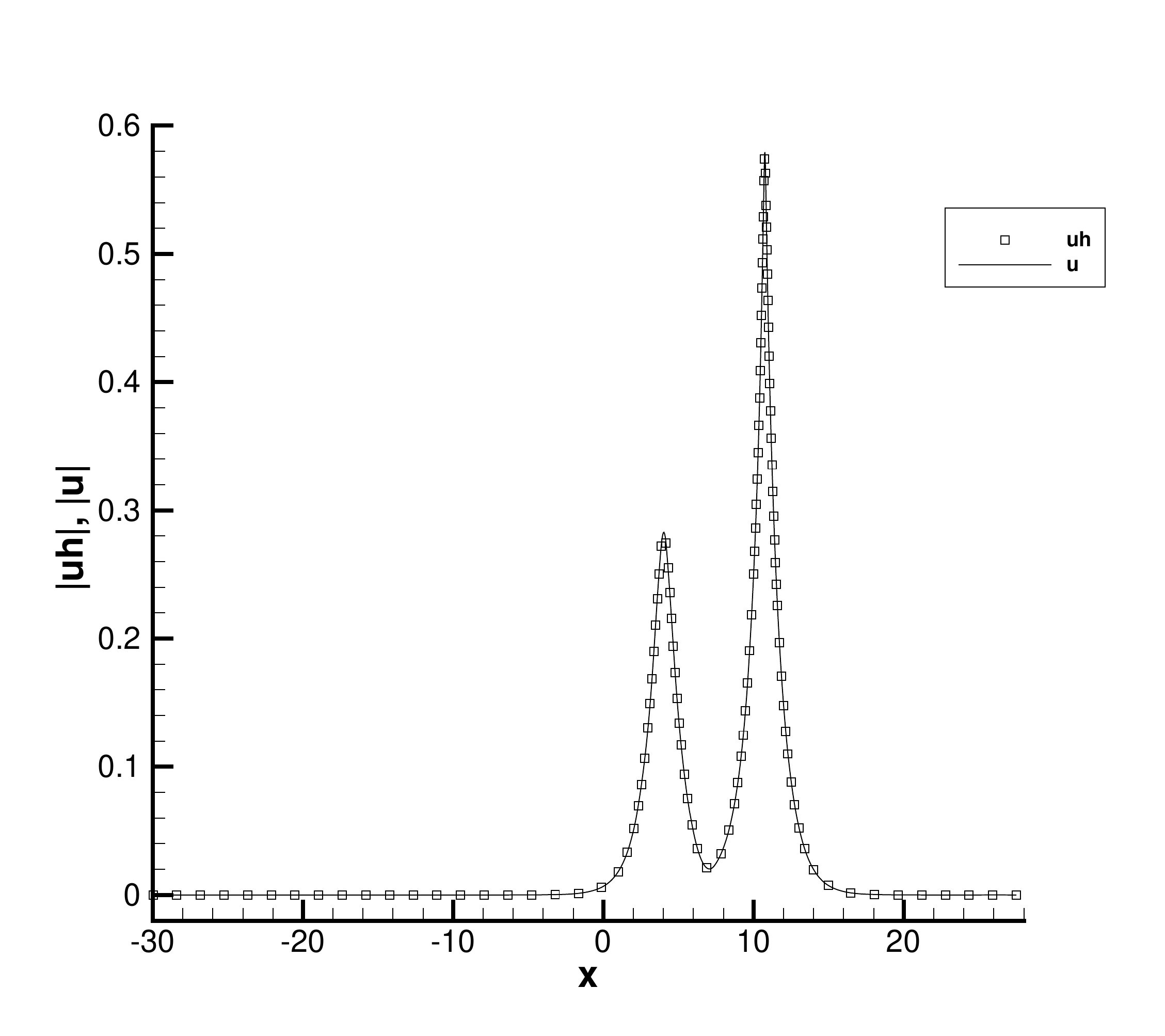}}
  \centerline{(a) t = 0.0}
\end{minipage}
\hfill
\begin{minipage}{0.48\linewidth}
  \centerline{\includegraphics[width=1.1\textwidth]{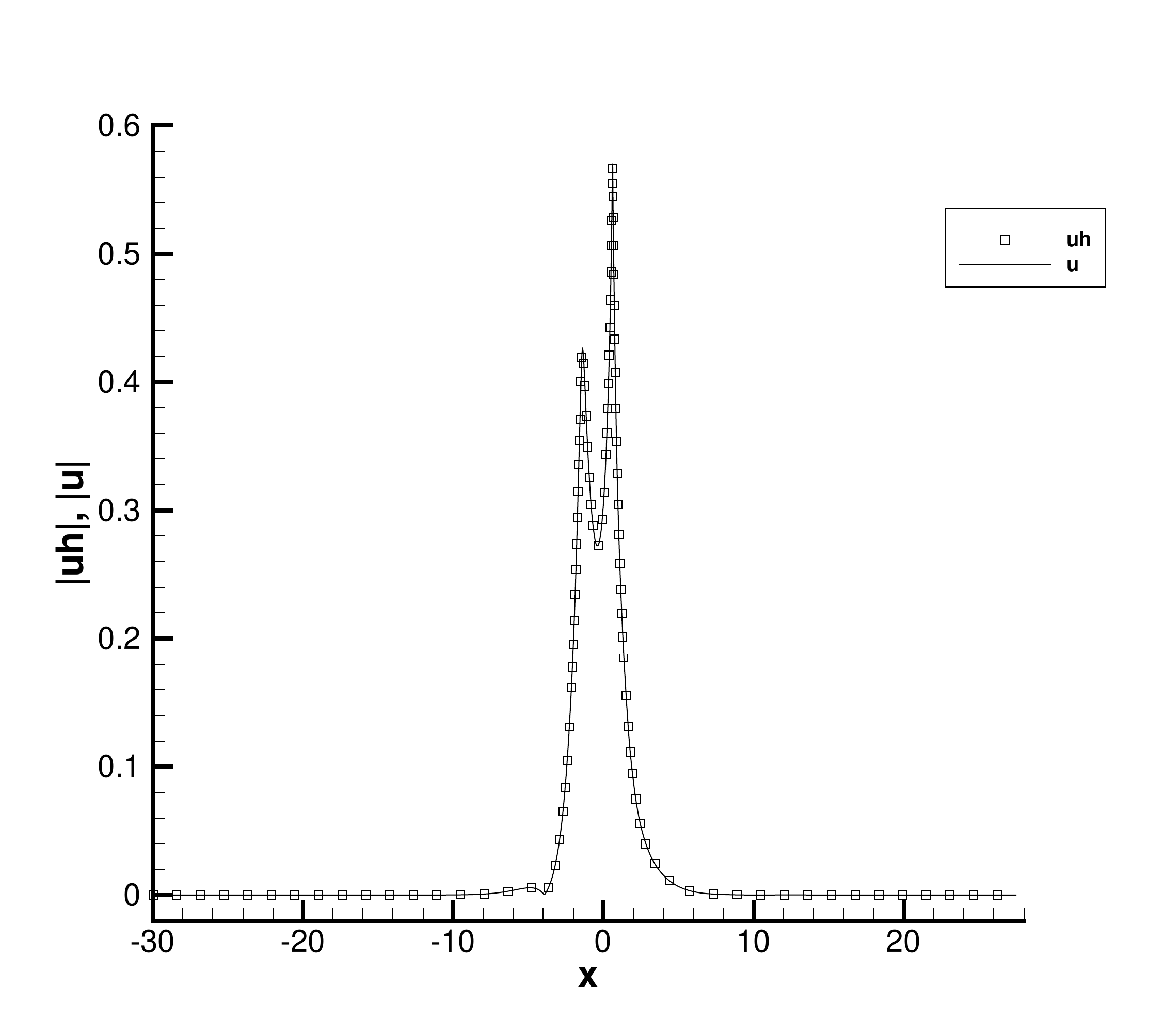}}
  \centerline{(b) t = 25.0}
\end{minipage}
\vfill
\begin{minipage}{0.48\linewidth}
  \centerline{\includegraphics[width=1.1\textwidth]{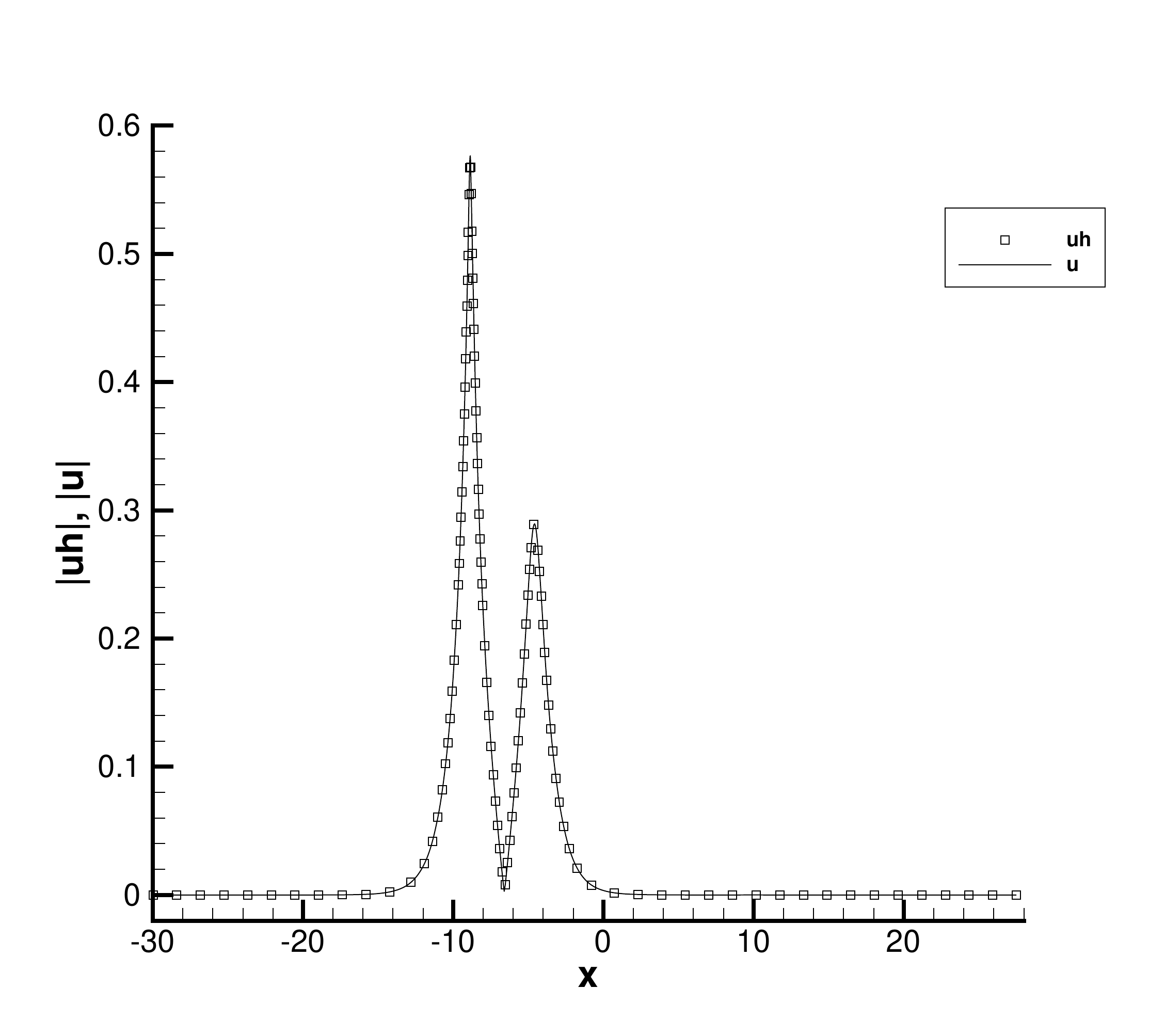}}
  \centerline{(c) t = 45.0}
\end{minipage}
\hfill
\begin{minipage}{0.48\linewidth}
  \centerline{\includegraphics[width=1.1\textwidth]{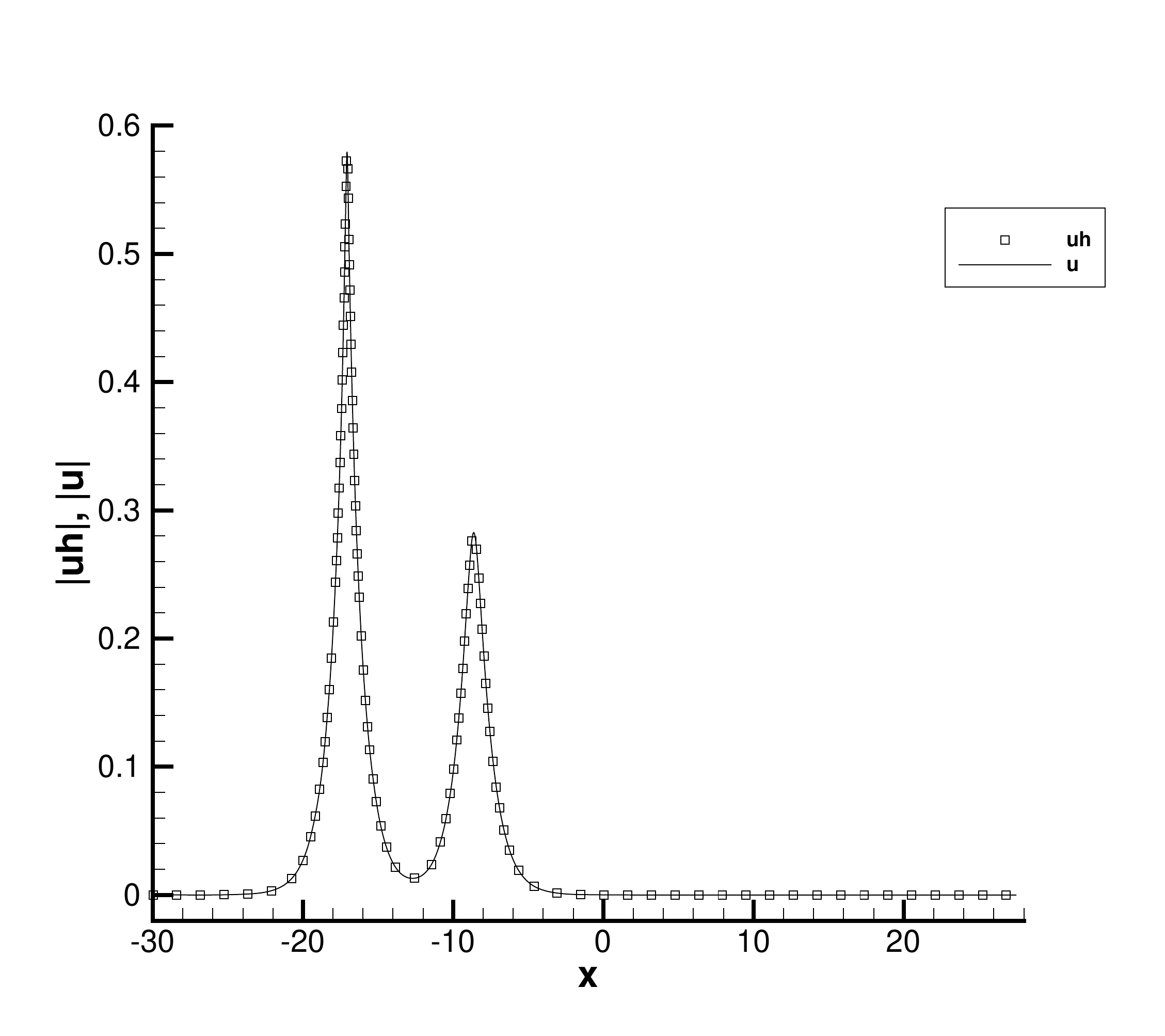}}
  \centerline{(d) t = 65.0}
\end{minipage}
\caption{\label{fig:CCSP} Two smooth-soliton solution of the coupled short pulse equation in complex form  \eqref{ex:eqnCCSP}: $H_0$ conserved DG scheme with $N = 320 $ cells, $P^2$ elements.  }
\end{figure}
\end{example}

\begin{example}
In this example, we give a 1-cuspon-soliton solution for the coupled modified short pulse equation
\begin{equation}\label{eqn:CMSP1}
u_{xt}  = u + \frac{1}{2}v(u^2)_{xx}, \ v_{xt} = v + \frac{1}{2}u(v^2)_{xx}, \ u, v\in \mathbb{R}
\end{equation}
in the form of
\begin{equation}
\begin{cases}
&u = \frac{g_1}{f}, \ v = \frac{g_2}{f}, \\
& x = y - (\ln f)_s, \ t = s.
\end{cases}
\end{equation}
The $\tau$-functions $f, g_1, g_2$ are
\begin{align}
&f = 1 + \frac{a_1b_1p_1^2}{4}e^{2\xi_1}, \ g_1 = a_1e^{\xi_1}, \ g_2 = b_1e^{\xi_1}, \\
& \xi_1 = p_1y + \frac{s}{p_1} +\xi_{10},
\end{align}
where $a_1, b_1, p_1, \xi_{10}$ are real constants. The numerical solution $u_h$ simulated by $H_1$ conserved DG scheme is shown in Figure \ref{fig:1-cuspon_MSP}.  The other two schemes, the $H_0$ conserved scheme and the integration DG scheme perform very similar to the $H_1$ conserved DG scheme.  The solution $v_h$ is also a cuspon-soliton solution whose shape is just like $u_h$.
\begin{figure}[!htp]
\begin{center}
\begin{tabular}{ccc}
\includegraphics[width=0.33\textwidth]{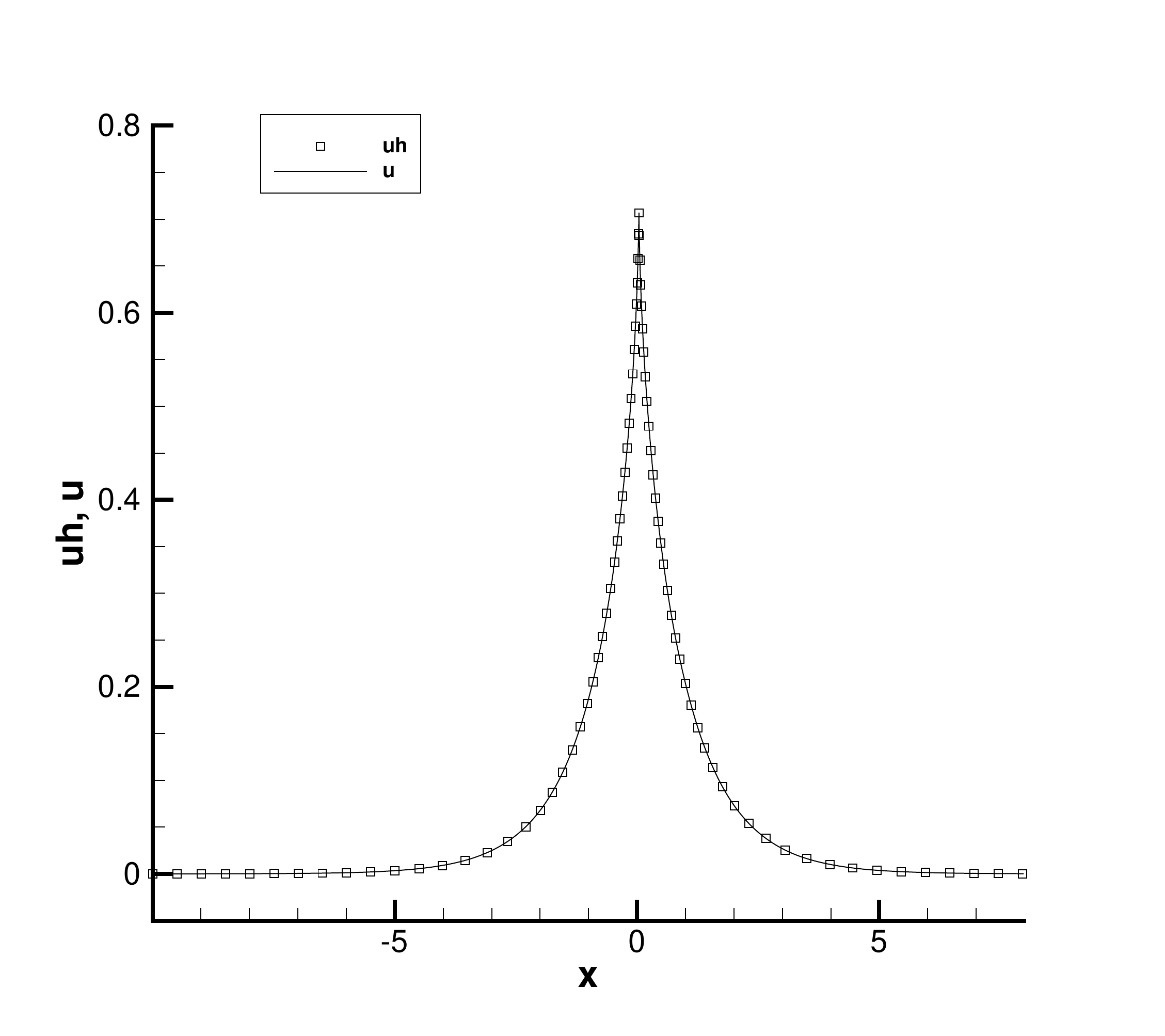}&\includegraphics[width=0.33\textwidth]{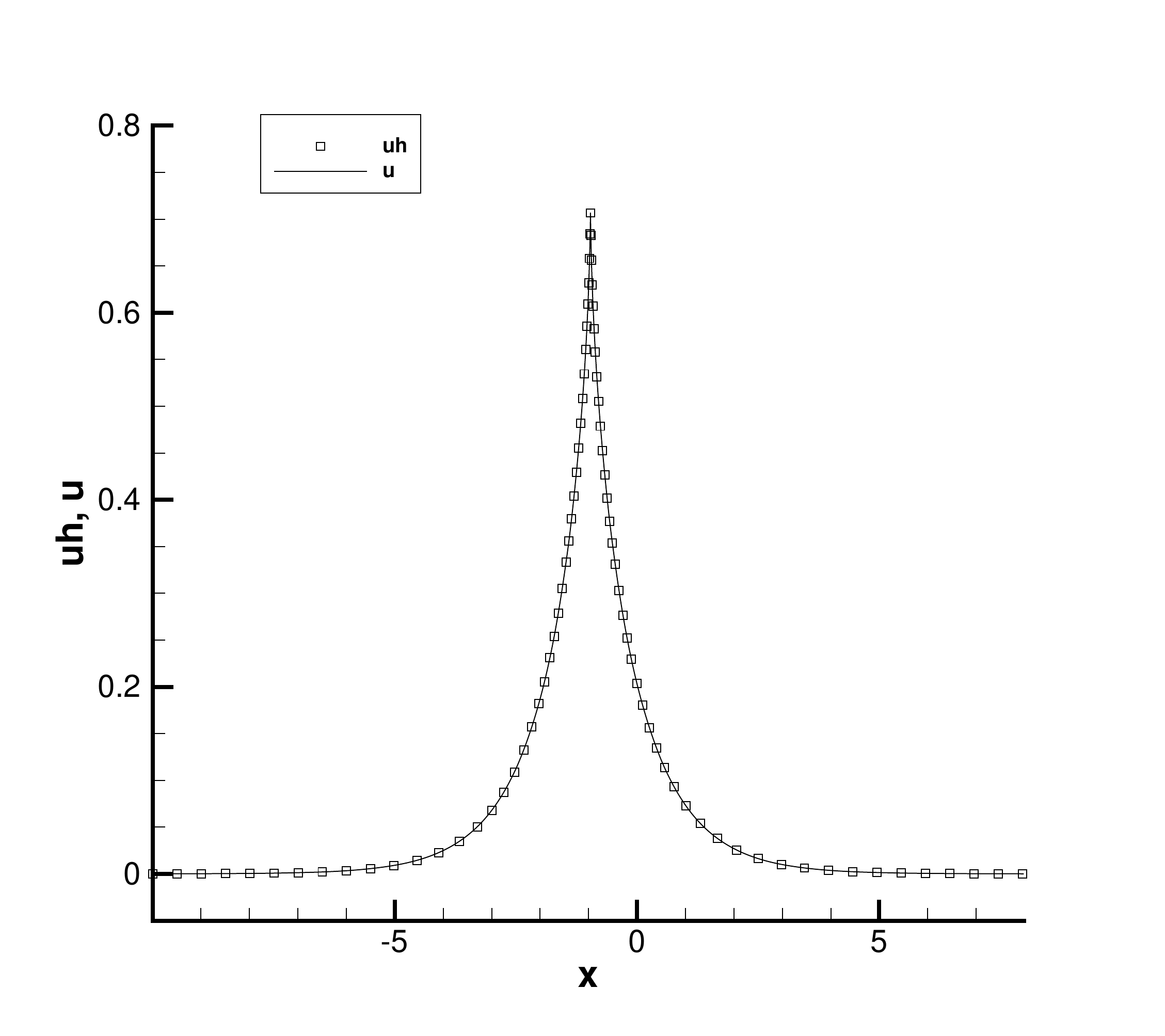}&\includegraphics[width=0.33\textwidth]{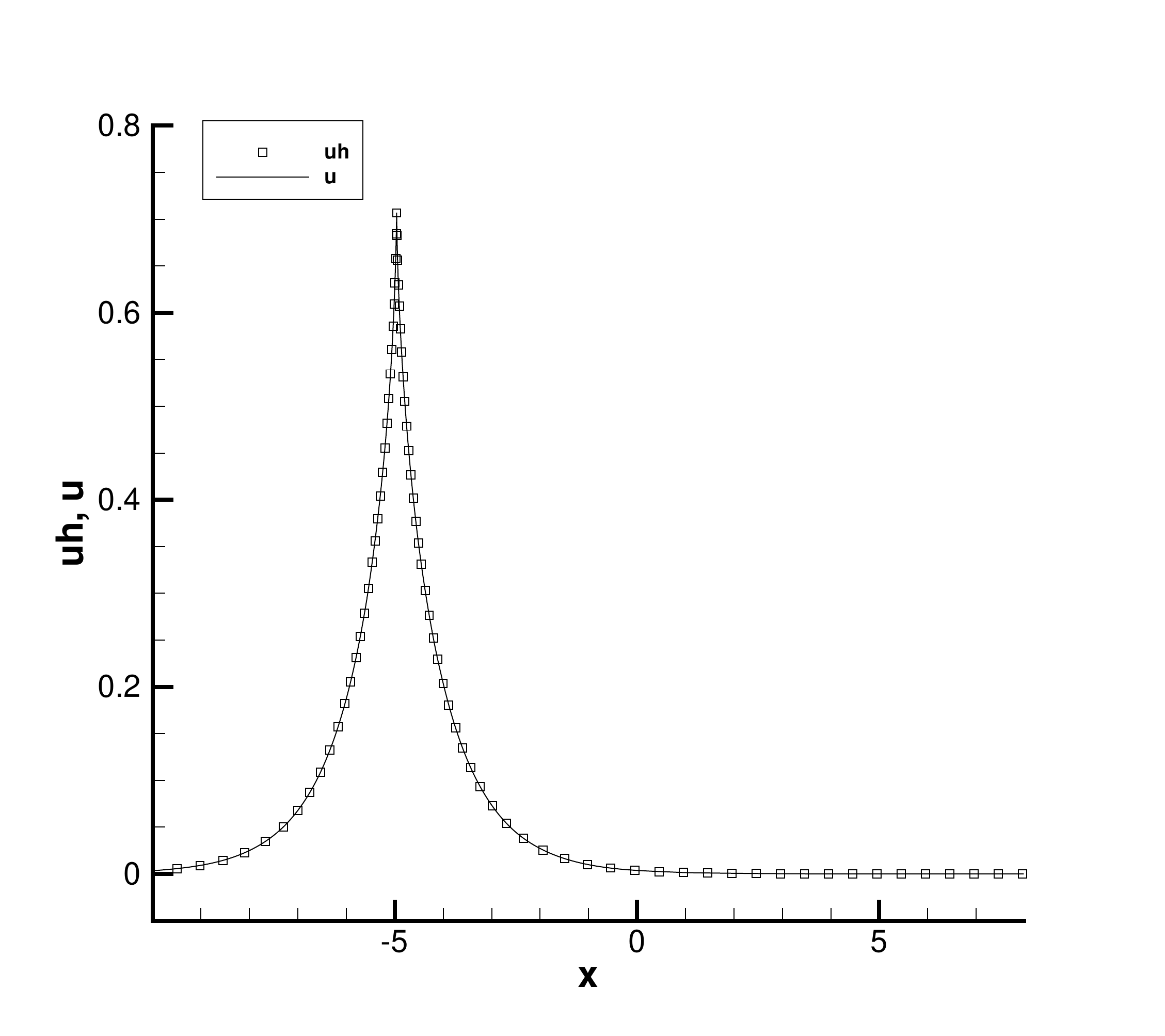}\\
(a) t = 0.0 & (b) t = 1.0  &(c) t = 5.0
\end{tabular}
\end{center}
\caption{\label{fig:1-cuspon_MSP} 1-cuspon-soliton solution $u$ of the coupled modified short pulse equation \eqref{eqn:CMSP1}: $H_1$ conserved DG scheme with $N=160$ cells, $P^2$ elements. The parameters $ p_1 = 1.0$, $a_1 = 0.5$, $b_1 = 1.0$.  }
\end{figure}
\end{example}

\begin{example} In this example, we consider the complex modified short pulse equation of defocusing type
\begin{equation}\label{eqn:MCSP_defocusing_ex}
u_{xt} = u - \frac{1}{2}u^*(u^2)_{xx}, \ u\in\mathbb{C},
\end{equation}
and its corresponding CD system \eqref{eqn:Complex_defocusing_MCD} has solution:
\begin{equation}
\begin{cases}
&u = \frac{1}{2}\frac{g}{f}e^{i(\kappa y + \gamma s)}, \\
&x = -\kappa\gamma y + \frac{1}{4}s - (\ln f)_s,
\end{cases}
\end{equation}
 where
\begin{align}
&f = 1+e^\xi, \ g = 1 + e^{\xi-2i\varphi},\\
&\xi = \beta y +  \omega s, \ \omega = -\sin\varphi, \ \beta = \frac{-\kappa \sin\varphi}{\cos\varphi - \gamma}.
\end{align}
In Figure \ref{fig:MCSP_defocusing}, we can see that $H_0$ conserved DG scheme can resolve the cuspon-soliton solution as well as the smooth-soliton solution. Similar numerical solutions can be obtained for the other two schemes, thus we omit it here for the concise.
\begin{figure}[!htp]
\begin{center}
\begin{tabular}{cc}
\includegraphics[width=0.45\textwidth]{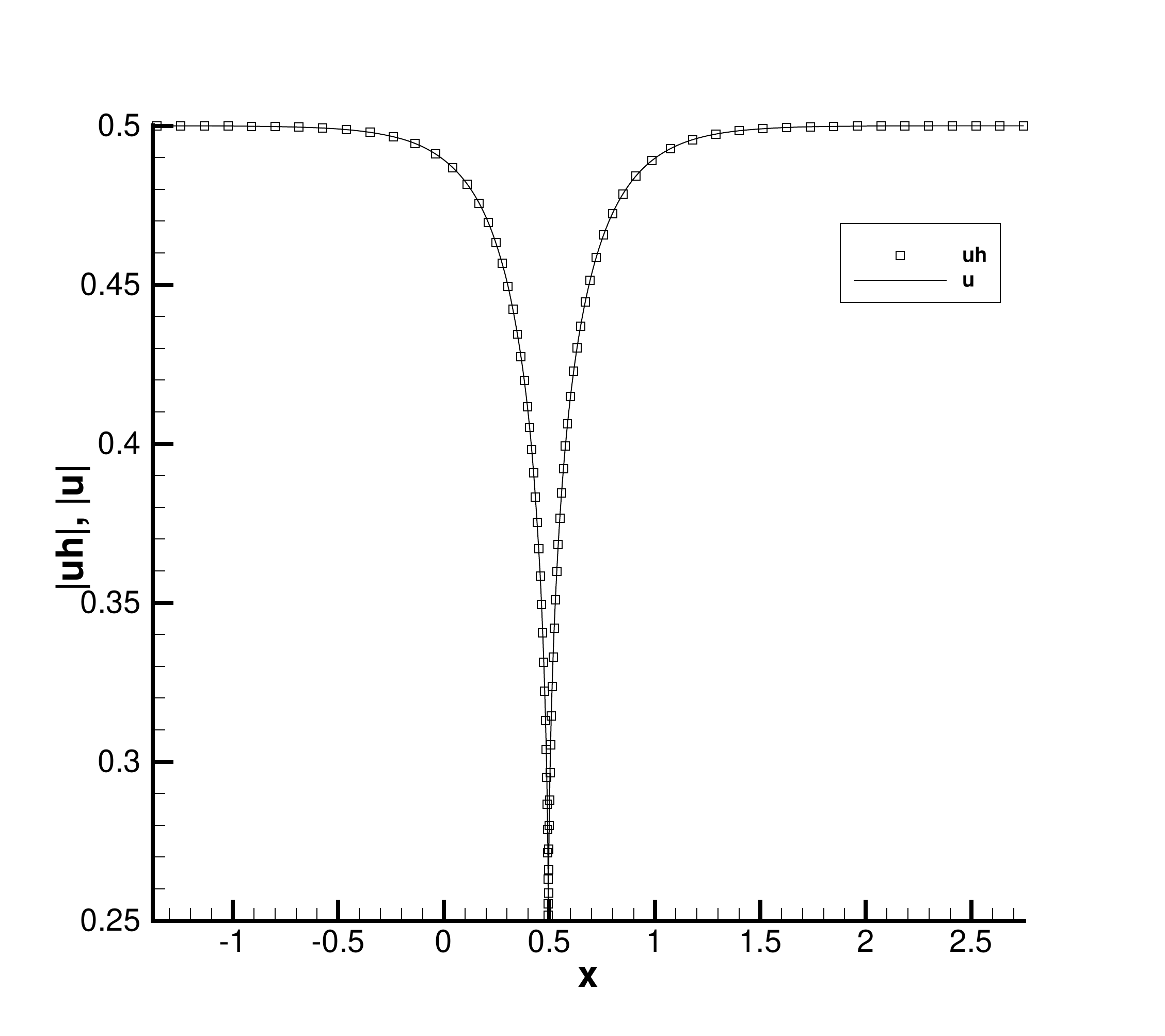}&\includegraphics[width=0.45\textwidth]{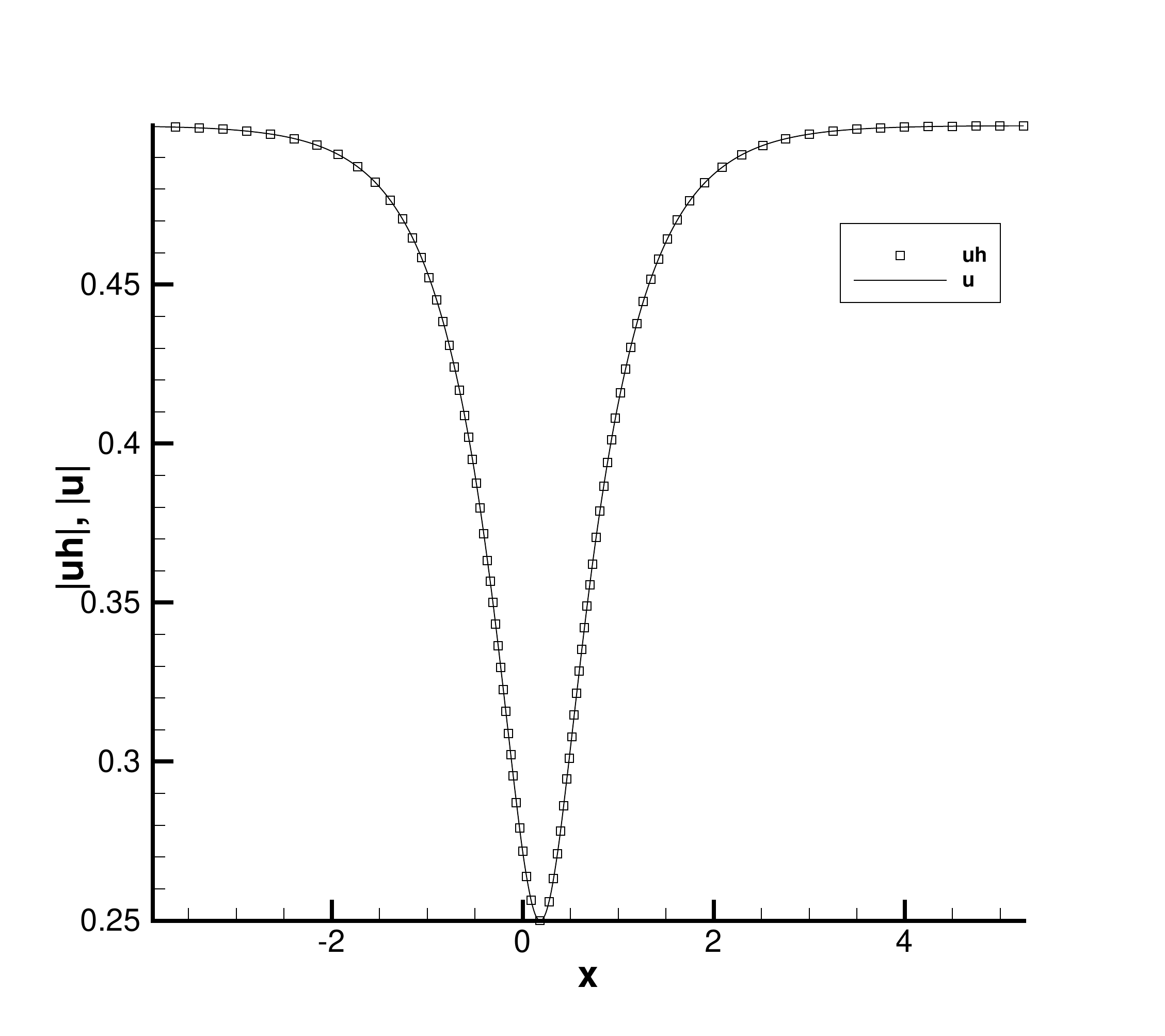}\\
(a) $\gamma = 0.5$ & (b) $ \gamma =0.25 $
\end{tabular}
\end{center}
\caption{\label{fig:MCSP_defocusing} Two types solution of the defocusing complex modified short pulse equation \eqref{eqn:MCSP_defocusing_ex}: $H_0$ conserved DG scheme with $N = 160$ cells, $P^2$ elements, at time T = 1. The parameters $ \kappa = 1.0 $, $\varphi = \frac{2}{3}\pi $. }
\end{figure}

\end{example}

\begin{example}
We consider the novel coupled short pulse system
\begin{equation}\label{eqn:New_CSP_ex}
\left.
\begin{array}{r}
u_{xt} = u + \frac{1}{6}(u^3)_{xx} + \frac{1}{2}v^2u_{xx},\\
v_{xt} = v + \frac{1}{6}(v^3)_{xx} + \frac{1}{2}u^2v_{xx}, 
\end{array}
\right\}\ u, v\in \mathbb{R}
\end{equation}
which can only be converted into the sine-Gordon system \eqref{eqn:Coupled_sine-Gordon_1}
with exact solution
\begin{equation}
 z = 2i\ln\frac{f^*}{f}, \ \tilde{z} = 2i\ln\frac{g^*}{g}.
\end{equation}
Linking with the hodograph transformation \eqref{eqn:hodograph_SG_SPE}, we obtain the solution of this novel coupled short pulse system \eqref{eqn:New_CSP_ex}
\begin{equation}
\begin{cases}
&u = \frac{1}{2}(z + \tilde{z})_s = i\Big(\ln(\frac{f^*g^*}{fg})\Big)_s, \\
&v = \frac{1}{2}(z - \tilde{z})_s = i\Big(\ln(\frac{f^*g}{fg^*})\Big)_s,\\
&x = y - 2(\ln( ff^*gg^*))_s, \ t = s.\\
\end{cases}
\end{equation}
The form of $\tau$ functions $f,g$ are listed in \cite{Feng_2012_JPMT}. We consider the solution like 3-soliton solution which is a combination of 1-loop-soliton and 2-breather solution. In Figure \ref{fig:3-soliton_uCSG}, \ref{fig:3-soliton_vCSG}, the process of interaction is displayed by the integration DG scheme, which resolve the collision well.
\begin{figure}[!htp]
\begin{center}
\begin{tabular}{ccc}
\includegraphics[width=0.33\textwidth]{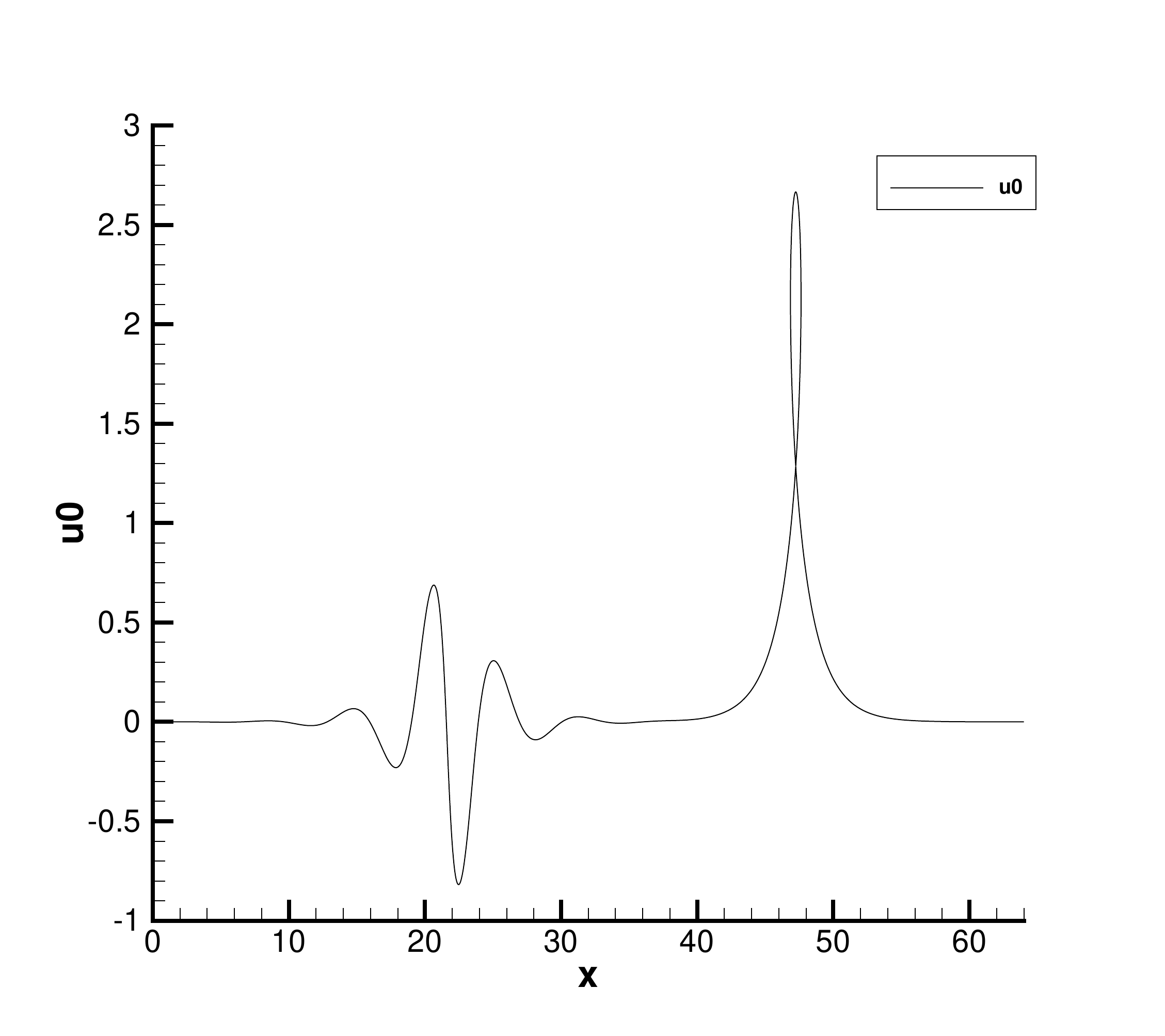}&\includegraphics[width=0.33\textwidth]{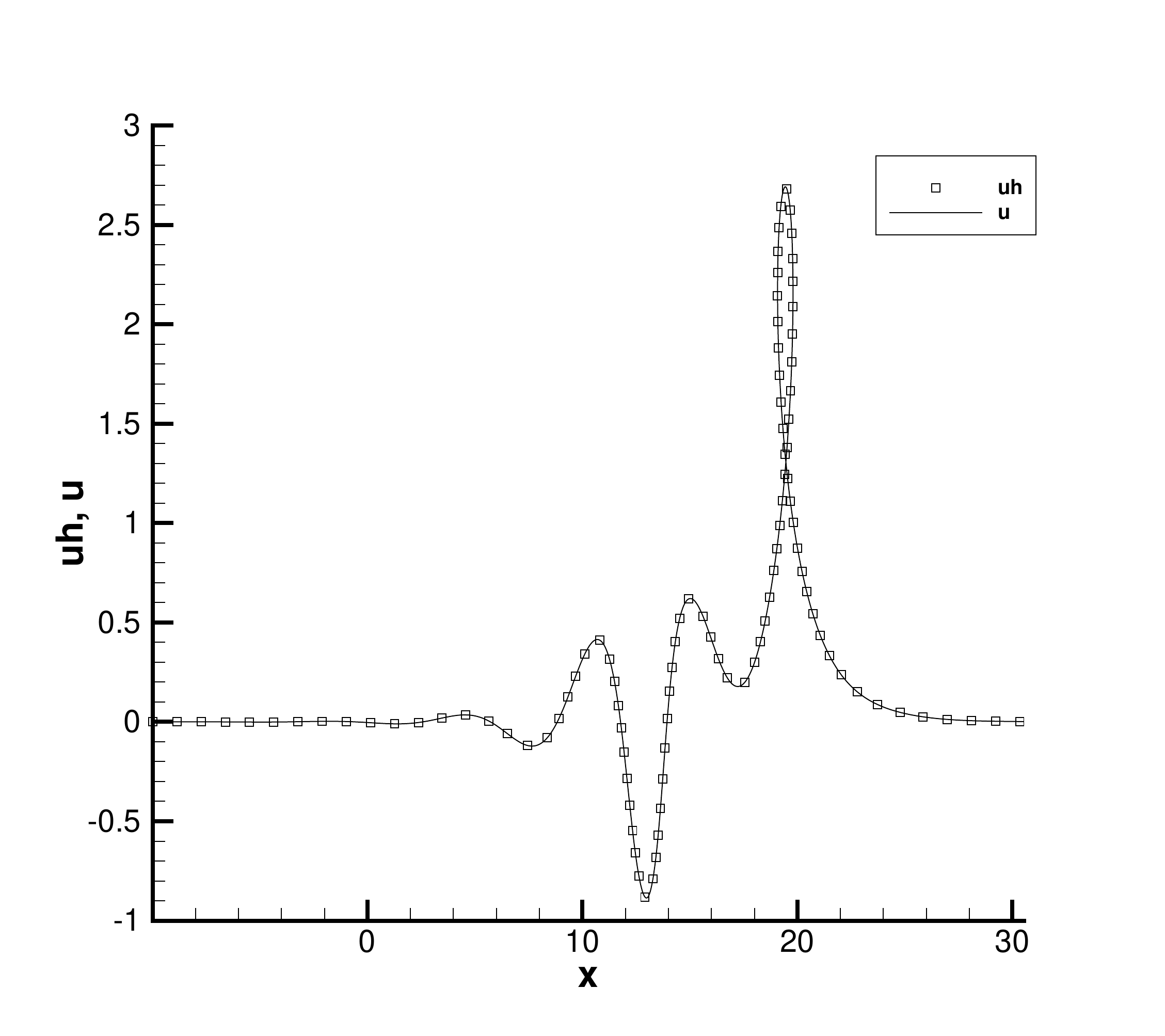}&\includegraphics[width=0.33\textwidth]{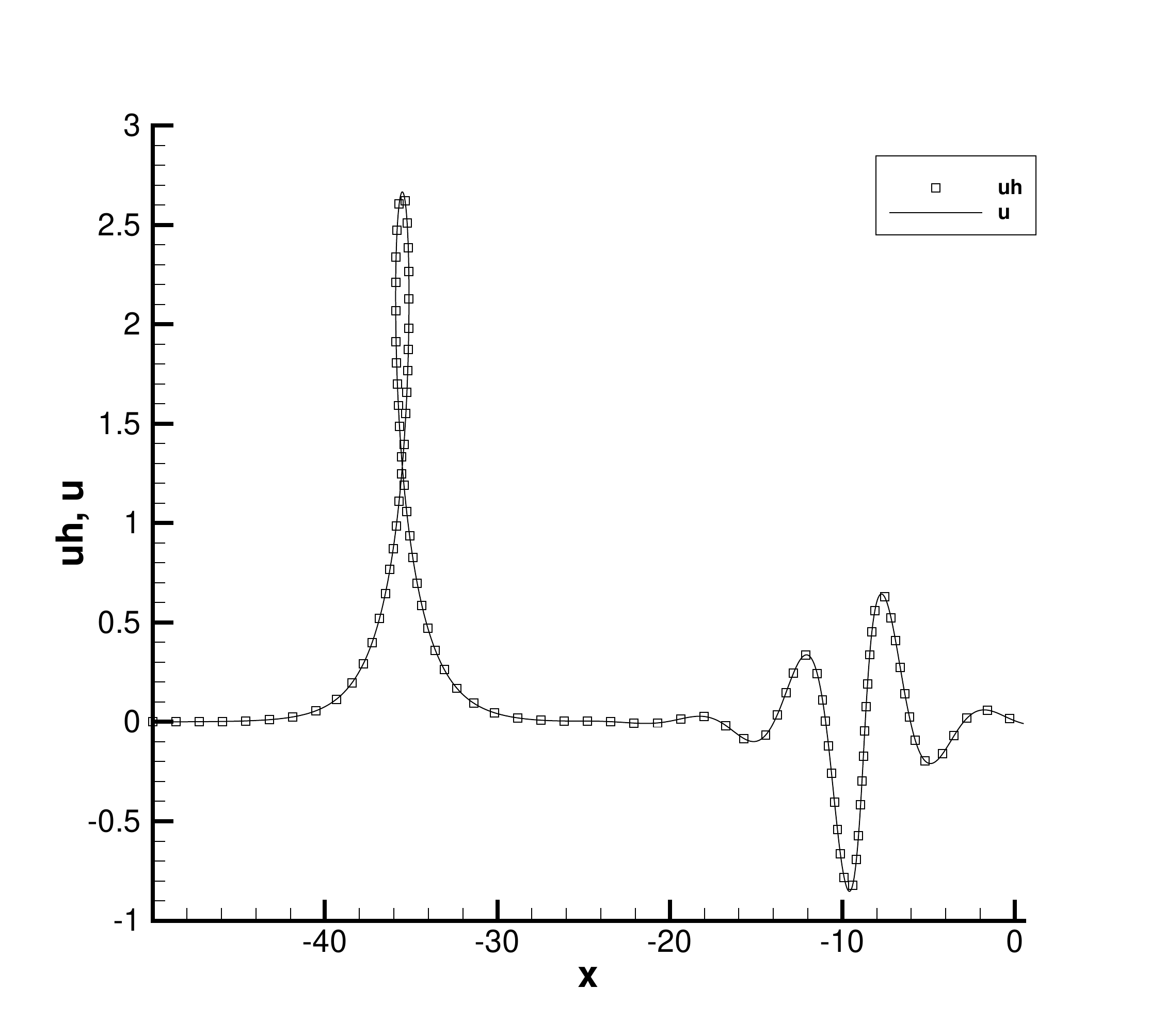}\\
(a) t = 0.0 & (b) t = 10.0  &(c) t = 30.0
\end{tabular}
\end{center}
\caption{\label{fig:3-soliton_uCSG} Loop-breather-soliton interaction $u$ of the novel coupled short pulse system \eqref{eqn:New_CSP_ex}: Integration DG scheme with $N = 160$ cells, $P^2$ elements.} 
\end{figure}
\begin{figure}[!htp]
\begin{center}
\begin{tabular}{ccc}
\includegraphics[width=0.33\textwidth]{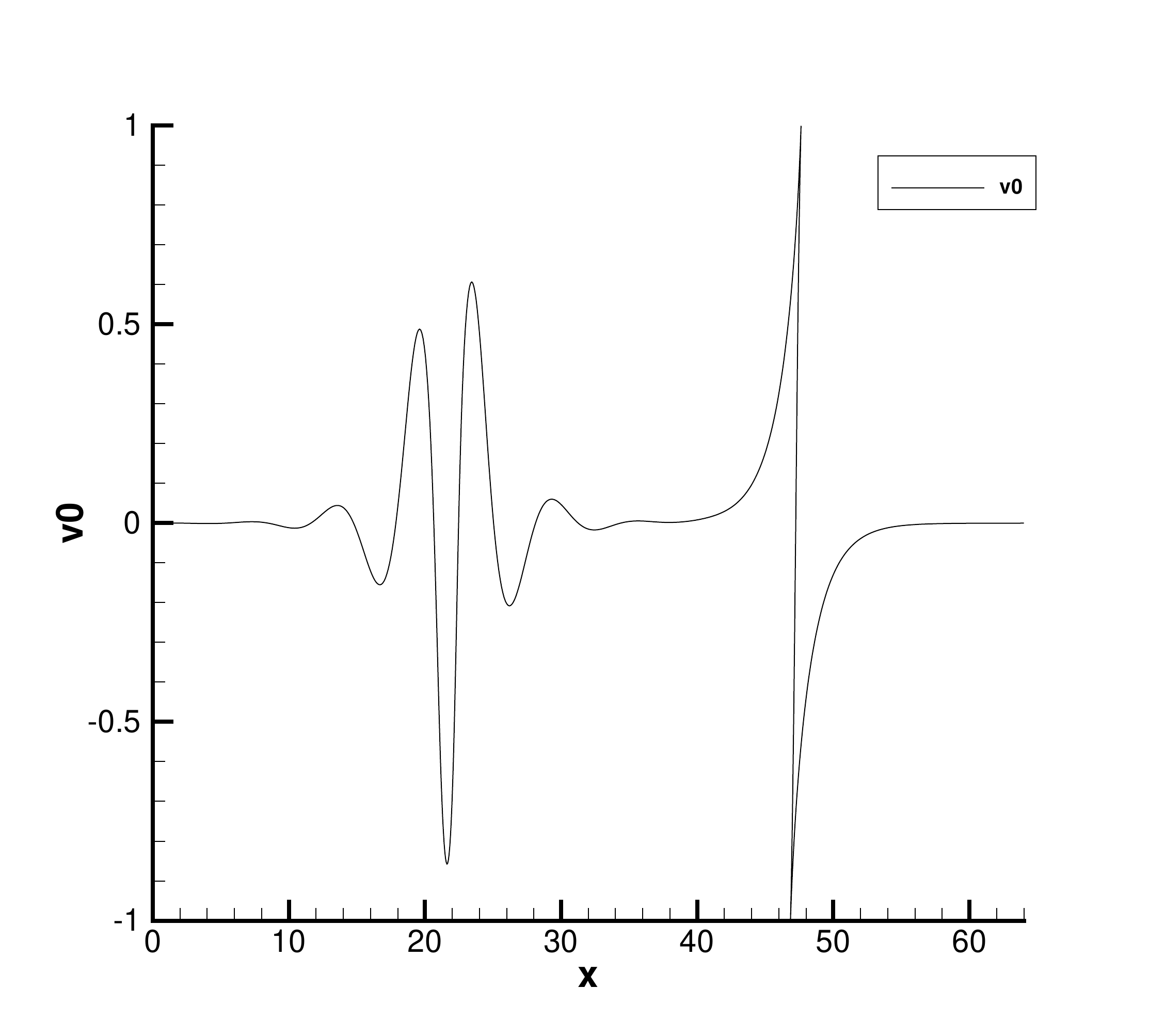}&\includegraphics[width=0.33\textwidth]{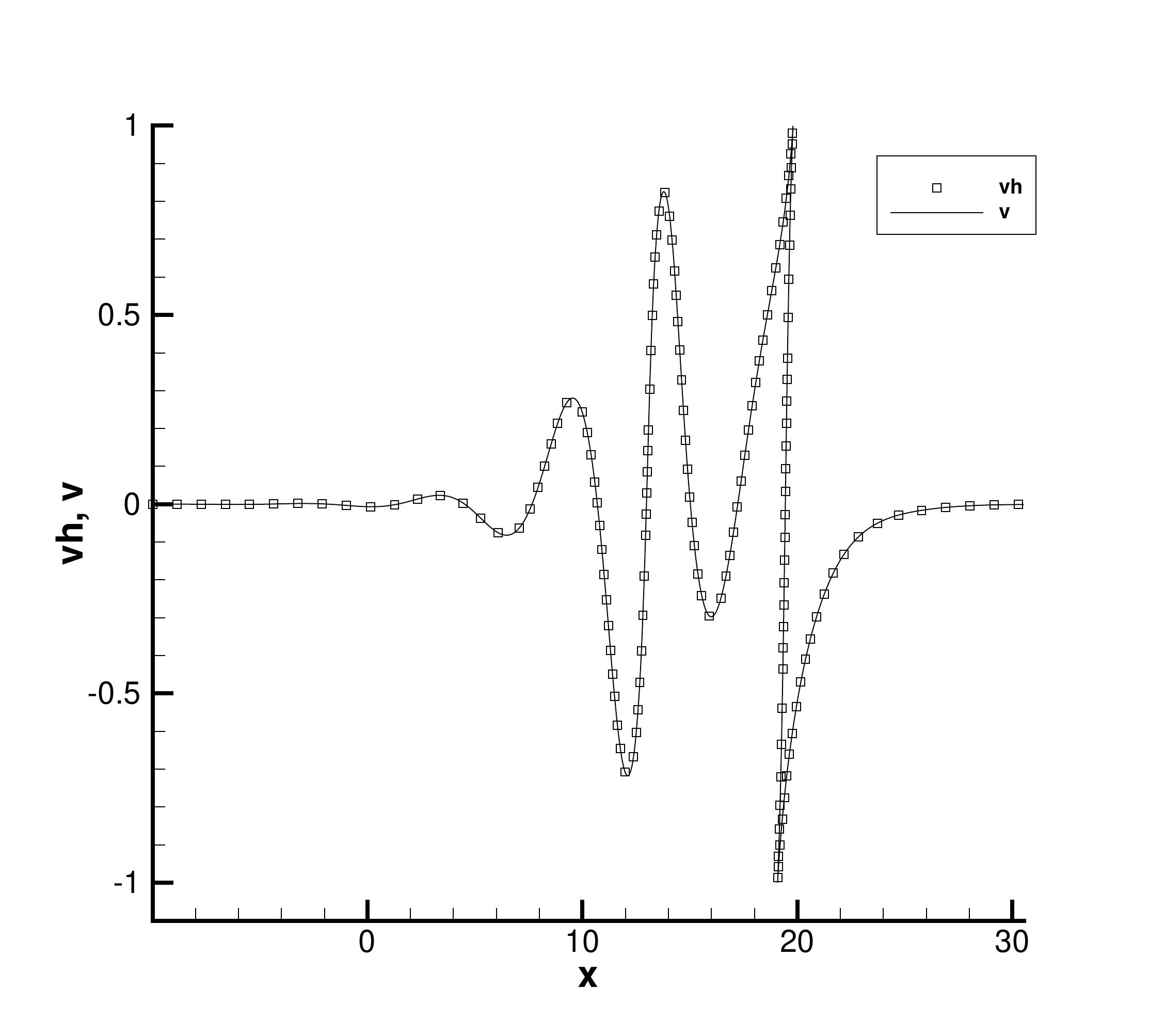}&\includegraphics[width=0.33\textwidth]{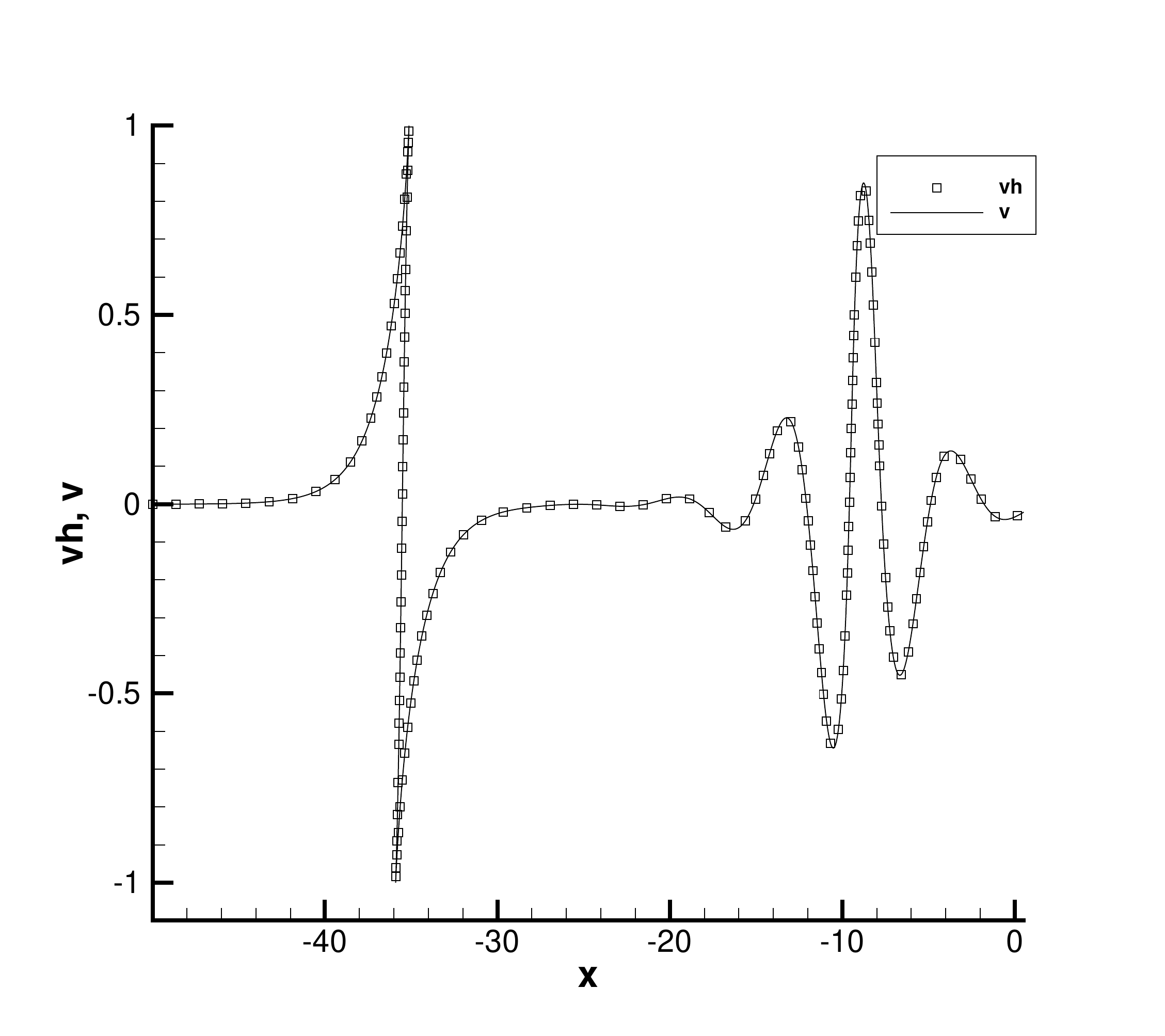}\\
(a) t = 0.0 & (b) t = 10.0  &(c) t = 30.0
\end{tabular}
\end{center}
\caption{\label{fig:3-soliton_vCSG}  Breather-cuspon-soliton interaction $v$ of the novel coupled short pulse system \eqref{eqn:New_CSP_ex}: Integration DG scheme with $N = 160$ cells, $P^2$ elements.} 
\end{figure}
\end{example}

\section{Conclusion}\label{conclusion}
In this paper, we developed  DG methods for short pulse type equations.  First, we proposed the $E_0$ conserved DG scheme for the short pulse equation directly for solving smooth solutions. Then for nonclassical solutions such as  loop-soliton, cuspon-soliton solutions, we introduced the DG schemes based on
the hodograph transformations, which link  the SP equation with the CD system or the sine-Gordon equation. For the CD system, we constructed the $H_0$ and $H_1$ conserved DG schemes, in addition to solving singular solutions, which can preserve the corresponding conserved quantities. Also an integration DG scheme was proposed.  Theoretically, we proved the a priori error estimates for  the $H_1$ conserved scheme and the integration DG scheme. More precisely, the optimal order  of accuracy in $L^2$  norm can be proved for the $H_1$ conserved DG scheme. For the integration DG scheme, the optimal  order of accuracy in $L^2$ norm for two variables can be obtained,  but suboptimal  order of accuracy in $L^\infty$  norm for the other one. Numerically, both the $H_1$ conserved scheme and the integration DG scheme can achieve the optimal convergence rates in our numerical tests, however, the order of accuracy for the $H_0$ conserved DG scheme is optimal for the even order piecewise polynomial space and suboptimal in the odd order case. For the sine-Gordon equation, we also proposed two effective DG schemes to solve it in case there is no transformation linking the SP equation with the CD system. These DG schemes can be adopted to the relevant generalized short pulse type equations, which have been shown in our numerical tests. Finally, several numerical examples in different circumstances were shown to illustrate the accuracy and capability of these DG schemes.


\begin{thebibliography}{99}
\bibitem{Abramowitz_1972_Dover} M. Abramowitz and I.A. Stegun. Handbook of Mathematical Functions. Dover Publications, 1972.6.
\bibitem{Bassi1997_JCP} F. Bassi and S. Rebay. A high-order accurate discontinuous finite element method for the numerical solution of the compressible Navier-Stokes equations. Journal of Computational Physics, 1997, 131(2): 267-279. 
\bibitem{Brunelli2005_JMP} J.C. Brunelli. The short pulse hierarchy. Journal of Mathematical Physics, 2005, 46(12): 123507.
\bibitem{Brunelli2006_PLA} J.C. Brunelli. The bi-Hamiltonian structure of the short pulse equation. Physics Letters A, 2006, 353(6): 475-478.
\bibitem{2007_brenner_SSBM} S. Brenner and R. Scott. The mathematical theory of finite element methods. Springer Science and Business Media, 2007.
\bibitem{1975_Ciarlet_NH} P.G. Ciarlet. The Finite Element Method for Elliptic Problems. North Holland, 1975 .
\bibitem{Shu1998_Siam} B. Cockburn and C.-W. Shu. The local discontinuous Galerkin method for time-dependent convection-diffusion systems. SIAM Journal on Numerical Analysis, 1998, 35(6): 2440-2463.
\bibitem{Chung_2005_N} Y. Chung, C. Jones and T. Sch\"{a}fer, et al. Ultra-short pulses in linear and nonlinear media. Nonlinearity, 2005, 18(3): 1351.
\bibitem{Bona2013_MC} J. Bona, H. Chen, O. Karakashian and Y. Xing. Conservative, discontinuous Galerkin methods for the generalized Korteweg-de Vries equation. Mathematics of Computation, 2013, 82(283): 1401-1432. 
\bibitem{Dimakis_2010_IGMA} A. Dimakis and F. M\"{u}ller-Hoissen. Bidifferential calculus approach to AKNS hierarchies and their solutions. Symmetry, Integrability and Geometry: Methods and Applications, 2010, 6(0): 55-27.
\bibitem{Feng_2012_JPMT} B.F. Feng. An integrable coupled short pulse equation. Journal of Physics A: Mathematical and Theoretical, 2012, 45(8): 085202.
\bibitem{Feng_2015_JPMT} B.F. Feng, J. Chen and Y. Chen, et al. Integrable discretizations and self-adaptive moving mesh method for a coupled short pulse equation. Journal of Physics A: Mathematical and Theoretical, 2015, 48(38): 385202.
\bibitem{Feng_2015_PDNP} B.F. Feng. Complex short pulse and coupled complex short pulse equations. Physica D: Nonlinear Phenomena, 2015, 297: 62-75.
\bibitem{Feng_2014_PJMI} B.F. Feng, K. Maruno and Y. Ohta. Self-adaptive moving mesh schemes for short pulse type equations and their Lax pairs. Pacific Journal of Mathematics for Industry, 2014, 6(1): 8.
\bibitem{Gottlieb_2001_SIAM} S. Gottlieb, C.-W. Shu and E. Tadmor. Strong stability-preserving high-order time discretization methods. SIAM review, 2001, 43(1): 89-112.
\bibitem{Kakuhata_1995_SCAN} H. Kakuhata and K. Konno. Lagrangian, Hamiltonian and conserved quantities for coupled integrable, dispersionless equations. SCAN-9510155, 1995.
\bibitem{Kakuhata_1996_JPSJ} H. Kakuhata and K. Konno. A generalization of coupled integrable, dispersionless system. Journal of the Physical Society of Japan, 1996, 65(2): 340-341.
    \bibitem{Xing2016_CiCP} O. Karakashian and Y.L. Xing. A posteriori error estimates for conservative local discontinuous Galerkin methods for the generalized Korteweg-de Vries equation. Communications in Computational Physics, 2016, 20(01): 250-278.
\bibitem{Levy2004_JCP} D. Levy, C.-W. Shu and J. Yan. Local discontinuous Galerkin methods for nonlinear dispersive equations. Journal of Computational Physics, 2004, 196(2): 751-772.
\bibitem{Liu2016_JCP} H. Liu and N. Yi. A Hamiltonian preserving discontinuous Galerkin method for the generalized Korteweg-de Vries equation. Journal of Computational Physics, 2016, 321: 776-796.
 \bibitem{Matsuno_2008_JMP} Y. Matsuno. Periodic solutions of the short pulse model equation. Journal of Mathematical Physics, 2008, 49(7): 073508.
 \bibitem{Matsuno_2011_JMP}  Y. Matsuno. A novel multi-component generalization of the short pulse equation and its multisoliton solutions. Journal of Mathematical Physics, 2011, 52(12): 123702.
 \bibitem{Matsuno_2016_JMP}  Y. Matsuno. Integrable multi-component generalization of a modified short pulse equation. Journal of Mathematical Physics, 2016, 57(11): 111507.
 \bibitem{Reed1973} W.H. Reed and T.R. Hill. Triangular mesh methods for the neutron transport equation. Los Alamos Report LA-UR-73-479, 1973.
\bibitem{Schafer_2004_PNP} T. Sch\"{a}fer and C.E. Wayne. Propagation of ultra-short optical pulses in cubic nonlinear media. Physica D: Nonlinear Phenomena, 2004, 196(1):90-105.
\bibitem{Sakovich_2005_JPSJ} A. Sakovich and S. Sakovich. The short pulse equation is integrable. Journal of the Physical Society of Japan, 2005, 74(1): 239-241.
\bibitem{Sakovich_2006_JPA} A. Sakovich and S. Sakovich. Solitary wave solutions of the short pulse equation. Journal of Physics A: Mathematical and General, 2006, 39(22): L361.
\bibitem{Sergei_2016_NSNS} S. Sakovich. Transformation and integrability of a generalized short pulse equation. Communications in Nonlinear Science and Numerical Simulation, 2016, 39: 21-28.
\bibitem{Shen_2017_JNMP} S. Shen, B.F. Feng and Y. Ohta. A modified complex short pulse equation of defocusing type. Journal of Nonlinear Mathematical Physics, 2017, 24(2): 195-209.
\bibitem{Wang_2015_SJNA} H. Wang, C.-W. Shu and Q. Zhang. Stability and error estimates of local discontinuous Galerkin methods with implicit-explicit time-marching for advection-diffusion problems. SIAM Journal on Numerical Analysis, 2015, 53(1): 206-227.

\bibitem{Xia2010JCP} Y. Xia, Y. Xu and C.-W. Shu. Local discontinuous Galerkin methods for the generalized Zakharov system. Journal of Computational Physics, 2010, 229(4):1238-1259.
\bibitem{Xia2014CiCP} Y. Xia and Y. Xu. A conservative local discontinuous Galerkin method for the Schr\"{o}dinger-KdV system. Communications in Computational Physics, 2014, 15(4): 1091-1107.
\bibitem{Xu2010_CiCP} Y. Xu and C.-W. Shu. Local discontinuous Galerkin methods for high-order time-dependent partial differential equations. Communications in Computational Physics, 2010, 7(4):1-46.
\bibitem{Xu2004_JCM} Y. Xu and C.-W. Shu. Local discontinuous Galerkin methods for three classes of nonlinear wave equations. Journal of Computational Mathematics, 2004: 250-274.
\bibitem{Xu2005_JCP} Y. Xu and C.-W. Shu. Local discontinuous Galerkin methods for nonlinear Schr\"{o}dinger equations. Journal of Computational Physics, 2005, 205(1): 72-97.
\bibitem{Xu2005_PDNP} Y. Xu and C.-W. Shu. Local discontinuous Galerkin methods for two classes of two-dimensional nonlinear wave equations. Physica D: Nonlinear Phenomena, 2005, 208(1): 21-58.
\bibitem{Xu2006_CMAME} Y. Xu and C.-W. Shu. Local discontinuous Galerkin methods for the Kuramoto-Sivashinsky equations and the Ito-type coupled KdV equations. Computer Methods in Applied Mechanics and Engineering, 2006, 195(25): 3430-3447.
\bibitem{Xu2010_JCM} Y. Xu and C.-W. Shu. Dissipative numerical methods for the Hunter-Saxton equation. Journal of Computational Mathematics, 2010: 606-620.
\bibitem{Yan2002_Siam} J. Yan and C.-W. Shu. A local discontinuous Galerkin method for KdV type equations. SIAM Journal on Numerical Analysis, 2002, 40(2): 769-791.
\bibitem{Yan2002_JSC} J. Yan and C.-W. Shu. Local discontinuous Galerkin methods for partial differential equations with higher order derivatives. Journal of Scientific Computing, 2002, 17(1-4): 27-47.
\bibitem{Yu2018_JCAM} C.H. Yu, B.F. Feng and T.W.H. Sheu. Numerical solutions to a two-component Camassa-Holm equation. Journal of Computational Applied Mathematics, 2018, 336: 317-337.
 \bibitem{Zhang2019CiCP} Q. Zhang and Y. Xia. Conservative and dissipative local discontinuous Galerkin methods for Korteweg-de Vries type equations. Communications in Computional Physics, 2019, 25: 532-563.
\end{thebibliography}
\end{document}